\documentclass{amsart}

\usepackage{amssymb}
\usepackage{fullpage}

\title[More limits of the elliptic beta integral]{More basic hypergeometric limits of the elliptic hypergeometric beta integral}

\author{Fokko J. van de Bult}

\newtheorem{theorem}{Theorem}[section]
\newtheorem{lemma}[theorem]{Lemma}
\newtheorem{proposition}[theorem]{Proposition}
\newtheorem{corollary}[theorem]{Corollary}

\theoremstyle{definition}
\newtheorem{definition}[theorem]{Definition}

\newcommand{\rphis}[2]{{}_{#1\vphantom{#2}}\phi_{#2\vphantom{#1}}}

\newcommand{\rpsis}[2]{{}_{#1\vphantom{#2}}\psi_{#2\vphantom{#1}}}

\newcommand{\rWs}[2]{{}_{#1\vphantom{#2}}W_{#2\vphantom{#1}}}

\newcommand{\rphisx}[4]{\rphis{#1}{#2}\left( \begin{array}{c} #3 \end{array};q,#4\right)}

\newcommand{\rpsisx}[4]{\rpsis{#1}{#2}\left( \begin{array}{c} #3 \end{array};q,#4\right)}

\newcommand{\IR}{I\!R}

\begin{document}

\begin{abstract}
In this article we continue the work from \cite{vdBR1}. In that article Eric Rains and the present author considered the limits of the elliptic beta integral as $p\to 0$ while the parameters $t_r$ have a $p$-dependence of the form $t_r=u_rp^{\alpha_r}$ (for fixed $u_r$ and certain real numbers $\alpha_r$). In this article we again consider such limits, but now we let $p\to 0$ along a geometric sequence $p=xq^{sk}$ (for some integer $s$, while $k\to \infty$), and only allow $\alpha_r\in \frac{2}{s} \mathbb{Z}$. These choices allow us to take many more limits. In particular we now also obtain bilateral basic hypergeometric series as possible limits, such as the evaluation formula for a very well poised ${}_6\psi_6$. 
%As before, limits of the known identities for the elliptic beta integral are preserved in the limit and lead to similar equations for the limiting functions.
\end{abstract}

\maketitle

\section{Introduction}
While hypergeometric and basic hypergeometric series have existed for centuries, the generalization to elliptic hypergeometric series is relatively recent. The first elliptic hypergeometric identity was Frenkel and Turaev's \cite{FT} summation formula for a very well poised ${}_{10}V_9$. In the 15 years since this original paper the structure of the elliptic hypergeometric world has been thoroughly researched, yet much still remains unknown.

It was immediately clear that the Frenkel-Turaev summation is a direct generalization of the evaluation formula for a terminating very-well poised ${}_8W_7$, which itself generalizes the evaluation of a very-well poised ${}_7F_6$. Indeed an often used idea in trying to learn new properties of elliptic hypergeometric functions is to generalize a property of ($q$-)hypergeometric functions. One might hope that every result for basic hypergeometric functions has an analogue for elliptic hypergeometric functions and vice versa. The way in which this statement can be made precise in a concrete example is by showing that the limit of $p\to 0$ of the elliptic hypergeometric analogue becomes the basic hypergeometric original.

Due to the fact that non-terminating elliptic hypergeometric series typically do not converge, the analogues of non-terminating basic hypergeometric series identities are results for elliptic hypergeometric integrals. Thus the elliptic hypergeometric beta integral evaluation of Spiridonov \cite{Spi1} should be seen as the non-terminating version of Frenkel-Turaev's summation formula. Indeed one can obtain the terminating series evaluation as a specialization of the beta integral evaluation: In this specialization the series appears as a sum of residues. In this article we focus on the elliptic beta integrals, as the elliptic integrals are in this sense more general than the series.

Elliptic hypergeometric identities typically have a lot of parameters (apart from the obligatory $p$ and $q$). For example the elliptic beta integral evaluation has 5, that is $t_1,\ldots, t_6$ under the balancing condition $\prod_{r=1}^6 t_r=pq$. As we take the limit $p\to 0$ we have to describe what happens to these other parameters: whether they remain constant or depend on $p$ in some way. The limit obtained depends on the way these parameters change as $p \to 0$. Thus one elliptic hypergeometric identity can typically have many basic hypergeometric analogues. In this article we attempt to obtain all interesting limits of the elliptic beta integrals where the extra parameters are all of the form $t_r=u_r p^{\alpha_r}$ for some constant $u_r$.  For different values of $\alpha_r$ one would expect different limits to appear.

Consider an evaluation formula $\int I(z;p) dz = C(p)$, where the dependence of $I$ on parameters other that $z$ and $p$ is suppressed. If we rescale the formula to $\frac{1}{C(p)} \int I(z;p) dz =1$ we can certainly take the limit as $p\to 0$ and the result will be 1. This does not lead to an interesting result. Therefore we restrict the methods we are allowed to use to take the limit. In particular we are not allowed to apply some elliptic hypergeometric identity before taking the limit. The goal is to find limits of such identities and therefore we must take different limits on both sides of the equation. In essence, the only method we do allow is to interchange limit and integral, more details are in Section \ref{secgenmeth}.

By systematically determining the limits for different values of $\alpha_r$ one could perhaps obtain new basic hypergeometric identities (though in this article we do not really do that). More importantly however, the goal is to better understand the known basic hypergeometric identities. Our analysis also determines the limits between the basic hypergeometric identities. Moreover by tallying which identities are derived from the elliptic beta integral we can mark the remaining identities as equations which still need an elliptic analogue. 

In previous articles Eric Rains and the author have already considered very similar problems. In \cite{vdBR1} we considered limits of the elliptic beta integral which can be obtained by a direct limit $p\to 0$, which led to basic hypergeometric integrals and unilateral basic hypergeometric series. Subsequently we wrote a trilogy, the first part \cite{vdBR2} of which covered the limits of the elliptic hypergeometric biorthogonal functions first obtained by Spiridonov \cite{Spi2}, whose measure is the elliptic beta integral evaluation. Amongst the limits we found the entire $q$-Askey scheme \cite{KLS} (including the limits between the families of orthogonal polynomials in this scheme).
In the second part \cite{vdBR3} we covered the extension to the multivariate elliptic hypergeometric biorthogonal functions discovered by Rains \cite{Rainsabelian}. Amongst other limits we obtained both the Macdonald polynomials and the Koornwinder polynomials.
The final part \cite{vdBR4} of the trilogy dealt with the measures for the multivariate biorthogonal functions, the $BC_n$ symmetric elliptic hypergeometric Selberg integral from \cite{Rainstrafo}. 

The one thing we did not do previously is to systematically obtain all limiting bilateral series. 
In particular we would like to show that the evaluations of a very well poised ${}_6\psi_6$ is a limit of the elliptic beta integral as well. We could not obtain that limit before as we considered limits $p\to 0$ in a continuous way. This restricts us to take limits of $\Gamma(p^{\alpha}z)$ for $\alpha \in [0,1]$ (where $\Gamma$ denotes the \textit{elliptic} gamma function). If we want to extend this result for other $\alpha$ we can use the difference equation: $\Gamma(p^{\alpha}z) = \frac{1}{\theta(p^{\alpha} z;q)} \Gamma(p^{\alpha+1} z)$. This allows us to ensure that all arguments of elliptic gamma functions are of the form $\Gamma(p^{\alpha} z)$ with $\alpha \in [0,1]$. However this comes at the cost of introducing theta functions, whose $p\to 0$ limits do not exist. However, if we specialize $p=xq^v$ such $\alpha v\in \mathbb{Z}$ we obtain $\theta(p^{\alpha} z;q) = \theta( x^{\alpha} q^{\alpha v} z;q) = (-\frac{1}{zx^{\alpha}})^{\alpha v} q^{-\binom{\alpha v}{2}} \theta(x^{\alpha} z;q)$. That is, under this specialization we can take the limit of the theta functions, as long as we allow a proper prefactor.

In this article we will restrict $p$ to a geometric sequence, ensuring that we can obtain those extra limits. Note that for every choice of rational $\alpha_r$'s there exist geometric sequences which ensure that we can take the limit as $p\to 0$ along that geometric sequence. 
For different values of $x$ we take the limit along a different geometric sequence, and thus the limit we obtain is typically dependent on $x$. But of course, multiplying $x$ by the proper power of $q$ (so that the geometric sequence $p=xq^v$ shifts by one step) will not change the limit, so the limits will be elliptic functions of $x$. 

The methods, and even a lot of the calculations, in this article can also be extended to the multivariate level. The reason we consider only the univariate integrals in this article is that the complications we want to focus on already fully appear on the univariate level. The most major extra complication that arises for multivariate analogues is that limits which are a sums of multiple series (for example the evaluation formula for a sum of two ${}_8W_7$'s) become very complicated expressions. An example of such a complicated expression is given by the  measure for multivariate big $q$-Jacobi polynomials in \cite{Stok}. To understand these expressions would detract too much from the issues we want to stress in this article.

The notational conventions in this article are all quite standard, but for convenience summed up in the next section. Subsequently we introduce the elliptic beta integral of which we want to study the limits. Section \ref{secgenmeth} discusses in more detail what limits we actually want to consider and can be seen as an extension of this introduction. Section \ref{secasym} discusses the asymptotic behavior of the elliptic gamma function and $\theta(x;q)$ as $p\to 0$. The next two section discuss integral and series limits respectively, and give the proofs of the limits; though they do not identify which are the interesting limits.
Section \ref{secsymres} discuss some symmetries of the residues of the integrand of the elliptic beta integral, which reduce the number of different expressions we obtain as limits.
The calculations to determine which are interesting limits are quite non-trivial so we first need an intermediate section on the Weyl group of type $E_6$ to most efficiently make use of the symmetries of the situation. Then we have two section which do the essentially combinatorial calculations necessary to determine which limits are interesting, and a section which consolidates the results of these calculations. In the final section we present a long list of all interesting limits we can obtain in the way described.

\section{Pochhammer symbols, theta functions and elliptic gamma functions}
In this section we introduce notation for the elliptic gamma function and related functions, and describe their most basic properties.

For essentially all parameters, with the notable exception of $p$, we will make a choice of logarithm. This allows us to write things like $x^{\alpha}= \exp( \alpha \log(x))$ for non-integer $\alpha$ without further explanation. Moreover throughout the article we assume $|q|,|p|<1$, which will ensure the infinite products defining the $q$-Pochhammer symbols, theta functions and elliptic gamma functions defined below converge.

The $q$-Pochhammer symbol is given by 
\[
(x;q)_n := \prod_{r=0}^{n-1} (1-xq^{r}), \qquad (x;q) = (x;q)_{\infty} := \prod_{r\geq 0} (1-xq^{r}),
\]
the $pq$-symbol by 
\[
(x;p,q) := \prod_{r,s\geq 0} (1-xp^rq^s),
\]
and the theta function by 
\[
\theta(x;q) := (x,q/x;q).
\]

The elliptic gamma function is given by 
\[
\Gamma(z) = \Gamma(z;p,q) := 
\frac{(pq/z;p,q)}{z;/p,q)} = 
\prod_{r,s\geq 0} \frac{1-p^{r+1}q^{s+1}/z}{1-p^rq^sz}.
\]
In particular by $\Gamma$ we do not imply Euler's gamma function (i.e. the one where $\Gamma(n+1)=n!$). 

For all these functions we use the notation that multiple parameters or $\pm$ signs in the part before the semicolon indicate a product of functions. For example
\[
(a,b,c;q) = (a;q) (b;q) (c;q), \qquad \Gamma(az^{\pm 1}) = \Gamma(az) \Gamma(a/z).
\]

The reflection equations for $\theta$ and $\Gamma$ are given by 
\[
\theta(z;q) = \theta(q/z;q), \qquad \Gamma(z) = \frac{1}{\Gamma(pq/z)}.
\]
These two functions satisfy the difference equations
\[
\theta(qz;q) = -\frac{1}{z} \theta(z;q), \qquad 
\Gamma(pz) = \theta(z;q) \Gamma(z), \qquad 
\Gamma(qz) = \theta(z;p) \Gamma(z).
\]
The fact that the elliptic gamma function satisfies a difference equation for both $p$ and $q$ is due to its $p\leftrightarrow q$ symmetry $\Gamma(z;p,q) = \Gamma(z;q,p)$. We will often need to use many instances of the difference equation for the theta function at once, using the formula
\[
\theta(q^kz;q) = \left( -\frac{1}{z}\right)^k q^{-\binom{k}{2}}\theta(z;q), \qquad k\in \mathbb{Z}.
\]

Observe that $\theta(z;q)$ is a holomorphic function on $\mathbb{C}^* = \mathbb{C} \setminus \{0\}$, with zeros at $z\in q^{\mathbb{Z}}$, and that $\Gamma(z)$ is an meromorphic function on $\mathbb{C}^*$ with poles at $z\in p^{-\mathbb{Z}_{\geq 0}}q^{-\mathbb{Z}_{\geq 0}}$ and zeros at $z\in p^{\mathbb{Z}_{>0}} q^{\mathbb{Z}_{>0}}$. For generic choices of $p$ and $q$ these poles and zeros are all simple. 
The residue of the elliptic gamma function at 1 is
\[
Res( \Gamma(z), z=1) = - \frac{1}{(p;p)(q;q)},
\]
and the other residues can be obtained from this residue by observing that 
\[
\frac{\Gamma(z)}{\Gamma(zp^{k} q^{l})} = \frac{1}{
p^{-l\binom{k}{2}} q^{-k\binom{l}{2}} \left( -\frac{1}{z}\right)^{kl}
\prod_{s=0}^{l-1} \theta(q^s z;p) \prod_{r=0}^{k-1} \theta(p^rz;q)}
\] 
has a removable singularity at $z=p^{-k} q^{-l}$  (for $k,l\geq 0$).

Let us end this section with a note on integration contours. The contours of all integrals appearing in this article are deformations of the unit circle whose deformations serve to ensure that certain poles are kept inside the contour, while others are left outside. For an elliptic hypergeometric integral, the integrand is a product of elliptic gamma functions, and the poles $z=\frac{1}{a} q^{-\mathbb{Z}_{\geq 0}} p^{-\mathbb{Z}_{\geq 0}}$ of terms $\Gamma(az)$ should be kept outside the contour, while the poles $z=aq^{\mathbb{Z}_{\geq 0}} p^{\mathbb{Z}_{\geq 0}}$ of terms $\Gamma(a/z)$ in the integrand should be kept inside the contour. For an basic hypergeometric integral, the poles of terms $1/(az;q)$ should be kept outside the contour, while the poles of terms $1/(a/z;q)$ should be kept inside the contour. We will use this convention throughout the article without explicitly mentioning it in each case. 

\section{The elliptic hypergeometric beta integral}
The elliptic beta integral evaluation was discovered first by Spiridonov \cite{Spi1}. It can be interpreted as an integral version of Frenkel and Turaev's \cite{FT} summation of an elliptic hypergeometric series. It is given by 
\begin{equation}\label{eqellbetaeval}
\frac{(p;p)(q;q)}{2} \int_{C} \frac{\prod_{r=1}^6 \Gamma(t_r z^{\pm 1})}{\Gamma(z^{\pm 2})} \frac{dz}{2\pi i z} = \prod_{1\leq r<s\leq 6} \Gamma(t_rt_s),
\end{equation}
under the balancing condition $\prod_{r=1}^6 t_r =pq$. The contour $C$ is chosen as a deformation of the unit circle which includes the poles of the integrand at $z=t_rp^l q^k$ for $k,l \geq 0$ and excludes the poles at $\frac{1}{t_r} p^{-l} q^{-k}$ for $k,l\geq 0$. Such a contour exists for generic values of the parameters (explicitly, whenever $t_rt_s \not \in p^{\mathbb{Z}_{\leq 0}} q^{\mathbb{Z}_{\leq 0}}$ for all (possibly equal) $r$ and $s$). In particular we can take the unit circle as a contour whenever $|t_r|<1$ for all $r$.

There are extensions of this integral with more parameters which also satisfy interesting equations, though not necessarily evaluations. In particular we are interested in the more general integral 
\begin{equation}\label{eqellbetagenm}
E_m = \frac{(p;p)(q;q)}{2} \int_{C} \frac{\prod_{r=1}^{6+2m} \Gamma(t_r z^{\pm 1})}{\Gamma(z^{\pm 2})} \frac{dz}{2\pi i z} 
\end{equation}
under the balancing condition $\prod_{r=1}^{6+2m} t_r = (pq)^{1+m}$. The case $m=0$ thus satisfies an evaluation formula. For $m=1$ we obtain transformation formulas \cite{Rainstrafo} which show that the $E_1$ has the symmetry group $W(E_7)$, the Weyl group of type $E_7$. For $m=2$ there exist multibasic identities \cite{Rainslwood}, for example relating $E_2( ;p,q)$ to $E_2(;p,q^2)$ for certain specific values of the $t_r$ on both sides.

\section{The general framework of the limits}\label{secgenmeth}
This section discusses the main goals of the paper. We describe which limits we want to take, and how we aim to take those limits.

We would like to obtain limits of the elliptic beta integral $E_m$ for all $m$. The analysis we perform can be split into two parts, an analytical part and a combinatorial part. The analytical part of the analysis, given in Section \ref{secasym}-\ref{secsymres}, is performed for all $m$. 
We do not know how to perform the combinatorial part for all $m$ at once. In this article, in Sections \ref{secweyl}-\ref{seclimits}, we perform the combinatorial analysis for the $m=0$ case. In a forthcoming article we will also perform the analysis for the $m=1$ case. For higher values of $m$ we do not know a feasible method of performing the combinatorial analysis completely.
% 
%The first restriction we impose upon ourselves is to only consider the limits of the elliptic beta integral with $m=0$. Many of our arguments will carry over to the case of general $m$, and in some cases where this is not an extra burden we immediately tackle the general problem. However, some parts of our argument are so computationally intensive that it would quite a lot of work to extend those. We intend to write another article describing the $m=1$ case. Since only few interesting equations are known for beta integrals with $m>1$, we refrained from performing a complete study of those cases. 

Since the parameters $t_r$ of the elliptic beta integral satisfy the balancing condition $\prod_r t_r=(pq)^{m+1}$, which depends on $p$, we can not take the limit as $p\to 0$ while keeping all other parameters constant. In particular we must describe how those parameters depend on $p$ as $p\to 0$. We will consider the cases where each parameter $t_r= u_r p^{\alpha_r}$ for some $\alpha_r\in \mathbb{R}$ and some non-zero complex number $u_r\in \mathbb{C}^*$ which are independent of $p$. The parameter $q$ will remain fixed, independently of $p$. The balancing condition of the $t_r$ thus becomes equivalent to the balancing conditions $\sum_r \alpha_r=m+1$ and $\prod_r u_r = q^{m+1}$ for the new parameters.

As noted previously, the elliptic beta integral with $m=0$ satisfies an evaluation formula. We can normalize this equation to 
\[
\frac{(p;p)(q;q)}{2\prod_{1\leq r<s\leq 6} \Gamma(t_rt_s)} \int_{C} \frac{\prod_{r=1}^6 \Gamma(t_r z^{\pm 1})}{\Gamma(z^{\pm 2})} \frac{dz}{2\pi i z} = 1.
\]
As this rescaled integral is constant function it is obvious that the limit $p\to 0$ can be taken and that the result is still 1. Of course, this is not an interesting result. Thus we want to find a limit of the elliptic beta integral which is some basic hypergeometric function which is not trivially equal to 1. To do this we restrict the allowed methods of obtaining the limit: in particular, we insist on obtaining this limit by interchanging limit and integral, or a similar technique. Moreover we allow ourselves to change the expression in a few simple ways, described below, before changing limit and integral.  On the other hand, we exclude techniques which first apply an (elliptic) hypergeometric identity before taking the limit. 

The most common change we perform before taking the limit is to multiply the integral by a rescaling factor which is a product of monomials in the parameters of the integral other than $p$ (e.g. $u_r$ and $q$) with exponents depending on $v$ (where we write $p=xq^v$, so $v$ is in essence $\log(p)$). Note that such a rescaling factor will only be necessary for $m>0$, as for $m=0$ we know that the integral converges as $p\to 0$ without multiplying it first with a prefactor.

Interchanging limit and integral only makes sense if we can keep the same contour for all values of $p$. Moreover, the integrand needs to show the same kind of asymptotic behavior as the integral. Otherwise a rescaling factor which makes the integrand have a proper limit will make the limit of the integral vanish (it can not occur that the integral blows up in the limit while the integrand has a proper limit, as the integral is bounded by the maximum value of the integrand times the fixed length of the contour). 

As mentioned, we allow ourselves to rewrite the elliptic integral before interchanging limit and integral. Ignoring the possible rescaling of the integral, these changes can take one or more of the following forms:
\begin{itemize}
\item We can break the symmetry of the integrand, by multiplying with a proper elliptic function (see Section \ref{sec6} for details of the symmetry breaking);
\item We can shift the contour by some $p$-dependent amount; and/or
\item We pick up (a $p$-dependent number of) residues while moving the contour over the associated poles. Subsequently we may shift the summation variable of this sum of residues. 
\end{itemize}
In the latter case, the condition that we should be able to interchange integral and limit is interpreted that we want to be able to interchange sum and limit. 

It turns out that that the limit $p\to 0$ often does not exist, while the limit for $p\to 0$ along $q^s$-geometric sequences for some $s\in \mathbb{N}$ does exist. Thus we write $p=xq^{v}$ and take limits as $v\to \infty$ under the condition $v\in s\mathbb{Z}$. Here $s$ is typically twice the smallest common denominator of rationally chosen $\alpha_r$'s (so that $q^{\frac12 v\alpha_r}$ is always an integer power of $q$).

\section{Asymptotic behavior of the elliptic gamma function and the theta function}\label{secasym}
In this section we describe the limiting behavior of the functions of the elliptic beta integral and the theta function when $p=xq^v$ and $v\to \infty$. The arguments of the functions are allowed to depend on $p$ as some constant times some (rational) power of $p$, i.e. the arguments are $t p^{\alpha}$. To obtain nice limits we then have to assume that $v$ is in some arithmetic sequence such that $v\alpha\in 2\mathbb{Z}$ for all relevant $v$. This last  condition explains why we insist $\alpha\in \mathbb{Q}$. 

Several bounds on $q$-Pochhammer symbols and theta functions were already given in the appendix of \cite{vdBR4}. Here we present a stronger theorem which extends those results.
\begin{theorem}
Define 
\[
g(x) = \frac16 x(x-1)(2x-1).
\]
Let $p=xq^v$. For any $M>0$ and any $\epsilon>0$ there exist $C_1, C_2>0$ such that
\[
C_1 \leq \left| \Gamma(p^{\alpha} z) q^{\frac12 v^2 (g(\alpha)-g(\{\alpha\})) }
\left(z^{\binom{\alpha}{2} - \binom{\{\alpha\}}{2}}
x^{g(\alpha) - g(\{\alpha\})}
q^{\frac12 \left( \binom{\{\alpha\}}{2} - \binom{\alpha}{2} \right)} \right)^v
 \right| \leq C_2,
\]
for all $\alpha \in [-M,M]$, all $z$ in the annulus $\frac{1}{M} \leq |z| \leq M$, all $v$ such that $|p|<|q|$ and such that $\left|1-\frac{p^{\alpha} z}{y} \right| >\epsilon$ for all poles and zeros $y$ of $\Gamma(p^{\alpha} z)$. 
\end{theorem}
\begin{proof}
We use the notation $\{\alpha\} \in [0,1)$ and $\lfloor \alpha \rfloor \in \mathbb{Z}$ for the fractional part, respectively floor, of $\alpha$, thus $\alpha = \{\alpha\} + \lfloor \alpha \rfloor$. Using the difference equations of the elliptic gamma function and the theta functions  we can write
\begin{align*}
\Gamma(p^{\alpha} z) &= \Gamma(p^{\{\alpha\}} z) \prod_{r=1}^{\lfloor \alpha \rfloor} \theta(p^{\alpha-r} z;q) 
= \Gamma(p^{\{\alpha\}} z) \prod_{r=1}^{\lfloor \alpha \rfloor}  \theta(x^{\alpha-r} q^{v(\alpha-r)} z;q) 
\\ &= \Gamma(p^{\{\alpha\}} z) \prod_{r=1}^{\lfloor \alpha \rfloor}  \theta(x^{\alpha-r} q^{\{v(\alpha-r)\}} z;q) \left( -\frac{1}{x^{\alpha-r} q^{\{v(\alpha-r)\}} z}\right)^{\lfloor v(\alpha-r)\rfloor} q^{-\binom{\lfloor v(\alpha-r)\rfloor}{2}}.
\end{align*}
Now observe that 
\[
\binom{\lfloor a \rfloor}{2} + \lfloor a\rfloor \{a\} 
= \frac12 (a-\{a\}) (a-\{a\}-1) + (a-\{a\}) \{a\} 
%= \frac12 a(a-1) - \frac12 \{a\} (2a-1) + \frac12 \{a\}^2 + a\{a\} - \{a\}^2 
= \binom{a}{2}  - \binom{\{a\}}{2},
\]
which allows us to simplify the exponent of $q$ in each term. We can collect the powers of $z$, $x$ and $q$ in the product together. For the exponent of $z$ we calculate
\[
\sum_{r=1}^{\lfloor \alpha \rfloor} \lfloor v(\alpha-r)\rfloor = 
\lfloor \alpha \rfloor \lfloor v\alpha \rfloor - \frac{v}{2} \lfloor \alpha \rfloor (\lfloor \alpha\rfloor +1)
= v\binom{\alpha}{2} - v\binom{\{\alpha\}}{2} - \{v\alpha\} \lfloor \alpha \rfloor,
\]
for the exponent of $x$
\[
\sum_{r=1}^{\lfloor \alpha \rfloor}  (\alpha-r) \lfloor v(\alpha -r)\rfloor
= v(g(\alpha) - g(\{\alpha\}))+ \{v\alpha \} \left( \binom{\{\alpha\}}{2} - \binom{\alpha}{2} \right),
\]
and for the exponent of $q$
\[
\sum_{r=1}^{\lfloor \alpha \rfloor} \binom{ v(\alpha-r)}{2} - \binom{\{v(\alpha-r)\}}{2} 
= \frac12 v^2 (g(\alpha)-g(\{\alpha\})) 
+ \frac12 v \left( \binom{\{\alpha\}}{2} - \binom{\alpha}{2} \right)
-\frac12 \binom{\{v\alpha\}}{2} \lfloor \alpha \rfloor.
\]
Thus we find that
\begin{align*}
\Gamma(p^{\alpha} z) &
z^{v  \left(\binom{\alpha}{2} - \binom{\{\alpha\}}{2} \right)}
x^{v(g(\alpha) - g(\{\alpha\}))}
q^{\frac12 v^2 (g(\alpha)-g(\{\alpha\})) 
+ \frac12 v \left( \binom{\{\alpha\}}{2} - \binom{\alpha}{2} \right)}
\\ &= \Gamma(p^{\{\alpha\}} z) 
 z^{\{v\alpha\} \lfloor \alpha \rfloor} x^{\{v\alpha \} \left( \binom{\alpha}{2} - \binom{\{\alpha\}}{2} \right)} q^{\frac12 \binom{\{v\alpha\}}{2} \lfloor \alpha \rfloor}
 \prod_{r=1}^{\lfloor \alpha \rfloor}  \theta(x^{\alpha-r} q^{\{v(\alpha-r)\}} z;q)  \left( -1\right)^{\lfloor v(\alpha-r)\rfloor} .
 \end{align*}
 Now every term on the right hand side can be uniformly bounded, using Lemma A.1 of \cite{vdBR4} for the theta functions and Lemma A.2 for the elliptic gamma function. Note that the number of theta functions appearing is itself also bounded by $M+1$.
\end{proof}
The last equation in the proof actually shows that when $v\alpha \in 2\mathbb{Z}$ we  have
\begin{align*}
\Gamma(p^{\alpha} z) &
z^{v  \left(\binom{\alpha}{2} - \binom{\{\alpha\}}{2} \right)}
x^{v(g(\alpha) - g(\{\alpha\}))}
q^{\frac12 v^2 (g(\alpha)-g(\{\alpha\})) 
+ \frac12 v \left( \binom{\{\alpha\}}{2} - \binom{\alpha}{2} \right)}
= \Gamma(p^{\{\alpha\}} z) 
\prod_{r=1}^{\lfloor \alpha \rfloor}  \theta(x^{\alpha-r}z;q).
 \end{align*}
So that the limit $p\to 0$, with $p$ chosen along a geometric sequence $p\in x q^{\frac{2}{\alpha}\mathbb{Z}}$, of a rescaling of the elliptic gamma function $\Gamma(p^{\alpha} z)$ can be easily calculated. For this we use  
\[
\lim_{p\to 0} \Gamma(p^{\alpha} z) = \begin{cases}  \frac{1}{(z;q)} & \alpha =0, \\
1 & 0<\alpha <1. \end{cases}
\]

We can derive similar theorems for theta functions as a corollary (theta functions are the quotient of two elliptic gamma functions). A direct proof would allow for vastly better constants, but we don't really care about the constants.
\begin{corollary}
Let $p=xq^v$.
\begin{itemize}
\item 
For any $M>0$ and any $\epsilon>0$ there exist $C_1, C_2>0$ such that 
\[
C_1 \leq \left| \theta(p^{\alpha} z;q) q^{\frac12 v^2 \alpha^2 }
\left(z^{\alpha} x^{\alpha^2 }q^{-\frac12 \alpha} \right)^v \right| \leq C_2,
\]
for all $\alpha \in [-M,M]$, all $z$ in the annulus $\frac{1}{M} \leq |z| \leq M$, all $v$ such that $|p|<|q|$ and such that $\left|1-\frac{p^{\alpha} z}{y} \right| >\epsilon$ for all zeros $y$ of $\theta(p^{\alpha} z;q)$. 
\item 
For any $M>0$ and any $\epsilon>0$ there exist $C_1, C_2>0$ such that 
\[
C_1 \leq \left| \theta(p^{\alpha} z;p) \left(q^{ \binom{\alpha}{2} - \binom{\{\alpha\}}{2} } \right)^v
 \right| \leq C_2,
\]
for all $\alpha \in [-M,M]$, all $z$ in the annulus $\frac{1}{M} \leq |z| \leq M$, all $v$ such that $|p|<|q|$ and such that $\left|1-\frac{p^{\alpha} z}{y} \right| >\epsilon$ for all zeros $y$ of $\theta(p^{\alpha} z;p)$. 
\end{itemize}
\end{corollary}
\begin{proof}
We have $\theta(p^{\alpha} z;q) = \frac{\Gamma(p^{\alpha+1}z)}{\Gamma(p^{\alpha} z)}$, thus the result follows from the above theorem, and similarly for $\theta(p^{\alpha} z;p)$. We are  actually slightly cheating here as the set of zeros of the gamma functions we have to avoid is larger than the set we avoid for the theta functions. In order to fix that we have to repeat the proof given above for the bounds on the elliptic gamma function for this situation.
\end{proof}
The statements of these bounds leads us to the following definition
\begin{definition}
Let $p=xq^v$ and suppose $f(z)$ is a meromorphic function on $\mathbb{C}^*$. Let $fob(\alpha)$ and $sob(\alpha)$ be functions $\mathbb{R} \to \mathbb{R}$, respectively $\mathbb{R} \to \mathbb{C}^*$, such that for any $M>0$ and any $\epsilon >0$ there exist $C_1, C_2>0$ such that 
\[
C_1 < \left| f(p^{\alpha} z) q^{\frac12 v^2 fob(\alpha)} sob(\alpha)^v \right|<C_2,
\]
for all $\alpha \in [-M,M]$, all $z$ in the annulus $\frac{1}{M} \leq z \leq M$, and all integer $v$ such that $|p|<|q|$ and such that $\left| 1-\frac{p^{\alpha}z}{y} \right|>\epsilon$ for all poles and zeros $y$ of $f(p^{\alpha} z)$.

Then $fob(\alpha)$ and $sob(\alpha)$ are called the first, respectively, second order behavior of $f$.
\end{definition}
Note that a high first order behavior means that, for small values of $p$, the function $f$ is large. On the other hand a small value of $|sob(\alpha)|$ implies that $f$ is large (though the second order behavior is always less influential than the first order behavior).

Let us end this section with some remarks on the form of the first and second order behavior of the elliptic gamma function. While at first sight one would think that $g(\alpha)-g(\{\alpha\})$ being the difference of a cubic and a piecewise cubic function is itself piecewise cubic, it is actually piecewise quadratic.
\begin{lemma}
Let $p$ be a polynomial of degree $n$, then $p(x)-p(\{x\})$ is a piecewise $n-1$st degree polynomial.
\end{lemma}
\begin{proof}
It suffices to proof this for $p(x)=x^n$. Then we find that on the interval $x\in [k,k+1]$ we have
$p(x)-p(\{x\}) = p(x) - p(x-k) = x^n -(x-k)^n$ is an $n-1$'st degree polynomial.
\end{proof}
This lemma also shows that $\binom{\alpha}{2}-\binom{\{\alpha\}}{2}$ is a piecewise linear function. We could have also derived this last fact from its relation with $g(\alpha)-g(\{\alpha\})$. Indeed $g(\alpha)-g(\{\alpha \})$ is a continuously differentiable function with derivative
\[
\frac{d}{da} g(a) -g(\{a\}) = 2\left(\binom{a}{2} - \binom{\{a\}}{2} \right).
\]

\section{How to take integral limits}\label{sec6}
In this section we describe how to obtain basic hypergeometric integrals as a limit. In particular we determine the conditions on $\alpha_r$ for which such an integral limit is possible. This section is written for arbitrary values of $m$. 

For simplicity we assume in this section only that our parameters are ordered such that $\alpha_1\leq \alpha_2\leq \cdots \leq \alpha_{2m+6}$. This choice allows for an easy seemingly $m$-independent way to represent the cases in which we can take integral limits. In the other sections we use the opposite ordering of the $\alpha_r$'s because that corresponds with the fundamental Weyl chamber with respect to a standard basis of the root systems $E_6$ and $A_5$.

In order to take the limit we need to make sure we obtain a $p$-independent contour. Thus we need to make sure that all poles of the integrand which go to 0 as $p\to 0$ are inside the contour, and all poles which go to infinity as $p\to 0$ are outside the contour. We do not want to pick up residues for an integral limit, so this condition must be valid without moving the contour over some poles. 

If we do manage to obtain a $p$-independent contour then interchanging limit and integral will always work, since the contour is contained  in some compact set, and the integrand converges uniformly on compact sets which do not contain the zeros and poles of the integrand (if we allow for the correct prefactor determined by the first and second order behavior of the integrand). Note that this argument does not ensure that the resulting integral does not vanish, which would imply that the underlying limit is uninteresting to us.

\begin{theorem}\label{thmintlim}
Suppose $\alpha_r, \beta_r \in \mathbb{Q}_{\geq 0}$ and $\gamma_r,\delta_r\in \mathbb{Q}$ and let 
\[
I(z) = \prod_{r} \Gamma(p^{\alpha_r} t_r z) \prod_r \Gamma(p^{\beta_r} u_r/z)
\prod_r \theta(p^{\gamma_r} v_r z;q) \prod_r \theta(p^{\delta_r} w_r z;p).
\]
Write $p=xq^v$ as usual and suppose 
\[
\lim_{v\to \infty} q^{\frac12 fob\cdot v^2} sob^v I(z) = I_0(z),
\]
where $v$ is taken along some arithmetic sequence, then 
\[
\lim_{v\to \infty} q^{\frac12 fob\cdot v^2} sob^v \int_{C} I(z) \frac{dz}{2\pi i z} = 
\int_C I_0(z) \frac{dz}{2\pi i z},
\]
where\;Êthe contour $C$ is a deformation of the positively oriented unit circle which includes the poles of $\Gamma(u_r/z)$ for those $r$ with $\beta_r=0$, and excludes the poles of $\Gamma(t_rz)$ for those $r$ with $\alpha_r=0$. 
\end{theorem}
\begin{proof}
In the previous section we have seen that the convergence of the elliptic gamma functions and theta functions is uniform on compacts away from the poles and zeros. If $C$ does not run through a zero of $I(z)$ convergence on $C$ is uniform and thus the limit holds. If $C$ does meet a zero of $I(z)$ we can slightly shift the contour to avoid the zero, take the limit for this new contour, and finally move the contour back to its original position.
\end{proof}
If we happen to have elliptic gamma functions in the denominator we can use the reflection equation $\Gamma(z)=\Gamma(pq/z)$ to move them to the numerator. If we have elliptic gamma functions of the form $\Gamma(t z^2)$ we can use the doubling formula $\Gamma(tz^2)= \Gamma(\pm \sqrt{t} z, \pm \sqrt{pt} z, \pm \sqrt{qt} z, \pm \sqrt{pqt} z)$ to change them to gamma functions of the desired form.

Now let us consider the elliptic beta integral \eqref{eqellbetagenm} specialized at $t_r=u_rp^{\alpha_r}$. From the poles of the elliptic gamma function we can immediately see that the poles of the integrand of the elliptic beta integral \eqref{eqellbetaeval} are located at $z=t_r q^k p^l$  (inside the contour, for $k,l\in \mathbb{Z}_{\geq 0}$) and at $z=t_r^{-1} q^{-k}p^{-l}$ (outside the contour, again with $k,l\in \mathbb{Z}_{\geq 0}$) . Thus if any of the $\alpha_r<0$, we have poles outside the contour which converge to $0$ (the one at $t_r^{-1}$) and poles inside the contour converging to infinity (the one at $t_r$). Therefore, even allowing for shifts of the contour, it becomes impossible to find a contour which is independent of $p$ and which passes these poles on the correct side. On the other hand if all $\alpha_r\geq 0$ we do not have to shift the contour to find a proper contour, and indeed the previous theorem provides the limit.

By using a trick which breaks the symmetry of the integrand we can take limits in more cases than just those with $\alpha_r\geq 0$. Rains and the author previously used this trick in \cite{vdBR1} and \cite{vdBR4}. Consider the identity
\[
\frac{\theta(w_1z,w_2z,w_3z,\frac{w_1w_2w_3}{z};q)}{\theta(z^{2};q)}
+ (z\leftrightarrow \frac{1}{z}) = \theta(w_1w_2,w_1w_3,w_2w_3;q),
\]
which is just a reformulation of the famous theta addition formula.
We can now use the $z\to \frac{1}{z}$ symmetry of the original integrand. Indeed if $I(z) = I(\frac{1}{z})$ and $f(z)+f(\frac{1}{z}) = c$, and if we choose the integration contour to be invariant under $z\to z^{-1}$ then 
\[
\int I(z) \frac{dz}{2\pi i z}  = 
\frac{1}{c}  \int I(z) f(z) \frac{dz}{2\pi i z} 
+ \frac{1}{c} \int I(z) f(\frac{1}{z}) \frac{dz}{2\pi i z}
= \frac{2}{c}  \int I(z) f(z) \frac{dz}{2\pi i z}. 
\]
Here the final identity follows by making the change of variables $z\to \frac{1}{z}$ in the second integral of the middle expression, thereby making it equal to the first integral of the middle expression. 

In our case we obtain the identity
\begin{align}
\frac{(q;q)(p;p)}{2}  &\int \frac{\prod_{r} \Gamma(t_r z^{\pm 1})}{\Gamma(z^{\pm 2})} \frac{dz}{2\pi i z}  \nonumber 
\\&  = 
(q;q)(p;p) \int \frac{\prod_r \Gamma(t_r z^{\pm 1}) \theta(w_1z,w_2z,w_3z,\frac{w_1w_2w_3}{z};q)}{\Gamma(z^{\pm 2}) \theta(z^2,w_1w_2,w_1w_3,w_2w_3;q)}
\frac{dz}{2\pi i z} \nonumber 
\\ &= (q;q)(p;p) \int \frac{\prod_r \Gamma(t_r z^{\pm 1}) \theta(w_1z,w_2z,w_3z,\frac{w_1w_2w_3}{z};q)}{\theta(w_1w_2,w_1w_3,w_2w_3;q)}
\theta(z^{-2};p)
\frac{dz}{2\pi i z}.  \label{eqsymbreak1}
\end{align}
The integrand of this new representation of the elliptic beta integral still has the same poles as the original integrand, so we have not gained anything yet. However, we now have the freedom to choose the new parameters $w_i$ as we please. In particular, choosing $w_1=t_1$ and using the identity $\Gamma(t_1z) \theta(w_1z;q) = \Gamma(pt_1z)$ allows us to remove the poles at $z=\frac{1}{t_1} q^{-k}$ ($k\in \mathbb{Z}_{\geq 0}$). Setting also $w_2=t_2$ and $w_3=t_3$ we obtain the expression
\begin{align*}
(q;q)(p;p) 
\int \frac{\prod_{r=1}^3  \Gamma(pt_r z, \frac{t_r}{z}) \prod_{r\geq 3} \Gamma(t_rz^{\pm 1}) \theta(\frac{t_1t_2t_3}{z};q)}{\theta(t_1t_2,t_1t_3,t_2t_3;q)}
\theta(z^{-2};p) \frac{dz}{2\pi i z}.
\end{align*}
Moreover replacing $z\to p^{-\zeta}z$, $t_r\to u_rp^{\alpha_r}$ and shifting the contour back to a deformation of the unit circle we obtain
\begin{align}\label{eqsymbroken}
(q;q)(p;p) 
\int &\prod_{r=1}^3  \Gamma(p^{1+\alpha_r-\zeta}u_r z, p^{\alpha_r+\zeta}\frac{u_r}{z}) \prod_{r\geq 3} \Gamma(p^{\alpha_r-\zeta} u_rz, p^{\alpha_r+\zeta} \frac{u_r}{z}) 
\\& \qquad \times \frac{\theta( \frac{u_1u_2u_3}{z} p^{\alpha_1+\alpha_2+\alpha_3+\zeta};q)}{\theta(u_1u_2p^{\alpha_1+\alpha_2},u_1u_3p^{\alpha_1+\alpha_3},u_2u_3p^{\alpha_2+\alpha_3};q)} \theta(p^{2\zeta} z^{2};p) \frac{dz}{2\pi i z} \nonumber 
\end{align}
For this integral we see that we can obtain a limit as long as  $\alpha_r+\zeta\geq 0$ and  $\alpha_r+1-\zeta \geq 0$ for $r=1,2,3$ and $\alpha_r-\zeta \geq 0$ for $r\geq 4$ by Theorem \ref{thmintlim}.

If we now use our assumption that the $\alpha_r$ are ordered, then these conditions imply that we can obtain an integral limit as long as 
\[
\zeta \geq -\alpha_1, \qquad \zeta \leq 1+\alpha_1, \qquad \zeta \leq \alpha_4
\]
While we consider the $\alpha_r$'s as given (we want to find a limit for a specific choice of $\alpha_r$), we are allowed to choose $\zeta$ as we wish. A $\zeta$ satisfying all these conditions exists exactly when the lower bound for $\zeta$ is less than both upper bounds, i.e. if 
\begin{equation}\label{eqintlimpos}
\alpha_1 \geq -\frac12, \qquad \alpha_1 +\alpha_4\geq 0.
\end{equation}
Thus these last equations determine when, given $\alpha$, there exists an integral limit after symmetry breaking.

One might think that, by specializing $w_1$, $w_2$, and $w_3$ differently in 
\eqref{eqsymbreak1}, we can obtain even more options for integral limits. However, we have not managed to find any other useful specializations. We have tried the following two additional options.

Firstly, observe that specializing $w_1=t_1$ and $w_2=\frac{q}{t_1}$ (since
$\theta(\frac{q}{t_1} z ;q) = \theta(\frac{t_1}{z};q)$) in order to cancel poles of both $\Gamma(t_1z)$ and $\Gamma(\frac{t_1}{z})$, turns the denominator-factor $\theta(w_1w_2;q)$  into $\theta(1;q)=0$. Thus this attempt fails in giving us a new integral representation of the elliptic beta integral. 

Secondly, when $\alpha_1$ is very negative we might want to specialize $w_1=t_1$ and $w_2=pt_1$ to cancel even more poles of $\Gamma(t_1z)$. This leads to a new integral representation of the elliptic beta integral, but not to new possibilities for limits. Indeed the conditions we get would include $\alpha_1+\alpha_3\geq 0$, implying $\alpha_1+\alpha_4\geq 0$. Thus, in any new limiting cases we have $\alpha_1<-\frac12$, so from $\alpha_r+\alpha_1\geq 0$ we obtain $\alpha_r\geq -\alpha_1> \frac12$ for $r\geq 5$. But then 
\[
\sum_{r} \alpha_r = (\alpha_1+\alpha_3) +(\alpha_2+\alpha_4) + \sum_{r\geq 5} \alpha_r 
\geq 2(\alpha_1+\alpha_3) + \sum_{r\geq 5} \alpha_r > 0+ 0 + (m+1)=m+1,
\] in violation with the balancing condition.

Now we have answered the question when integral limits exist, the question remains when these limits are meaningful. In particular we must be careful to choose the correct value of $\zeta$ before interchanging limit and integral. If we choose the wrong value of $\zeta$ the asymptotic behavior of the integrand might be larger than that of the integral itself, and thus the rescaling factor needed to make the integrand converge will make the limit of the integral itself vanish.  

To minimize the possibility of the limiting integral vanishing we should choose $\zeta$ such that the asymptotic behavior of the integrand is as small as possible. This is similar to the idea of the saddle point method for approximating integrals. The task of finding the extremal points of the relevant function, the first order behavior of the integrand (for arbitrary choice of $\alpha$), is undertaken in Section \ref{sec8} for the $m=0$ integrand. We do not know how to generalize the results obtained there for general $m$ (though we will do the $m=1$ case in an upcoming paper). 

Preferably though, we would like to determine the asymptotic behavior of the integral explicitly, and then choose $\zeta$ such that the asymptotic behavior of the integrand coincides with it. Observe that this choice of $\zeta$, if it exists, must coincide with the choice which minimizes the size of the integrand, as the size of the integral can not be larger than the size of the integrand (considering the contour is bounded and constant). Since we have an explicit evaluation of the $m=0$ integral, for this case we can easily perform this calculation.

Let us end this section with some examples. We can take a symmetric integral limit in the elliptic beta integral associated to the vector $(0,0,0,0,\frac12,\frac12)$. Indeed we can take a direct limit $p\to 0$ and find 
\begin{align*}
1&= \lim_{p\to 0} \frac{(p;p)(q;q)}{2\prod_{1\leq r<s\leq 4} \Gamma(u_ru_s)
\prod_{r=1}^4 \prod_{s=5}^6 \Gamma(p^{\frac12} u_ru_s) 
\Gamma(pu_5u_6)}
\int_C \frac{\prod_{r=1}^4 \Gamma(u_r z^{\pm 1}) \prod_{r=5}^6 \Gamma(p^{\frac12} u_r z^{\pm 1})}{\Gamma(z^{\pm 2})} \frac{dz}{2\pi i z} \\ 
&= 
\int_C \lim_{p\to 0} \frac{(p;p)(q;q)}{2\prod_{1\leq r<s\leq 4} \Gamma(u_ru_s)
\prod_{r=1}^4 \prod_{s=5}^6 \Gamma(p^{\frac12} u_ru_s) 
\Gamma(pu_5u_6)} 
\frac{\prod_{r=1}^4 \Gamma(u_r z^{\pm 1}) \prod_{r=5}^6 \Gamma(p^{\frac12} u_r z^{\pm 1})}{\Gamma(z^{\pm 2})} \frac{dz}{2\pi i z}
\\ &=
\int_C \frac{(q;q) \prod_{1\leq r<s\leq 4} (u_ru_s;q)}{2 (q/u_5u_6;q)}
\frac{ (z^{\pm 2};q)}{\prod_{r=1}^4 (u_rz^{\pm 1};q)} \frac{dz}{2\pi i z}, 
\end{align*}
which is equivalent to the Askey-Wilson integral evaluation.
Let us now consider the vector $(-\frac12,0,0,\frac12,\frac12,\frac12)$. Notice that this vector is a permutation and a shift along a half-integer vector away from $(0,0,0,0,\frac12,\frac12)$. 
For this new vector we have to break symmetry before taking the limit. We use the case $\zeta=\frac12$ in \eqref{eqsymbroken}. Thus we consider the limit
\begin{align*}
1 &= \lim_{p\to 0} \frac{(q;q)(p;p)}{\prod_{r=2}^3 \Gamma(p^{-\frac12} u_1u_r)
\prod_{r=4}^6 \Gamma(u_1u_r) \Gamma(u_2u_3) \prod_{r=2}^3 \prod_{s=4}^6 \Gamma(p^{\frac12} u_ru_s) \prod_{4\leq r<s\leq 6} \Gamma(pu_ru_s)}
\\& \qquad \times \int_C \Gamma(u_1z^{\pm 1}) \prod_{r=2}^3 \Gamma(p^{\frac12} u_rz^{\pm 1})\prod_{r=4}^6 \Gamma( u_rz,pu_r/z)  \frac{\theta(\frac{u_1u_2u_3}{z};q)}{\theta(u_1u_2p^{-\frac12},u_1u_3p^{-\frac12}, u_2u_3;q)} \theta(pz^2;p) \frac{dz}{2\pi i z}
\\ &= 
\int_C \lim_{p\to 0} \frac{(q;q)(p;p)}{\prod_{r=2}^3 \Gamma(p^{\frac12} u_1u_r)
\prod_{r=4}^6 \Gamma(u_1u_r) \Gamma(p u_2u_3) \prod_{r=2}^3 \prod_{s=4}^6 \Gamma(p^{\frac12} u_ru_s) \prod_{4\leq r<s\leq 6} \Gamma(pu_ru_s)}
\\ & \qquad \times  \Gamma(u_1z^{\pm 1}) \prod_{r=2}^3 \Gamma(p^{\frac12} u_rz^{\pm 1})\prod_{r=4}^6 \Gamma( u_rz,pur_/z)  \theta(\frac{u_1u_2u_3}{z};q)\theta(pz^2;p) \frac{dz}{2\pi i z}
 \\ &= 
 \int_C \frac{(q;q)  \prod_{r=4}^6 (u_1u_r;q) }{(q/u_2u_3;q) \prod_{4\leq r<s\leq 6} (q/u_ru_s;q)} 
 \frac{ 1 }{(u_1z^{\pm 1};q)} \prod_{r=4}^6 \frac{(qz/u_r;q)}{(u_rz;q)}  
 \theta(\frac{u_1u_2u_3}{z};q) (1-z^{-2}) \frac{dz}{2\pi i z}  
\end{align*}
Note that the result here is a symmetry broken version of the Askey-Wilson integral. Indeed 
\begin{align*}
& \frac{(q;q) \prod_{1\leq r<s\leq 4} (u_ru_s;q)}{2 (q/u_5u_6;q)}
\frac{ (z^{\pm 2};q)}{\prod_{r=1}^4 (u_rz^{\pm 1};q)} 
\frac{\theta(u_2z,u_3z,u_4z, \frac{u_2u_3u_4}{z} ;q)}{\theta(u_2u_3,u_2u_4,u_3u_4,z^2;q)}
 \\ & = 
 \frac{(q;q) \prod_{r=2}^4 (u_1u_r;q) }
 {2 (q/u_5u_6;q)  \prod_{2\leq r<s\leq 4} (q/u_ru_s;q)} \frac{1}{(u_1z^{\pm 1};q)} 
 \prod_{r=2}^4 \frac{(q/u_rz;q)}{(u_r/z;q)}  
 \theta(\frac{u_2u_3u_4}{z} ;q)
  (1-z^{-2}) 
\end{align*}
which is exactly the integrand of the limit for $\alpha = (-\frac12,0,0,\frac12,\frac12,\frac12)$, up to the permutation $u_2\leftrightarrow u_5$ and $u_3\leftrightarrow u_6$ and $z\to 1/z$ (note that the balancing condition implies that $u_2u_3u_4=q/u_1u_5u_6$).

In the previous two cases we did not have to restrict $p$ to a $q$-geometric sequence. In the final example we consider we will have to do so. It is associated to the vector $(-\frac12,-\frac12,-\frac12,\frac12,1,1)$, a half-integer shift away from the vector $(0,0,0,0,\frac12,\frac12)$ giving the Askey-Wilson integral. We again consider \eqref{eqsymbroken} to obtain
\begin{align*}
1 & = \frac{(p;p)(q;q)}{\prod_{1\leq r<s\leq 3} \Gamma(p^{-1} u_ru_s) 
\prod_{r=1}^3 \Gamma(u_ru_4) \prod_{r=1}^3 \prod_{s=5}^6 \Gamma(p^{\frac12} u_ru_s)
\prod_{s=5}^6 \Gamma(p^{\frac32} u_4u_s) \Gamma(p^2 u_5u_6)}
\\ & \qquad \times \int_C \prod_{r=1}^3 \Gamma(u_rz^{\pm 1}) \Gamma(u_4z,pu_4/z) 
\prod_{r=5}^6 \Gamma(p^{\frac12} u_rz, p^{\frac32} u_r/z)
\frac{\theta(\frac{u_1u_2u_3}{z} p^{-1};q)}
{\prod_{1\leq r<s\leq 3} \theta(u_ru_sp^{-1};q)}
\theta(pz^2;p) \frac{dz}{2\pi i z}
\\ &=\frac{(p;p)(q;q)}{\prod_{1\leq r<s\leq 4} \Gamma( u_ru_s) 
\prod_{r=1}^3 \prod_{s=5}^6 \Gamma(p^{\frac12} u_ru_s)
\prod_{s=5}^6 \Gamma(p^{\frac32} u_4u_s) \Gamma(p^2 u_5u_6)}
\\ & \qquad \times \int_C \prod_{r=1}^3 \Gamma(u_rz^{\pm 1}) \Gamma(u_4z,pu_4/z) 
\prod_{r=5}^6 \Gamma(p^{\frac12} u_rz, p^{\frac32} u_r/z)
\theta(\frac{u_1u_2u_3}{z} p^{-1};q)
\theta(pz^2;p) \frac{dz}{2\pi i z} 
\end{align*}
Note that now we have elliptic gamma functions of the form $\Gamma(p^{\frac32} y)$ and $\Gamma(p^2 y)$, of which we can't take the limit $p\to 0$. Thus let us write $p=xq^{2v}$ and take the limit $v\to \infty$ for integer $v$. This implies that we take $p\to 0$ along a $q^2$ geometric sequence. Doing so gives the limit
\begin{align*}
1 &= \lim_{v\to \infty} 
\frac{(p;p)(q;q)}{\prod_{1\leq r<s\leq 4} \Gamma( u_ru_s) 
\prod_{r=1}^3 \prod_{s=5}^6 \Gamma(p^{\frac12} u_ru_s)
\prod_{s=5}^6 \theta(x^{\frac12} q^v u_4u_s;q) \Gamma(p^{\frac12} u_4u_s) 
\theta(xq^{2v} u_5u_6;q) \Gamma(p u_5u_6)}
\\ & \qquad \times \int_C \prod_{r=1}^3 \Gamma(u_rz^{\pm 1}) \Gamma(u_4z,pu_4/z) 
\prod_{r=5}^6 \theta(x^{\frac12} q^v u_r/z;q) \Gamma(p^{\frac12} u_rz^{\pm 1})
\theta(\frac{u_1u_2u_3}{xz} q^{-2v};q)
\theta(pz^2;p) \frac{dz}{2\pi i z} 
\\ & = 
\frac{(q;q) \prod_{1\leq r<s\leq 4} (u_ru_s;q)}{(\frac{q}{u_5u_6};q) 
\prod_{s=5}^6 \theta(x^{\frac12} u_4u_s ;q)
\theta(xu_5u_6;q)  }
%\\ & \qquad \times 
\int_C \frac{(qz/u_4;q)  }{\prod_{r=1}^3 (u_rz^{\pm 1};q) (u_4z;q)}
\prod_{r=5}^6 \theta(\frac{x^{\frac12} u_r}{z}, \frac{u_1u_2u_3}{xz};q) (1-z^{-2}) \frac{dz}{2\pi i z},
\end{align*}
where we use that 
\[
\frac{\prod_{r=5}^6 \theta(x^{\frac12} q^v u_r/z;q) \theta(\frac{u_1u_2u_3}{xz} q^{-2v};q)}{\prod_{s=5}^6 \theta(x^{\frac12} q^v u_4u_s;q) \theta(xq^{2v} u_5u_6;q) }
= 
\frac{\prod_{r=5}^6 \theta(x^{\frac12} u_r/z;q) \theta(\frac{u_1u_2u_3}{xz} ;q)}{\prod_{s=5}^6 \theta(x^{\frac12} u_4u_s;q) \theta(x u_5u_6;q) }
\]
by the difference equation for theta functions. The resulting integral is again a symmetry broken version of the symmetric Askey-Wilson integral (in the symmetry breaking function we have to specialize $w_1=u_4$, $w_2=x^{\frac12} u_5$, $w_3=x^{\frac12} u_6$ and $z\to 1/z$).

\section{How to take series limits}
This section describes the second method of obtaining limits. Essentially this method consists of picking up residues by moving the contour over some poles. Moreover we ensure that the contour is shifted to a location, where the limit of the integrand (times the proper rescaling factor) vanishes. Then we would like to take the limit by interchanging limit and sum, and observing that the limit of the integral vanishes. Like the previous section, this section is valid for arbitrary $m$. The entire section is summed up in Theorem \ref{thmserlim}. I think that  the expression for the limit of the integral given in that theorem is quite hard to read and comprehend. To explain where this result comes from we start with an intuitive explanation of the ideas behind the theorem.

In the previous section we already noted that the poles of the integrand $I_0(z)$ of the elliptic beta integral are located at $z^{\pm 1} = t_r q^k p^l$ ($k,l\in \mathbb{Z}_{\geq 0}$). A fixed contour (as $p\to 0$) which is a deformation of the unit circle is impossible if some $\alpha_r<0$ (where we again use the convention $t_r=u_rp^{\alpha_r}$). In that case the first elements in the sequence $t_r q^kp^0$ converge to infinity, while this entire sequence should be contained inside the contour. Likewise the sequence $t_r^{-1} q^{-k} p^0$ should be completely outside the contour, whereas the first terms converge to the origin. 

In order to obtain an expression with a fixed contour we shift the contour over some of the poles of the integrand, picking up the associated residues. We should only pick up residues for the poles on the ``wrong'' side of the unit circle. For any fixed $p$ this is a finite number, but this number increases as $p\to 0$ (and $t_r\to \infty$ since $\alpha_r<0$).  

Let $I_m$ denote the integrand of the elliptic beta integral \eqref{eqellbetagenm}. Then, before taking the limit, we rewrite the integral as 
\[
\int_C I_m(z) \frac{dz}{2\pi i z} = 
\int_{\hat C} I_m(z) \frac{dz}{2\pi i z} + 
\sum_{k=0}^{n(p)} Res(\frac{I_m(z)}{z}, z=t_r q^k)
+ \sum_{k=0}^{n(p)} Res(\frac{I_m(z)}{z}, z=t_r^{-1} q^{-k}),
\]
where the upper bound $n(p)$ of the series depends on $p$ (and is such that $|t_rq^{n(p)}| \approx 1$), the contour $\hat C$ is close to the unit circle for all values of $p$ and excludes the poles $t_rq^k$ and includes the poles $t_r^{-1} q^{-k}$ for 
$k\leq n(p)$. The behavior of $n(p) \approx -\alpha_r \log_q(p)$ is roughly linear in the $q$-logarithm of $p$.  In the case multiple $\alpha_r$ are negative or if $\alpha_r<-1$ we have several such series of residues corresponding to those extra poles which we have to pass over before the contour can be taken to be near the unit circle for all values of $p$. 

In order to obtain a limit from this expression we subsequently want to show that the integrand near the unit circle vanishes in the limit (after proper rescaling). Since the new contour $\hat C$ remains close to the unit circle, the integrand of the new integral vanishes on $\hat C$ in the limit. Since the length of the new contour is roughly constant, the integral then converges to 0. We are left with the limit of the sums of residues, which we want to calculate by interchanging limit and sum. 

If we can interchange limit and sum directly in $\lim \sum_{k =0}^{n(p)} Res(z=t_rq^k)$ the limit will become the series $\sum_{k= 0}^{\infty} \lim Res(z=t_rq^k)$, thus a unilateral series. This interchange is allowed if the size of the residues for large $k$ becomes sufficiently small as $p\to 0$. If this is not the case we have to shift the summation index $k$ so that the large terms in the series are around $k=0$ (for the shifted summation index) and hope that we can interchange limit and integral in this new series. Suppose $p=xq^v$ for some integer $v$ (so we can replace $n(p)$ by $\approx -\alpha_r v$). Then we would write (for some properly chosen $\beta$)
\[
\sum_{k= 0}^{-\alpha_r v} Res (z=t_rq^k) = \sum_{k=-\beta v}^{-(\alpha_r+\beta)v} Res(z=t_rq^{k+\beta v}) = 
\sum_{k=-\beta v}^{-(\alpha_r+\beta)v} Res(z=\frac{t_r}{x^{\beta}} p^{\beta} q^{k})  
\]
and interchange sum and limit in this expression to get
\[
\lim \sum_{k= 0}^{-\alpha_r v} Res (z=t_rq^k) = 
\sum_{k\in \mathbb{Z}} \lim Res(z=\frac{t_r}{x^{\beta}} p^{\beta} q^k).
\]
In this derivation we need $\beta v$ to be an integer. Thus we need $\beta$ to be a rational number, say $g/h$, and while taking the limit we should restrict $p$ to the geometric sequence $x(q^{h})^n$, with $n\in \mathbb{N}$. 

Of course it might be that there are multiple possible values of $k$ where the residues are large. In those cases it turns out we can split the sum of residues in several pieces:  
\[
\sum_{k= 0}^{-\alpha_rv} = \sum_{k=0}^{\gamma_1v} + \sum_{k=\gamma_1v}^{\gamma_2v} + \cdots + \sum_{k=\gamma_s v}^{-\alpha_rv}
\]
such that we can take a limit of each piece by interchanging sum and limit, after a proper shift in the summation index $k$. In this case the first series might be either unilateral (if no shift in $k$ is necessary) or bilateral, but all the subsequent series will be bilateral.

\begin{theorem}\label{thmserlim}
Let us fix $t_r$ and $u_r$. Suppose $\alpha_r,\beta_r, \gamma_r, \delta_r \in \mathbb{Q}$ and 
\[
I(z) = \prod_r \Gamma(p^{\alpha_r} t_r z) \prod_r \Gamma(p^{\beta_r} u_r/z) \prod_r \theta(p^{\gamma_r} v_r z;q) \prod_r \theta(p^{\delta_r} w_r z;p).
\]
Let $fob(\zeta)$ and $sob(\zeta)$ denote the first and second order behavior of $I(zp^{\zeta})$ (the dependence of these functions on $z$, $\alpha_r$, $t_r$, $\beta_r$, $u_r$, etc. is typically suppressed). 
%
%
%be functions such that for any $\zeta\in \mathbb{Q}$ we have
%\[
%\lim_{v\to \infty} I(zp^{\zeta}) q^{\frac12 fob(\zeta;\alpha_r,\beta_r) v^2} sob(\zeta;\alpha_r,\beta_r; t_r, u_r,z)^v = I_{\zeta}(z)
%\]
%exists for $v$ along some arithmetic sequence. 
Let $m_-=\min(\alpha_r,0)$ and $m_+=- \min(\beta_r,0)$. Suppose in the interval $[m_-,m_+]$ the function $fob(\zeta)$ is maximized in the set $M$ by $fob_{max}$. Then $sob(\zeta;\alpha_r,\beta_r;t_r,u_r,z)$ is independent of $z$ for $\zeta\in M$. 

Suppose moreover that within the set $M$ the function $|sob(\zeta)|$ is minimized in a finite set $M_0$ by $sob_{min}$. Then $M_0\subset \mathbb{Q}$. 

Let $p=xq^v$ and restrict $v$ so that $\alpha_r v, \beta_r v \in 2\mathbb{Z}$ for all $r$ and $\mu v\in 2\mathbb{Z}$ for $\mu \in M_0$ (thus $v$ must be chosen along some arithmetic sequence). Then we have
\[
\lim_{v\to \infty} q^{\frac12 fob_{max} v^2} sob_{min}^v \int_C I(z) \frac{dz}{2\pi i z} = -
\sum_{\mu\in M_0, \mu \geq 0} 
\sum_{l \in \mathbb{Z}_{\geq 0} } \sum_{r:\alpha_r+\mu +l\leq 0} seq_{+,\mu,r,l}
+ \sum_{\mu\in M_0, \mu <0} \sum_{l \in \mathbb{Z}_{\geq 0} } \sum_{r:\beta_r-\mu +l <0} seq_{-,\mu,r,l}
\]
where 
\[
seq_{+,\mu,r,l} = \begin{cases}
\sum_{n=0}^{\infty} \lim_{v\to \infty} q^{\frac12 fob_{max} v^2} sob_{min}^v Res(\frac{1}{z}I(z), z=q^{-n} \frac{1}{t_r} p^{\mu} )
 & \mu = -\alpha_r-l \\
 \sum_{n\in \mathbb{Z}} \lim_{v\to \infty} q^{\frac12 fob_{max} v^2} sob_{min}^v Res(\frac{1}{z} I(z), z=q^{-n} \frac{1}{t_r} p^{\mu} x^{-\alpha_r-l-\mu} )
 & \mu< -\alpha_r  -l
\end{cases}
\]
and
\[
seq_{-,\mu,r,l} = \begin{cases}
\sum_{n=0}^{\infty} \lim_{v\to \infty} q^{\frac12 fob_{max} v^2} sob_{min}^v Res(\frac{1}{z}I(z), z=q^{n} u_r p^{\mu} )
 & \mu = l+\beta_r \\
 \sum_{n\in \mathbb{Z}} \lim_{v\to \infty} q^{\frac12 fob_{max} v^2} sob_{min}^v Res(\frac{1}{z} I(z), z=q^{n} u_r p^{\mu} x^{l+\beta_r-\mu} )
 & \mu> l+\beta_r  
\end{cases}
\]

\end{theorem}
\begin{proof}
We have explicit expressions for the first and second order behavior, namely
\begin{align*}
fob(\zeta) &= \sum_r [g(\alpha_r+\zeta) - g(\{\alpha_r+\zeta\}) ]
+\sum_r [g(\beta_r-\zeta) -g(\{\beta_r-\zeta\})] + \sum_r (\gamma_r+\zeta)^2 , \\
sob(\zeta) &=  x^{fob(\zeta)} z^{\frac12\frac{d}{d\zeta} fob(\zeta)}
\prod_{r} \left(\frac{t_r}{\sqrt{q}}\right)^{\binom{\alpha_r+\zeta}{2}-\binom{\{\alpha_r+\zeta\}}{2}} 
\prod_{r} \left(\frac{u_r}{\sqrt{q}}\right)^{\binom{\beta_r-\zeta}{2}-\binom{\{\beta_r-\zeta\}}{2}} 
\prod_r \left(\frac{v_r}{\sqrt{q}} \right)^{\gamma_r+\zeta} 
\prod_r q^{\binom{\delta_r+\zeta}{2} - \binom{\{\delta_r+\zeta\}}{2}}
\end{align*}
Notice that $fob$ is a continuously differentiable, piecewise quadratic function of $\zeta$, and that $sob(\zeta)$ is the product of monomials in $x$, $q$, $z$, $t_r$,  $u_r$ and $v_r$, the exponents of which are piecewise linear functions of $\zeta$. Let $B =-\alpha_r+\mathbb{Z}$ and $\beta_r+\mathbb{Z}$, then $fob$ and $sob$ are quadratic, respectively linear, on any interval $[a,b]$ such that $(a,b) \cap B=\emptyset$. 

From this expression it is clear that on any point on which $fob$ is maximized, we have $\frac{d}{d\zeta} fob(\zeta)=0$, and hence $sob(\zeta)$ is independent of $z$. 

A direct calculation shows that the second order derivative of $fob(\zeta)$ (which exists on each quadratic piece) is always integer. For the derivative we have the formula
\[
2 \sum_r \left[\binom{\alpha_r+\zeta}{2} - \binom{\{\alpha_r+\zeta\}}{2} \right]
- 2 \sum_r \left[ \binom{\beta_r-\zeta}{2} - \binom{\{\beta_r-\zeta\}}{2} \right].
\]
It is obvious that this derivative is rational for $\zeta\in \mathbb{Q}$ and thus in particular is rational in $B$. It follows that on each linear piece $\frac{d}{d\zeta} fob(\zeta)=a\zeta+b$ with $a\in \mathbb{Z}$ and $b\in \mathbb{Q}$. 
Therefore the zeros of $\frac{d}{d\zeta} fob(\zeta)$, and hence also the extremal values of $fob(\zeta)$, are attained in rational points. The only situation in which non-rational points appear as extremal values is when $\frac{d}{d\zeta} fob$ is identically 0 on a linear piece, in which case $fob$ is constant on this entire piece. 

Between two point in $B$ the functions $sob(\zeta)$ is of the form $sob(\zeta) = a b^\zeta$, for some $a$ and $b$ which are monomials in $t_r$, $u_r$, $v_r$, $q$,  $z$ and $x$, with exponents depending on $\alpha_r$, $\beta_r$, $\gamma_r$ and $\delta_r$. In particular, unless $b=1$, $sob$ is minimized on such a linear piece at one of the end points, i.e. in a point in $B$. The previous paragraph shows that the extremal points of $fob$ appear on isolated rational points and certain intervals between elements of $B$, it follows that in $M$ the function  $|sob(\zeta)|$ is minimized either in some points of $B$, or on these isolated rational extremal points of $fob$. Regardless $M_0$ is a set of rational points.  The case $b=1$ which would also allow for non-rational minimal points is excluded in the theorem as we insist on a finite set $M_0$.

Suppose $0\in M_0$, then we can choose $\epsilon \in \mathbb{Q}$, $\epsilon<0$ such that $[\epsilon,0) \cap B\cup M_0 =\emptyset$. Then replacing $z\to zp^{\epsilon}$ before taking the limit (and shifting the contour back to a deformation of the unit circle while not moving over any poles) changes this situation to one with $0\not \in M_0$. Indeed such a transformation shifts all elements of $M_0$ by $-\epsilon$ (and also all $\alpha_r \to \alpha_r+\epsilon$, $\beta_r\to \beta_r-\epsilon$, etc.). Therefore in the remainder we can assume $0\not \in M_0$ and that $\alpha_r,\beta_r \not \in \mathbb{Z}$. 

Now we look at the integral and move the contour over the residues until the new  integration contour $C'$ is contained in the annulus $A= \{z\in \mathbb{C}~|~ |q|<|z|<|q^{-1}|\}$. Thus we move the contour over all the poles of the form $z=a_{lkr} = p^{-l} q^{-k} p^{-\alpha_r} /t_r$ for all $k$, $l$  for which $|a_{lkr}| <1$, and all poles of the form $z=b_{lkr} = p^{l} q^{k} p^{\beta_r} u_r$ for which $|b_{lkr}|>1$. This leads to series of poles of the form 
\[
S_{lr} = \sum_{k=0}^{k_{max}} Res(\frac{I(z)}{z},z = p^{-l} q^{-k} p^{-\alpha_r} /t_r), \qquad 
T_{lr} = \sum_{k=0}^{k_{max}} Res(\frac{I(z)}{z},z = p^{l} q^{k} p^{\beta_r} u_r).
\]
Here $k_{max}$ is chosen such that $|p^{-l -\alpha_r} q^{-k_{max}}/t_r| \in A$. In particular $k_{max}$ increases linearly with $v$. 
Note that if $l+\alpha_r <0$ the number of terms in $S_{lr}$ increases as $p\to 0$, while if $l+\alpha_r>0$ the number of terms decreases (and the series vanishes for small enough $p$). Thus we only have to consider the series with $l+\alpha_r<0$ and with $l+\beta_r<0$ (by assumption $\alpha_r\not \in \mathbb{Z}$, so $l+\alpha_r\neq 0$). 

%For convenience, let us write the residue explicitly 
%\[
%2\pi i Res(\frac{I(z)}{2\pi i z}, z= p^{-l} q^{-k} p^{-\alpha_r}/t_r)
%= \prod_{s\neq r} \Gamma(p^{\alpha_s-\alpha_r -l} q^{-k} \frac{t_s}{t_r})
%\prod_s \Gamma(p^{\beta_r+\alpha_r+l} q^k t_r u_s)
%\]

We can immediately see that the limit of the remaining integral vanishes. Indeed, as $v$ increases in discrete steps to ensure integrality of $\alpha_r v$ etc., the location of the poles in the annulus remains fixed, and thus the contour can be chosen constant. As $0\not \in M_0$ the integrand converges uniformly to 0 (either the first order behavior ensures convergence to 0, or it is constant, in which case the second order behavior ensures such convergence). Thus we are allowed to interchange limit and integral and  see that the limit of the integral vanishes.

Let us now choose $\epsilon>0$, such that $(\mu,\mu+\epsilon) \cap B=\emptyset$ and $(\mu-\epsilon,\mu) \cap B = \emptyset$ for all $\mu \in M_0$. Now a series
$S_{lr}$ can be rewritten as 
\[
S_{lr} = \sum_{k=0}^{k_{max}} Res(\frac{I(z)}{z}, z = p^{-l-\alpha_r} q^{-k} /t_r)
= \sum_{k=0}^{k_{max}} Res(\frac{I(z)}{z}, z = p^{-l -\alpha_r - \frac{k}{v}} x^{\frac{k}{v}} /t_r)
\]
and can be split in parts with $-l-\alpha_r-\frac{k}{v} \in (\mu-\epsilon,\mu + \epsilon)$
 for some $\mu \in M_0$ and the remaining parts. Notice (by a direct calculation) that the first and second order behavior of $Res(\frac{I(z)}{2\pi i z}, z= p^{\zeta} z_0)$ for poles $p^{\zeta} z_0$ equal the first and second order behaviors of $I(p^{\zeta}z_0)$. As a consequence the remaining parts vanish in the limit. Indeed the number of terms grows only linearly in $v$, while the size of the terms is uniformly bounded by an exponentially decreasing function in $v$. 

For the part around $\mu$ we obtain
\[
\sum_{-v\epsilon}^{v\epsilon} Res(\frac{I(z)}{2\pi i z}, z=p^{-l-\alpha_r} 
q^{ (l+\alpha_r+\mu)v - k} /t_r) = 
\sum_{-v\epsilon}^{v\epsilon} Res(\frac{I(z)}{2\pi i z}, z=p^{\mu} 
q^{- k} x^{-l-\alpha_r-\mu} /t_r) 
\]
The summand of the new series times the scaling factor
$q^{\frac12 fob_{max} v^2} sob_{min}^v$ is bounded (in absolute value) by 
\[
C q^{\frac12 [fob(\mu)- fob(\mu - \frac{k}{v})] v^2} 
\left|\frac{sob(\mu)}{sob(\mu-\frac{k}{v})}\right|^v.
\]
for some constant $C$ (notice that none of the residues are close to a pole of the residue). Now $fob$ is quadratic on $(\mu-\epsilon,\mu)$ and on $(\mu,\mu+\epsilon)$, and it has a maximum at $\mu$, so on both pieces we have $fob(\zeta)= fob(\mu) + d (\zeta-\mu)^2$ for some constant $d\leq 0$. Moreover we have that $sob(\zeta)$ is a product of monomials with linear exponents on these two intervals, so that $\frac{sob(\mu)}{sob(\mu-\frac{k}{v})} = f^{\frac{k}{v}}$ for some $f$. Plugging this in gives an upper bound
\[
C q^{- \frac12 d k^2}  |f|^k,
\]
with possibly different $d$ and $f$ for $k>0$, respectively $k<0$. This bound converges to zero exponentially quadratically if $d<0$, and if $d=0$, the extra condition that $|sob|$ is minimized in $\mu$, gives $|f|<1$ if $k>0$, and $|f|>1$ if $k<0$, which implies that the bound converges to zero exponentially. In either case we have a convergent bound on the series, and thus we are allowed to interchange limit and summation and obtain the desired result.

Likewise we obtain the limits of the $T_{lr}$ series.
\end{proof}

Let us end this section with some examples. We first consider the vector $\alpha= (-\frac12,0,0,\frac12,\frac12,\frac12)$ for which we obtained an integral limit in the previous section. The first order behavior for this vector is constant 0. The absolute value of the second order behavior either has a minimum in $\frac12 + \mathbb{Z}$ if $|u_2u_3|<1$ or in $\mathbb{Z}$ if $|u_2u_3|>1$. It turns out that the rescaling factor we need to use in the case $|u_2u_3|>1$ is such that the resulting series we obtain as limit vanish. Thus we assume $|u_2u_3|<1$ and hence $M_0=\{\pm \frac12\}$. The limit of the integral is therefore the sum of two univariate series, one for the residues at $z=q^{-n} \frac{1}{u_1} p^{-\frac12}$ and one at the reciprocals. Due to symmetries of the residues, discussed in the next section, these two series are identical. Therefore we end up with an evaluation for a single unilateral series, which is the evaluation for a very-well poised ${}_6W_5$, equation (II.20) in Gasper and Rahman \cite{GR}. 

For the vector $\alpha=(-1,0,\frac12,\frac12,\frac12,\frac12)$, under the condition $|u_1u_2|<1$, we have  $M_0=\{\pm \frac12\}$ as well. In this case the resulting limit will be the sum of two bilateral series, which are again equal due to symmetries of the residues. Indeed, writing 
$p=xq^v$ the limit of the integral equals the limit of the series 
\[
\sum_{n=0}^v Res(I(z)/z,z=q^{-n} \frac{1}{u_1} p) = 
\sum_{n= 0}^v Res(I(z)/z,z=q^{\frac{v}{2} -n} \frac{1}{u_1} x^{\frac12} p^{\frac12}) = 
\sum_{n= -\frac{v}{2} }^{\frac{v}{2}} Res(I(z)/z,z=q^{-n} \frac{1}{u_1} p^{\frac12} x^{\frac12}),
\]
where we can interchange limit and series in the last expression. This leads to the summation formula for a very-well poised ${}_6\psi_6$, equation (II.33) in Gasper and Rahman \cite{GR}. 

If we consider the vector $\alpha = (-\frac32,0,\frac12,\frac12,\frac12,1)$ under the condition $|u_2u_6|<1$ we find $M_0 =\{\pm \frac32, \pm \frac12\}$. In principle we thus get the sum of 6 series as the limit, though the fact that residues at reciprocal points are equal reduces the number to 3. Writing $p=x^2q^v$, the three series we have to take the limit of are
\begin{align*}
&\sum_{n= 0}^{\frac{v}{2}} Res(I(z)/z,z=q^{-n} \frac{1}{u_1} p^{\frac32}), \\
&\sum_{n=\frac{v}{2}}^{\frac{3v}{2}} Res(I(z)/z,z=q^{-n} \frac{1}{u_1} p^{\frac32})
 = 
 \sum_{n=-\frac{v}{2}}^{\frac{v}{2}} Res(I(z)/z,z=q^{-n} \frac{1}{u_1} p^{\frac12} x^2 )
  , \\
&\sum_{n= 0}^{\frac{v}{2}} Res(I(z)/z,z=q^{-n} \frac{1}{u_1} p^{\frac12}).
\end{align*}
The first and third series lead to a unilateral sequence, while the middle series gives a bilateral sequence. Even more symmetries of the residues show that, apart from a constant, the unilateral sequences of the first and third row are identical, and that both are equal to a very-well poised ${}_6W_5$, for which we have an evaluation formula. The limiting bilateral series from the middle becomes the ${}_6\psi_6$ which we also know how to evaluate. Thus the evaluation formula we obtain in this case does not tell us anything new:
\begin{align*}
1& =\frac{(\frac{qu_1}{u_3},\frac{qu_1}{u_4}, \frac{qu_1}{u_5};q) }{(qu_1^2, \frac{q}{u_3u_4},\frac{q}{u_3u_5}, \frac{q}{u_4u_5}, \frac{q}{u_2u_6};q)}
\frac{\theta(\frac{u_2x}{u_1}, \frac{u_6x}{u_1}, \frac{u_3x^2}{u_1}, \frac{u_4x^2}{u_1}, \frac{u_5x^2}{u_1}, \frac{u_6x^3}{u_1};q)}
{\theta(u_3u_6x,u_4u_6x,u_5u_6x, \frac{x^2}{u_1^2}, \frac{x^4}{u_1^2};q)}
  {}_6W_5(u_1^2  ;u_1u_3,u_1u_4,u_1u_5  ; u_2u_6)
  \\& \qquad  + 
  \frac{(u_1u_3,u_1u_4,u_1u_5,
  \frac{qx^2}{u_1u_3}, \frac{qx^2}{u_1u_4}, \frac{qx^2}{u_1u_6}, \frac{qu_1}{u_3x^2}, \frac{qu_1}{u_4x^2}, \frac{qu_1}{u_5x^2}  ;q)}{(q,\frac{q}{u_3u_4}, \frac{q}{u_3u_5}, \frac{q}{u_4u_5},\frac{q}{u_2u_6},\frac{u_1^2}{x^2},x^2, \frac{qu_1^2}{x^4}, \frac{qx^4}{u_1^2}   ;q)}
 \\& \qquad \qquad \times  \frac{\theta( \frac{u_1u_2}{x}, \frac{u_1u_6}{x}, \frac{u_6x^3}{u_1};q)}{\theta(u_3u_6x,u_4u_6x,u_5u_6x;q)}
  \rpsisx{6}{6}{ \frac{u_1^2}{x^2}, \pm \frac{qu_1}{x^2}, \frac{u_1u_3}{x^2}, \frac{u_1u_4}{x^2}, \frac{u_1u_5}{x^2} \\ \frac{q}{x^2}, \pm \frac{u_1}{x^2}, \frac{qu_1}{u_3x^2}, \frac{qu_1}{u_4x^2}, \frac{qu_1}{u_5x^2}   }{u_2u_6}
 \\&    \qquad   +   \frac{(\frac{qu_1}{u_3},\frac{qu_1}{u_4}, \frac{qu_1}{u_5};q) }{(qu_1^2, \frac{q}{u_3u_4},\frac{q}{u_3u_5}, \frac{q}{u_4u_5}, \frac{q}{u_2u_6};q)}
    \frac{\theta(\frac{u_1u_2}{x^3}, \frac{u_1u_3}{x^2},\frac{u_1u_4}{x^2}, \frac{u_1u_5}{x^2}, \frac{u_1u_6}{x}, \frac{u_6x}{u_1};q)}{\theta( \frac{u_1^2}{x^4}, \frac{1}{x^2}, u_3u_6x,u_4u_6x,u_5u_6x;q)}
{}_6W_5(u_1^2   ;u_1u_3,u_1u_4,u_1u_5  ; u_2u_6).
\end{align*}
Also observe that if we specialize $x\to 1$, then the summand of the bilateral series ${}_6\psi_6$ vanishes for negative $n$, and the series becomes the same very-poised ${}_6W_5$ that already occurs twice in this limit.

\section{Symmetries of the residues}\label{secsymres}
In this section we discuss how symmetries of the integrand of the elliptic beta integral imply symmetries of the residues. As a result of those symmetries  the different series of residues we have to pick up are equal or similar. This significantly reduces the number of different series that can appear as a limit. For example, we will show that given a vector $\alpha$ (such that $t_r=u_rp^{\alpha_r}$) there is essentially a unique bilateral series that we can obtain. This section is valid for all $m$.

Here we used the notation $I_m(z)$ for the integrand of the elliptic beta integral. To be precise we define  
\[
I_m(z) = I_m(t_r;z) := \frac{\prod_{r=1}^{2m+6} \Gamma(t_r z^{\pm 1})}{\Gamma(z^{\pm 2})}, \qquad 
sb(s_r;z) = \frac{\theta(s_1z,s_2z,s_3z, \frac{s_1s_2s_3}{z};q)}{\theta(z^2,s_1s_2,s_1s_3,s_2s_3;q)}
\]
for parameters $t_r$ satisfying $\prod_{r=1}^{2m+6} t_r= (pq)^{m+1}$.
Direct calculations using the difference equations of the elliptic gamma functions and theta functions lead to the following proposition:
\begin{proposition}
The integrand $I(t_r;z)$ satisfies the following symmetries:
\begin{align*}
%I(pz)I(qz) &= I(z) I(pqz) \\
I_m(t_r;z) &= I_m(t_r;z^{-1}) = I_m(\sigma(t_r);z), \qquad \forall \sigma \in S_{2m+6}\\
\frac{I_m(t_r;qz)}{I_m(t_r;z)} &= \frac{I_m(p^{\alpha_r}t_r;qz)}{I_m(p^{\alpha_r} t_r;z)}, \qquad \alpha \in \mathbb{Z}^m \\
\frac{I_m(t_r;qz)}{I_m(t_r;z)} &= \frac{I_m(p^{\alpha_r}t_r;qp^{\frac12}z)}{I_m(p^{\alpha_r} t_r;p^{\frac12} z)}, \qquad \alpha \in \mathbb{Z}^m+ \{\frac12\}^{2m+6}.
\end{align*}
Here $S_{2m+6}$ denotes the group of permutations of $2m+6$ elements, and $\alpha$ is always chosen so that $\sum_r \alpha_r=0$. The symmetry-breaking term satisfies 
\[
sb(s_1,s_2,s_3;qz)= sb(s_1,s_2,s_3;z), \qquad 
sb(s_1,s_2,s_3;z) + sb(s_1,s_2,s_3;\frac{1}{z}) = 1.
\]
\end{proposition}
Note that the last two equations for the integrand of the elliptic beta integral imply that it satisfies 
\[
I_m(pz)I_m(qz)=I_m(pqz)I_m(z),
\]
the defining equation for an elliptic hypergeometric integral.

As a corollary to this proposition we obtain several equations between residues
\begin{proposition}
Throughout this proposition we assume that the parameters are generic (i.e. $t_r^{\pm 1}t_s^{\pm 1}p^c q^d\neq 1$ for all choices of the signs and any $c,d\in \mathbb{Z}$ only for $c,d=0$). Moreover we assume $k\in \mathbb{Z}_{\geq 0}$ and that $a=t_rp^b$  for some values of $1\leq r\leq 2m+6$ and $b\in \mathbb{Z}_{\geq 0}$ (in particular $aq^k$ is a pole of $I_m(z)$).
\begin{enumerate}
\item We have
\[
Res( \frac{I_m(z)}{z} sb(s_1,s_2,s_3;z),z=aq^k) = sb(s_1,s_2,s_3;a) Res(\frac{I_m(z)}{z}, z=aq^k).
\]
\item We have
\[
Res( \frac{I_m(z)}{z}, z=\frac{1}{aq^k}) = - Res( \frac{I_m(z)}{z},z=aq^k).
\]
\item We have 
\[
Res( \frac{I_m(z)}{z}, z=paq^k) = \frac{I_m(pa)}{I_m(a)} Res( \frac{I_m(z)}{z}, z=aq^k)
\]
where the quotient $I_m(pa)/I_m(a)$ should be interpreted as the value of $I_m(pz)/I_m(z)$ at the point $a$. In the relevant situations $I_m(pz)/I_m(z)$ typically has a removable pole at $a$.
\item We have for $\alpha \in \mathbb{Z}^{2m+6}$ with $\sum_r \alpha_r=0$ 
\[
Res( \frac{I_m(t_r p^{\alpha_r};z)}{z}, z=aq^k) = \frac{I_m(t_rp^{\alpha_r};a)}{I_m(t_r;a)} Res( \frac{I_m(t_r;z)}{z}, z=aq^k),
\]
and for $\alpha \in \mathbb{Z}^{2m+6}+ (\frac12)^{2m+6}$ with $\sum_r \alpha_r=0$ we have
\[
Res( \frac{I_m(t_r p^{\alpha_r};zp^{\frac12})}{z}, z=aq^k) = \frac{I_m(t_rp^{\alpha_r};ap^{\frac12})}{I_m(t_r;a)} Res( \frac{I_m(t_r;z)}{z}, z=aq^k).
\]
\end{enumerate}
\end{proposition}
\begin{proof}
The proofs are all relatively straightforward calculations involving the aforementioned symmetries. 
\begin{enumerate}
\item Due to the genericity of the parameters we see that $sb$ is analytic at $z=aq^k$. Combining this with the $q$-ellipticity of $sb$ we find
\begin{align*}
Res( \frac{I_m(z)}{z} sb(s_1,s_2,s_3,z),z=aq^k) &=
 sb(s_1,s_2,s_3,aq^k) Res(\frac{I_m(z)}{z}, z=aq^k) \\& 
= sb(s_1,s_2,s_3,a) Res(\frac{I_m(z)}{z}, z=aq^k).
\end{align*}
\item This follows from the $z\to z^{-1}$ symmetry of the integrand. Indeed, for a small positively oriented circle $C$ around the pole $a$ (such that no other poles of the integrand are contained inside this contour) we have
\begin{align*}
Res( \frac{I_m(z)}{z}, z=a) &= \frac{1}{2\pi i} \int_C \frac{I_m(z)}{z} dz = \frac{1}{2\pi i} \int_{C^{-1}} \frac{I_m(z^{-1})}{z^{-1}}
(-\frac{1}{z^2}) dz
\\ &= -\frac{1}{2\pi i } \int_{C^{-1}} \frac{I_m(z)}{z} dz = - Res( \frac{I_m(z)}{z}, z=a^{-1}).
\end{align*}
Here we replace $z\to z^{-1}$ in the second equality. The contour $C^{-1}$ will be closed contour containing $a^{-1}$, which is still traversed in positive direction (as the origin is not contained in the circle).
\item An initial calculation gives
\[
Res( \frac{I_m(z)}{z}, z=paq^k) = p Res( \frac{I_m(pz)}{pz}, z=aq^k) = Res( \frac{I_m(pz)}{z}, z=aq^k).
\]
Subsequently we remark that, for generic values of the parameters, if $I$ has a pole at $a$ and $k$ is a positive integer, then 
$I(pz)/I(z)$ is analytic in $aq^k$ (i.e. it has a removable singularity at that point). Thus we get
\[
Res( \frac{I_m(pz)}{z}, z=aq^k) = 
Res( \frac{I_m(z)}{z} \frac{I_m(pz)}{I_m(z)}, z=aq^k) = 
\frac{I_m(paq^k)}{I_m(aq^k)} Res( \frac{I_m(z)}{z} , z=aq^k) 
\]
Using the $q$-ellipticity of $I_m(pz)/I_m(z)$ now gives the desired result.
\item This follows from calculations nearly identical to the ones performed above, now using the $q$-ellipticity of 
$I(t_rp^{\alpha_r};z)/I(t_r;z)$.
\end{enumerate}
\end{proof}

These identities for the residues show that many different series of residues are equal, at least up to a constant. This means that their respective limits are also equal. Thus, if we are interested in knowing which basic hypergeometric functions can occur as limit, we have to consider many fewer cases than one might otherwise think. 
\begin{corollary}
The equations from the previous proposition lead to the following results:
\begin{enumerate}
\item We only have to consider limits of residues from the symmetric integral: The series obtained from taking limits in a symmetry broken version of the integral are equal to the series obtained from the symmetric integral, up to a multiplicative constant. It should be noted that specializing $s_1$, $s_2$, $s_3$ in a proper way will make this constant 0, thus reducing the number of series in (and hence the complexity of) the limit. However, this usually does not help because we can only remove one of the two identical series identified in the next point.
\item Taking limits of the symmetric integral, we obtain a factor 2 for moving the contour not just over the poles inside the unit circle, which have to be excluded from the original contour, but also the poles outside the unit circle, which have to be included in the original contour: The factor $-1$ in the equation between the residues negates the factor $-1$ obtained from the fact that we move the contour in opposite direction over the pole. 
\item Given a vector $\alpha$ and parameters $t_r$ we have: 
\begin{itemize}
\item There is at most one different unilateral series for each different value of $r$: The series of residues starting from $p t_rq^k$ is, up to a constant, identical to the series starting from $t_rq^k$. In particular we obtain at most a single different unilateral series for each $t_r$. Due to the permutation symmetry of the $t_r$'s these series are very similar for different $t_r$'s, as long as the associated $\alpha_r$'s are equal up to an integer.
\item All bilateral series for a given vector $\alpha$ (thus for different values of $r$ and different parts of the same series) are essentially equal: Writing $p=xq^v$, we obtain a bilateral series as the limit of 
\[
\sum_{k=\beta v}^{\gamma v} Res( z=t_r q^k) = 
\sum_{k=(\beta-\rho) v}^{(\gamma-\rho)v} Res(z=t_r q^{k+\rho v})
= 
\sum_{k=(\beta-\rho) v}^{(\gamma-\rho)v} Res(z=\frac{t_r}{x^{\rho}}  q^{k} p^{\rho}).
\]
Changing the value of $x$ allows us to obtain any value for $t_r/x^{\rho}$.
\end{itemize}
\item The same series appear for two vectors $\alpha$ which differ by an element of $\mathbb{Z}^{2m+6} \cup (\mathbb{Z}^{2m+6} + (\frac12)^{2m+6})$. The limit will still be different for these different $\alpha$'s as which series we have to include in the limit is still dependent on $\alpha$.
\end{enumerate}
\end{corollary}
Note that the unilateral series are equal to a bilateral series with one of the parameters specialized, to be precise: a bottom parameter specialized to $q$, or a top parameter specialized as $1$.

\section{The Weyl group $W(E_6)$}\label{secweyl}
Starting from this section we will approach the combinatorial problem of finding which values of $\alpha_r$ and $\zeta$ lead to interesting limits. We will restrict ourselves from now on to the $m=0$ (evaluation) version of the elliptic beta integral.
The Weyl group of type $E_6$ describes the symmetric structures we encounter when considering the elliptic beta integral evaluation. Understanding those symmetries is vital to be able to make several quick reductions and simplifications. Thus we start our exposition on the evaluation formula by discussing the Weyl group of type $E_6$. A more thorough discussion of Weyl groups can be found in \cite{Hum}.

\subsection{The root system}
The root system, a set of vectors satisfying certain specific requirements, is the starting point for defining a Weyl group. While the 6 in $E_6$ denotes that we consider a root system in a 6-dimensional space, for our purposes it is convenient to consider the root system to be embedded in an 8-dimensional space, as a subset of the root system of $E_8$. 
\begin{definition}
Let $e_r$ denote the $r$'th standard unit vector in $\mathbb{R}^8$ and set 
\[
\rho = \frac12 \sum_{r=1}^8 e_r = (\frac12,\frac12,\frac12,\frac12,\frac12,\frac12,\frac12,\frac12) \in \mathbb{R}^8.
\]
The root systems of type $E_8$, $E_7$ and $E_6$ are given by
\begin{align*}
R(E_8) & := \{ v\in \mathbb{Z}^8 \cup (\mathbb{Z}^8 + \rho) ~|~ v\cdot v=2, v\cdot \rho \in \mathbb{Z} \} \\
R(E_7) & := \{ v\in R(E_8) ~|~ v\cdot \rho=0\} \\
R(E_6) &:= \{v\in R(E_7) ~|~ v\cdot (e_7+e_8) =0 \}.
\end{align*}
\end{definition} 
The roots of $E_8$ are given by the vectors $(0,0,0,0,0,0,\pm 1,\pm 1)$ with arbitrary signs and all its permutations, and by $(\pm \frac12,\pm \frac12,\cdots, \pm \frac 12)$ with an even number of minus signs. The roots of $E_7$ are the permutations of $(-1,0,0,0,0,0,0,1)$ and $(-\frac12,-\frac12,-\frac12,-\frac12,\frac12,\frac12,\frac12,\frac12)$. 

The set of roots of $E_6$ in this representation is only $S_6 \times S_2$ symmetric (due to the special role of the last two coordinates). They are given by the $S_6$ permutations of $(-1,0,0,0,0,1;0,0)$, the $S_2$ permutations of $(0,0,0,0,0,0;-1,1)$ and the $S_6\times S_2$ permutations of $(-\frac12,-\frac12,-\frac12,\frac12,\frac12,\frac12;-\frac12,\frac12)$.

\begin{definition}
A basis $\Delta$ of a root system $R$ is a subset of the roots such that any root $r$ or its inverse $-r$ can be written as a positive linear combination of the roots in $\Delta$. The elements of the basis are called simple roots.
\end{definition}
An example of a basis for $R(E_6)$ is given by the set
\[
\{ e_1-e_2,e_2-e_3,e_3-e_4,e_4-e_5,e_5-e_6,u\}, \qquad u = \frac12(-e_1-e_2-e_3+e_4+e_5+e_6+e_7-e_8).
\]
It can be extended to a basis for $R(E_7)$ by adding $e_6-e_7$ and to $R(E_8)$ by adding both $e_6-e_7$ and $e_7+e_8$\footnote{For those interested in the geometry of this configuration we would like to note that $\rho$ is the longest root of $E_8$ for this basis.}. 

\begin{definition}
The root lattice $\Lambda$ is the lattice generated by the roots $R$. Thus it is the set of all $\mathbb{Z}$-linear combinations of roots.
\end{definition}

\subsection{The Weyl group}
Once we know a root system we can define the Weyl group which then describes the symmetries of the root system.
\begin{definition}
The Weyl group $W$ associated to a root system $R$ is the group generated by the reflections $s_r$ in the hyperplanes orthogonal to the root $r\in R$. That is 
\[
s_r(v) := v- 2\frac{v\cdot r}{r\cdot r} r.
\]
\end{definition}
An important property of root systems is that they are preserved under the action of the Weyl group. Indeed, the definition for $R(E_8)$ immediately shows that 
$r_1 \cdot r_2 \in \mathbb{Z}$ for $r_1,r_2 \in R(E_8)$ and 
$r_1\cdot r_1 =2$. Thus $s_{r_1}(r_2)= r_2 - (r_1\cdot r_2) r_1$ is another vector in the lattice $\mathbb{Z}^8 \cup (\mathbb{Z}^8 + \rho)$ with squared norm equal to 2, and integer inner product with $\rho$, i.e. $s_{r_1}(r_2) \in R(E_8)$ as well. 

Subsequently it is clear that if one restricts a set of roots to a certain hyperplane the remaining roots are invariant under the new smaller Weyl group as it is generated by reflections preserving this hyperplane. This shows that $R(E_7)$ is preserved by $W(E_7)$ and $R(E_6)$ is preserved by $W(E_6)$. 

Note that the fact that root systems are preserved under the Weyl group implies that the Weyl group is finite. Indeed, as the root system spans the underlying vector space, the action of any element in a group of linear transformations on that vector space is completely determined by its action on the roots. This gives an interpretation of the Weyl group as a subgroup of the group of permutations of the set of roots. 

A more concise representation of a Weyl group is obtained by observing that it is already generated by the reflections in the roots of a basis of the root system. If we want to prove invariance under the Weyl group we thus only have to check invariance under those reflections.

Like the root systems the associated root lattices are preserved by the Weyl group action. If we consider the root lattice to be a commutative group, we can define the affine Weyl group.
\begin{definition}
The affine Weyl group $\tilde W$ associated to a root system $R$ of type $E_n$ is given as the semi-direct product $\tilde W := W \ltimes \Lambda$. \footnote{In general one would need the co-root lattice instead of the root lattice, but in this case these two lattices are identical.}
\end{definition}
A natural representation of the affine Weyl group is obtained by letting the Weyl group elements act as (products of) reflections, and the lattice act by translations (the translation $t_{\lambda}$ associated to $\lambda \in \Lambda$ being $t_{\lambda}(v) =v+\lambda$). This gives a representation of $\tilde W(E_6)$ not just on the space 
$\{v\in \mathbb{R}^8 ~|~ v\cdot \rho = v\cdot (e_7+e_8)=0\}$ spanned by the roots, but in fact on any space of the form $V_{a,b}:= \{v\in \mathbb{R}^8 ~|~ v\cdot \rho=a, v\cdot (e_7+e_8) =b\}$ for $a,b\in \mathbb{R}$. We will be mostly interested in the representation on $V_{\frac12,0}$. 

Due to the symmetry breaking between the first six and the last two coordinates, we typically write elements of $V_{\frac12,0}$ as 
$(\alpha_1,\alpha_2,\ldots,\alpha_6;\zeta)\in \mathbb{R}^6 \times \mathbb{R}$, which stands for 
 $v:=(\alpha_1,\alpha_2,\ldots,\alpha_6,\frac12 -\zeta,\zeta-\frac12) \in \mathbb{R}^8$. The equation $(e_7+e_8)\cdot v=0$ is then always satisfied, so we only have the balancing condition $1=2\rho \cdot v=\sum_{r=1}^6 \alpha_r$.

\section{Maximizing and minimizing the integrand of the $m=0$ elliptic beta integral}\label{sec8}
In this section we consider the asymptotic behavior of the integrand of the elliptic beta integral, with $t_r=u_rp^{\alpha_r}$ substituted and $z$ replaced by $zp^{-\zeta}$. The goal of this section is to determine, for every $\alpha$, which values of $\zeta$ maximize respectively minimize this asymptotic behavior. The idea is that we need to know where the maximum is attained to apply Theorem \ref{thmserlim}. If we want to obtain an integral limit it is best to attempt to do so with $\zeta$ at a location where the first order behavior is minimized as explained in Section \ref{sec6}. 

To be precise we will consider the rescaled integrand
\[
\IR_{0}= 
\frac{\prod_{r=1}^6 \Gamma(u_r p^{\alpha_r-\zeta} z, u_r p^{\alpha_r+\zeta}\frac{1}{z})}{
\prod_{1\leq r<s\leq 6} \Gamma(p^{\alpha_r+\alpha_s} u_ru_s)
\Gamma(p^{-2\zeta} z^2, p^{2\zeta} z^{-2})}.
\]
and look at (mainly) its first order behavior.

However, let us first quickly consider the asymptotic behavior of the function which we use for symmetry breaking, 
\[
sb(s_r;z) = \frac{\theta(s_1z,s_2z,s_3z,\frac{s_1s_2s_3}{z};q)}{\theta(z^2,s_1s_2,s_1s_3,s_2s_3;q)}.
\]
Using the substitutions $s_r \to p^{\beta_r} s_r$ and $z\to z p^{-\zeta}$ we find that 
\[
 sb(s_rp^{\beta_r};zp^{-\zeta})  
= sb(s_r x^{\beta_r};zx^{-\zeta})
%\frac{\theta(s_1z x^{\beta_1-\zeta},s_2z x^{\beta_2-\zeta},s_3z x^{\beta_3-\zeta},\frac{s_1s_2s_3}{z} x^{\beta_1+\beta_2+\beta_3 +\zeta};q)}{\theta(z^2 x^{-2\zeta},s_1s_2x^{\beta_1+\beta_2},s_1s_3x^{\beta_1+\beta_3},s_2s_3x^{\beta_2+\beta_3};q)}
\]
as long as $p=xq^v$ and $v\zeta, v\beta_r\in \mathbb{Z}$. In particular the first order behavior of $sb$ is 0 (i.e.\ non-existent) and the second order behavior is 1. This means that the asymptotic behavior of the symmetric integrand is exactly equal to the asymptotic behavior of the asymmetric integrand. Therefore we can and will restrict our attention to the first.

\begin{lemma}\label{lem91}
The rescaled symmetric integrand $\IR_{0}$ has first order behavior given by 
\begin{equation}\label{eqfob}
fob(\alpha,\zeta) = \sum_{1\leq r<s\leq 6} g(\{\alpha_r+\alpha_s\}) -\sum_{r=1}^6 \big[g(\{\alpha_r-\zeta\}) + g(\{\alpha_r+\zeta\})\big] 
\end{equation}
where $g(x) = \frac16 x(x-1)(2x-1)$ as before. The second order behavior  is given by 
\begin{align} \label{eqsob}
sob(\alpha,\zeta;u_r,x,q,z) &   = 
%
%\frac{\prod_{r=1}^6 ( (\frac{u_rz}{\sqrt{q}})^{\binom{\alpha_r-\zeta}{2} - \binom{\{\alpha_r-\zeta\}}{2}}
%(\frac{u_r}{z\sqrt{q}})^{\binom{\alpha_r+\zeta}{2} - \binom{\{\alpha_r+\zeta\}}{2}}}
%{(\frac{z^2}{\sqrt{q}})^{\binom{-2\zeta}{2} - \binom{\{-2\zeta\}}{2}} (\frac{1}{z^{2}\sqrt{q}})^{\binom{2\zeta}{2} - \binom{\{2\zeta\}}{2}}} 
%x^{fob(\alpha,\zeta)} 
%\\ &= 
x^{fob(\alpha,\zeta)} 
z^{-\frac12 \frac{\partial}{\partial\zeta} fob(\alpha,\zeta)}
\prod_{r=1}^6 u_r^{\frac12 \frac{\partial}{\partial\alpha_r} fob(\alpha,\zeta) }
\\& \qquad \times
q^{-\frac14 -\binom{\{2\zeta\}}{2} +
\frac12 \sum_{r=1}^6 \left[\binom{\{\alpha_r-\zeta\}}{2}  + \binom{\{\alpha_r+\zeta\}}{2}\right]
-\frac12 \sum_{1\leq r<s\leq 6} \binom{\{\alpha_r+\alpha_s\}}{2} 
}  \nonumber
%x^{fob(\alpha,\zeta)} z^{\frac12 \frac{d}{d\zeta} fob(\alpha,\zeta)}.
%\prod_{r=1}^6 \left(\frac{u_r}{\sqrt{q}}\right)^{2\binom{\alpha_r}{2} - \binom{\{\alpha_r-\zeta\}}{2}- \binom{\{\alpha_r+\zeta\}}{2}}
\end{align}
\end{lemma}
\begin{proof}
We obtain this result by viewing the integrand as a product of elliptic gamma functions whose first and second order behaviors are known. At first this leads to the expression
\begin{align*}
fob(\alpha,\zeta) & = \sum_{r=1}^6 [g(\alpha_r-\zeta) + g(\alpha_r+\zeta) - g(\{\alpha_r-\zeta\}) -g(\{\alpha_r+\zeta\}) ] \\ & \qquad - \sum_{1\leq r<s\leq 6} [g(\alpha_r+\alpha_s)- g(\{\alpha_r+\alpha_s\})]
\\ & \qquad - g(2\zeta) -g(-2\zeta) + g(\{2\zeta\})+ g(\{-2\zeta\}), 
\end{align*}
which can be simplified to the expression given by expanding the $g$ terms whose arguments are no fractional parts, using the balancing condition, and observing that $g(x)=-g(1-x)$ and $\{2\zeta\}=1-\{-2\zeta\}$.  Likewise we have simplified the expression for the second order behavior.
% Calculation for sob in outsidepolytopearticlem0calc
\end{proof}

Let us first make a few observations about these expressions. First of all we observe that the first and second order behaviors are 1-periodic in $\zeta$, and almost $\zeta \to -\zeta$ symmetric (the only non-symmetric term is the power of $z$ in the second order behavior). Moreover we observe that typically we would like to choose $\zeta$ at an extreme value of the first order behavior, which ensures that the derivative to $\zeta$ of $fob(\alpha,\zeta)$ vanishes at those points, and thus the power of $z$ in the second order behavior disappears. This is as desired as any rescaling factor should surely be independent of the integration or summation variable. 

Let us now focus on the first order behavior. The integrand is largest for values of $\zeta$ that maximize this function and smallest when $fob$ is minimized. Only if the extremal value is non-unique does the second order behavior come into play. 
The following lemma describes the basic symmetries of $fob$
\begin{lemma}
The function $fob$ is invariant under the action of the affine Weyl group $W(\tilde E_6)$. Here we should take as coordinates $(\alpha_1,\ldots,\alpha_6,\frac12-\zeta,\zeta-\frac12)\in V_{\frac12,0}$.

Moreover we have 
\begin{equation}\label{eqfobnegrefl}
fob(\alpha,\zeta) = -fob(w-\alpha,\zeta), \qquad \forall w\in \mathbb{Z}^6 : \sum_r w_r=2.
\end{equation}
\end{lemma}
\begin{proof}
All identities follow from direct verification of invariance under a set of generators of $W(\tilde E_6)$. The affine Weyl group is generated by the simple reflections (which together generate the ordinary Weyl group) together with one translation along a root. \footnote{This is true as the roots of $E_6$ form a single orbit under the Weyl group action.} 

The invariance of $fob$ under permutations of the $\alpha_r$ is immediate, so we only need to check explicitly the reflection in $u=\frac{-e_1-e_2-e_3+e_4+e_5+e_6+e_7-e_8}{2}$. Explicitly this reflection maps
\begin{align*}
(\alpha;\zeta) &\to (\frac{1+\alpha_1-\alpha_2-\alpha_3-\zeta}{2},
\frac{1-\alpha_1+\alpha_2-\alpha_3-\zeta}{2},
\frac{1-\alpha_1-\alpha_2+\alpha_3-\zeta}{2} \\ & \qquad \qquad ,
\frac{\alpha_4-\alpha_5-\alpha_6+\zeta}{2},
\frac{-\alpha_4+\alpha_5-\alpha_6+\zeta}{2},
\frac{-\alpha_4-\alpha_5+\alpha_6+\zeta}{2} \\ & \qquad \qquad ;
\frac{\zeta+\alpha_4+\alpha_5+\alpha_6}{2}= \frac{\zeta+1 -\alpha_1-\alpha_2-\alpha_3}{2})
\end{align*}
Using $g(\{x\})=-g(\{-x\})$ the verification that $fob(s_u(\alpha;\zeta))=fob(\alpha;\zeta)$ now follows.

As for translations, we note that $fob$ is obviously invariant for translations of the $\alpha_r$ with integers, and in particular also for the translation along the root $e_1-e_2$. 

The final equation follows from direct verification, using once again that $g(\{x\}) = -g(\{-x\})$.
\end{proof}

If the question was to determine when $fob$ was maximized over all $\alpha_r$ and $\zeta$, we know that this maximum occurs somewhere in a fundamental domain of these symmetries; thus we would only have to consider the case $(\alpha;\zeta)$ in such a fundamental domain. However, the question at hand is to find the extremal values of $fob$ as a function of $\zeta$, for given, fixed, values of $\alpha_r$. 

We can still use these symmetries. We can restrict $\zeta$ to a fundamental domain of the symmetries which fix the $\alpha_r$. Moreover we can use the symmetries of the form $w(\alpha;\zeta) = (f(\alpha);g(\alpha;\zeta))$ for some functions $f$ and $g$; i.e. those symmetries where the new $\alpha_r$ do not depend on $\zeta$, to equate the problem of finding the extremal values for $\alpha$ with the same problem for $f(\alpha)$.

%These symmetries allow us to restrict our attention to a fundamental domain for these symmetries. As long as we know for which $\zeta$'s the function $fob$ is maximized or minimized in a fundamental domain, we have found $\zeta$ for which the function is globally maximal, respectively minimal. Moreover we can find all other locations where the function is maximized by taking the orbit of these points. 
Thus we restrict ourselves to 
\begin{equation}\label{eqdom1}
\alpha_1 \geq \alpha_2 \geq \alpha_3\geq \alpha_4\geq \alpha_5 \geq \alpha_6 \geq \alpha_1-1, \qquad 
0\leq \zeta\leq \frac12.
\end{equation}
This is not the smallest set we can find as we could even use the symmetries to restrict to the case $\alpha_1+\alpha_2\leq 1$, but it small enough that we can analyze the function completely for these values of $\alpha_r$. As the orbit of the set of $\alpha$'s considered is the entire domain (even if we just consider translations), this knowledge suffices to obtain the locations of the extremal values for all values of $\alpha$.

\begin{theorem}\label{thm93}
The location of the extremal values of $fob(\alpha,\zeta)$ in the domain given by \eqref{eqdom1} are, for generic values of $\alpha$, determined by the following table:

\begin{tabular}{cc|cc}
$\alpha_1+\alpha_2$ & $\alpha_4+\alpha_5$ & min & max \\
\hline
 $\leq 1$ &$\geq 0$ & $[0,\max(0,-\alpha_5)]$ & $[\min(\frac12,1-\alpha_1),\frac12]$ \\
$\geq 1$  &$\leq 0$ &  $[\min(\frac12,\alpha_2),\frac12]$ & $[0,\max(0,\alpha_4)]$ \\
$\leq 1$ & $\leq 0$ &  $-\alpha_4-\alpha_5-\alpha_6$ & $[0,\max(\alpha_4)]$ or $[\max(\frac12,1-\alpha_1),\frac12]$ \\
$\geq 1$ & $\geq 0$ &  $[0,\max(0,-\alpha_5)]$ or $[\min(\frac12,\alpha_2),\frac12]$ & $\alpha_3+\alpha_4+\alpha_5$ 
\end{tabular}

Here generic values of $\alpha$  are those for which no pair of parameters sums to an integer (i.e. $\alpha_r+\alpha_s\not \in \mathbb{Z}$ for all $r<s$). The first two columns of the table give conditions on which part of the domain is under consideration, the last two columns give the appropriate locations of extremal values. If an interval is given, this means that the value is constant and extremal on that interval.
\end{theorem}
This result was originally obtained by splitting the domain in small pieces by choosing the order of the 12 numbers $\{\alpha_r\}$ and $\{-\alpha_r\}$, and fixing 
$\lfloor \alpha_r\rfloor$ for each $r$. On such a piece the analysis was quite simple, as $fob$ becomes explicitly piecewise quadratic (with the jumps between the different quadratic functions only occurring whenever $\zeta = \{\pm \alpha_r\}$ for some $r$). The proof using this method, however, is very tedious as there are a large number of pieces. Thus we present a more concise proof below.  
\begin{proof}
We make a fixed choice of $\alpha$ and consider $fob(\alpha,\zeta)$ as a function of $\zeta$.
First we observe that $fob(\alpha,\zeta)$ is a differentiable function with derivative
\[
\frac{d}{d\zeta} fob(\alpha,\zeta) = 
-2 \sum_{r=1}^6 
\binom{\{\alpha_r+\zeta\}}{2} - \binom{\{\alpha_r-\zeta\}}{2}
\]
This derivative is a piecewise linear function, which vanishes at $\zeta=0$ and $\zeta=\frac12$. As a piecewise linear function it is itself differentiable almost everywhere with derivative given by
\[
\frac{d^2}{d\zeta^2} fob(\alpha,\zeta) = 
-2 \sum_{r=1}^6 \left( \{\alpha_r+\zeta\} + \{\alpha_r-\zeta\} -1 \right).
\]
In particular we see that the second derivative is always an even integer, and that it has jumps of size 2 at $\{\zeta\} = \{\pm \alpha_r\}$. For $\zeta\in [0,\frac12]$ we thus have exactly 6 jumps. If $\{\zeta\} = \{\alpha_r\}$ the jump is downward and if $\{\zeta\}=\{-\alpha_r\}$ the jump is upward.  Moreover we see that 
\[
\frac{d^2}{d\zeta^2} fob (\alpha,0) = 4 \left(3-\sum_{r=1}^6  \{\alpha_r\} \right) 
\]
is a multiple of 4. Thus the second derivative can only vanish at the two endpoints or after an even number of jumps. In order for $fob$ not to be monotone on $[0,\frac12]$ the derivative has to change sign. Since the derivative vanishes at the endpoints, this implies that the second derivative has to change sign twice. This can only happen if the second derivative is $-2$ after 1 jump, $+2$ after 3 jumps and once again $-2$ after 5 jumps; or similarly with negative and positive interchanged. 

Being negative after the first jump is equivalent to the condition $\alpha_4+\alpha_5\leq 0$ (either the second derivative at 0 must be negative, which means $\sum_{r} \{\alpha_r\} \geq 4$, which implies $\alpha_4<0$, or the second derivative at 0 vanishes and the first jump is downward, which means $\alpha_5<0<\alpha_4$ and $\alpha_4<-\alpha_5$, so $\alpha_4+\alpha_5\leq 0$). Being negative after the fifth jump is equivalent to  $\alpha_1+\alpha_2 \leq 1$, by the same argument which showed $\alpha_4+\alpha_5\leq 0$ and the symmetry 
\[
fob(\alpha,\zeta) = fob(\alpha_4+\frac12,\alpha_5+\frac12,\alpha_6+\frac12,\alpha_1-\frac12,\alpha_2-\frac12,\alpha_3-\frac12,\frac12-\zeta).
\]
Since the integral of the second derivative is the difference $\frac{d}{d\zeta} fob(\alpha,\frac12)-\frac{d}{d\zeta} fob(\alpha,0) =0$ we see that the second derivative must be positive after 3 jumps if it is negative after 1 and 5 jumps (as that is the last place where it can be positive at all). 

Let us now calculate the internal location where the derivative vanishes in the case $\alpha_4+\alpha_5\leq 0$ and $\alpha_1+\alpha_2\leq 1$. This location corresponds to a point where the derivative changes from being negative to being positive and thus to a minimum of $fob$. The maxima of $fob$ in this case will be given on one of the two boundaries. Considering the second derivative must be negative after the first jump and then jump upwards twice, we obtain that either $0<\alpha_4<-\alpha_5<-\alpha_6<\alpha_3, 1-\alpha_1$ or 
$0<-\alpha_4<-\alpha_5<-\alpha_6<\alpha_3,1-\alpha_1$. Since we can explicitly calculate the second derivative in each case and integrate that function we can determine the value of the derivative. In both cases we find that for $\zeta \in [-\alpha_5,-\alpha_6]$ we have $\frac{d}{d\zeta} fob(\alpha,\zeta) = 2(\alpha_4+\alpha_5)$. Subsequently we have a period in which the second derivative is 2, so the zero of the $\frac{d}{d\zeta} fob$ must be at $-\alpha_6 - \frac{2(\alpha_4+\alpha_5)}{2}= -\alpha_4-\alpha_5-\alpha_6$. 

In the symmetric case $\alpha_4+\alpha_5\geq 0$ and $\alpha_1+\alpha_2\geq 1$ we can similarly determine the internal point where the derivative vanishes. In this case it must correspond to a maximum of $fob(\alpha,\zeta)$. The minima are located at the endpoints of the interval. In the case $\alpha_4+\alpha_5\geq 0$ and $\alpha_1+\alpha_2\leq 0$ or vice versa, the first order behavior is monotone and thus the minimum is on one endpoint and the maximum on the other. Which is which is easily determined from the sign of the second derivative at those endpoints.
\end{proof}

The reader will realize that we determined where the local extrema occur. To determine which of the two local maxima is the global maximum  in the case $\alpha_1+\alpha_2\leq 1$ and $\alpha_4+\alpha_5\leq 0$, we can calculate the difference $fob(\alpha_r;\frac12)-fob(\alpha_r;0)$. This difference is a piecewise quadratic function of the $\alpha_r$, whose zero set is not polytopal. In particular we cannot give a simple expression for when which of the two local extrema are global, which is why we refrain from doing so. Fortunately it turns out that we do not need to know this information, considering the conditions on $\zeta$ obtained in the next section. 

\section{The value of $\zeta$ for which the integral and integrand have the same first behavior.}
In this section we take a different view of determining which value of $\zeta$ we need to use in our limits. In particular we determine for which values of $\zeta$ the first order behavior of the integrand equals that of the integral itself. This is a requirement if the integrals and series we obtain as limits are to be non-vanishing.

The easiest way to tackle this problem is again to write the elliptic beta integral evaluation as 
\[
1 = \int_C \frac{\prod_{r=1}^6 \Gamma(t_r z^{\pm 1})}{\Gamma(z^{\pm 2}) \prod_{1\leq r<s\leq 6} \Gamma(t_rt_s)} \frac{dz}{2\pi i z}.
\]
For the left hand side it is clear that the first order behavior is 0. The first order behavior of the integrand was calculated to be $fob(\alpha;\zeta)$ in Lemma \ref{lem91}, and must therefore also vanish. The equation
\[
fob(\alpha;\zeta)=0
\]
is invariant under the full group $W(\tilde E_6)$ of symmetries of $fob$, and even under the $\alpha\to w-\alpha$ with $w\in \mathbb{Z}^6$ with $\sum_r w_r=2$ reflection (which negates the value of $fob$). Thus to determine when $fob$ vanishes, we only have to consider a fundamental domain of $(\alpha;\zeta)$ modulo these symmetries.

The standard fundamental domain for the action of $W(\tilde E_6)$ is given (using the basis $\Delta$) by the set
\[
\{ v~|~ \forall \delta \in \Delta: v\cdot \delta \geq  0, 
v\cdot long \leq 1\},
\]
where $long$ is the longest root, which in this case is $e_7-e_8$. Thus it is given by
\begin{equation}\label{eqfdae6}
\alpha_1\geq \alpha_2\geq \alpha_3\geq \alpha_4\geq \alpha_5\geq \alpha_6, \qquad 
\alpha_4+\alpha_5+\alpha_6\geq \zeta\geq 0, \qquad \sum_{r=1}^6 \alpha_r=1.
\end{equation}
Of course, we want a fundamental domain of the action of the group generated by $W(\tilde E_6)$ together with the $fob$-negating reflection \eqref{eqfobnegrefl}. This is given in the following lemma.
\begin{lemma}\label{lem101}
A fundamental domain of the action of the group generated by the affine Weyl group $W(\tilde E_6)$ together with \eqref{eqfobnegrefl} is given by the bounding inequalities 
\begin{equation}\label{eqfdae6p}
\alpha_1\geq \alpha_2\geq \alpha_3\geq \alpha_4\geq \alpha_5\geq \alpha_6 \geq -\zeta, \qquad 
\alpha_4+\alpha_5+\alpha_6\geq \zeta\geq 0,  \qquad \sum_{r=1}^6 \alpha_r=1.
\end{equation}
\end{lemma}
Notice that this is just the fundamental domain of $W(\tilde E_6)$ restricted to the set $\alpha_6+\zeta\geq 0$. 
\begin{proof}
The following transformation is contained in the product of $W(\tilde E_6)$ with \eqref{eqfobnegrefl}:
\begin{multline*}
(\alpha;\zeta) \mapsto (\alpha_1+\frac{\alpha_5+\alpha_6}{2}, \alpha_2+\frac{\alpha_5+\alpha_6}{2}, \alpha_3+\frac{\alpha_5+\alpha_6}{2}, \alpha_4+\frac{\alpha_5+\alpha_6}{2}, \\
\zeta-\frac{\alpha_5+\alpha_6}{2}, -\zeta-\frac{\alpha_5+\alpha_6}{2};\frac{\alpha_5-\alpha_6}{2})
\end{multline*}
Indeed applying a translation after \eqref{eqfobnegrefl} and negating $\zeta$ allows us to map $(\alpha;\zeta)$ to $(\frac12-\alpha_5,\frac12-\alpha_4,\frac12-\alpha_3,\frac12-\alpha_2,\frac12-\alpha_1,-\frac12-\alpha_6;\frac12-\zeta)$. If we subsequently apply $(354) s_u (13524) s_u$, where we use the cycle-notation for permutations of the six $\alpha_r$ parameters, we arrive at the desired result. 

Now observe that this transformation maps the fundamental domain of $W(\tilde E_6)$ to itself. Moreover it maps the expression $\alpha_6+\zeta$ to $-(\alpha_6 +\zeta)$, which proves that modulo this transformation we can assume $\alpha_6+\zeta\geq 0$. 
\end{proof}
It turns out that the hyperplane $\alpha_6+\zeta=0$ splits the fundamental domain of $W(\tilde E_6)$ in a part where $fob$ is positive and a part where it is negative. Indeed this is how we originally determined the fundamental domain of the extended group (the part where $fob$ is positive has to be a fundamental domain of the extended group as long as the zero-set of $fob$ has empty interior). Being the boundary between positive and negative values, this hyperplane must be part of the zero-locus of $fob$. All locations where $fob$ vanishes are given below. 
\begin{lemma}\label{lemfobvanish}
Within the fundamental domain \eqref{eqfdae6p} the function $fob$ vanishes if and only if either $\alpha_6+\zeta=0$ or $\zeta=0$.
\end{lemma}
\begin{proof}
Restricting $fob$ to the given domain we can simplify the expression for $fob$. For example, on this domain we have $0\leq \alpha_1+\alpha_2\leq 1$, so that $\{\alpha_1+\alpha_2\} =\alpha_1+\alpha_2$. This allows us to simplify the expression for $fob$ to 
\[
fob(\alpha;\zeta) = 4\zeta^2 + 1_{\{\alpha_5+\alpha_6<0\}} (\alpha_5+\alpha_6)^2 - 
\sum_{r:\alpha_r<\zeta} (\alpha_r-\zeta)^2.
\]
As $\alpha_5+\alpha_6\leq 0$ implies $\alpha_4\geq \zeta\geq \alpha_5$, this gives us eight distinct pieces on which $fob$ is quadratic (namely one where $\alpha_5+\alpha_6\leq 0$ and seven where $\alpha_5+\alpha_6\geq 0$, which depend on where $\zeta$ falls in the sequence $\alpha_1\geq \alpha_2\geq \cdots \geq \alpha_6$).

On the piece $\alpha_5+\alpha_6\geq 0$ and $\alpha_6\geq \zeta$ the expression for $fob$ is just $fob(\alpha;\zeta)=4\zeta^2$, so it vanishes if and only if $\zeta=0$. Moreover the condition $\zeta=0$ implies that $\alpha_5+\alpha_6\geq 0$ and $\alpha_6\geq \zeta$, so this shows that $fob$ vanishes on the entire hyperplane $\zeta=0$.

%When $\alpha_5 \geq \zeta \geq \alpha_6$ on the other hand we have $fob = 4\zeta^2-(\alpha_6-\zeta)^2 = (\alpha_6+\zeta)(3\zeta-\alpha_6)$. This vanishes if and only if $\alpha_6+\zeta=0$ or $3\zeta-\alpha_6=0$. The latter can only happen if $\alpha_6+\zeta=0$ anyway, as $3\zeta-\alpha_6=2(\zeta-\alpha_6) + (\zeta+\alpha_6) \geq \zeta+\alpha_6$ (since we had assumed $\zeta \geq \alpha_6$). Hence on this part of the fundamental domain $fob$ vanishes if and only if $\alpha_6+\zeta=0$. 

Next we note that when $\alpha_6+\zeta=0$ we have $\alpha_4\geq \zeta \geq \alpha_6$, so we only have to consider the three parts 
\begin{enumerate}
\item $\alpha_5\geq \zeta \geq \alpha_6$;
\item $\alpha_4\geq \zeta\geq \alpha_5$ and $\alpha_5+\alpha_6\leq 0$; 
\item $\alpha_4\geq \zeta \geq \alpha_5$ and $\alpha_5+\alpha_6\geq 0$.
\end{enumerate}
In case (1) we find that 
\[
fob(\alpha;\zeta) =4\zeta^2-(\zeta-\alpha_6)^2 = (\alpha_6+\zeta)(3\zeta-\alpha_6) = (\alpha_6+\zeta)( [\zeta+\alpha_6]+2[\zeta-\alpha_6]). 
\]
It is clear that this expression is always positive and vanishes only if $\alpha_6+\zeta=0$.
In case (2) we obtain
\[
fob(\alpha;\zeta) = 4\zeta^2 + (\alpha_5+\alpha_6)^2 - (\zeta-\alpha_6)^2-(\zeta-\alpha_5)^2 = 2(\alpha_6+\zeta)(\alpha_5+\zeta).
\]
This also vanishes exactly when $\alpha_6+\zeta=0$ (as $\alpha_5+\zeta=0$ also implies $\alpha_6+\zeta=0$). 
In the final case (3) we obtain
\[
fob(\alpha;\zeta) = 4\zeta^2 - (\zeta-\alpha_6)^2-(\zeta-\alpha_5)^2 = 
(\zeta-\alpha_5)(2\zeta+\alpha_5-\alpha_6) + (\alpha_5+\alpha_6)(3\zeta-\alpha_6)
\]
This is a sum of two positive terms, so it only vanishes if both terms are zero. Now $2\zeta+\alpha_5-\alpha_6=0$ if and only if $\zeta=\alpha_5=\alpha_6=0$ and likewise 
$3\zeta-\alpha_6=0$ if and only if $\zeta=\alpha_5=\alpha_6=0$. Thus we see that $fob$ vanishes if and only if both $\zeta-\alpha_5=0$ and $\alpha_5+\alpha_6=0$. These two equations imply $\zeta+\alpha_6=0$, and on the other hand $\zeta+\alpha_6=0$ implies these two equations (given $\alpha_4\geq \zeta\geq \alpha_5$ and $\alpha_5+\alpha_6\geq 0$). 

So we find that in all three case considered, $fob$ vanishes if and only if $\zeta+\alpha_6=0$. As these cases covered the entire intersection of the hyperplane $\zeta+\alpha_6=0$ with the fundamental domain, this shows that $fob$ vanishes on this entire intersection. The last thing we have to show is that $fob$ does  not vanish elsewhere on the fundamental domain, in particular in the region $\alpha_4\leq \zeta$. 

In the region $\alpha_4\leq \zeta$ there are still four different quadratic expressions for $fob$, which are all rather convoluted. A direct calculation shows that they equal
\begin{align*}
fob(\alpha;\zeta) &= 
(\alpha_4+\alpha_5+\alpha_6-\zeta)^2 + 2(\alpha_4-\alpha_5)^2 + 6(\zeta-\alpha_4)^2 - \sum_{r<4:\zeta\geq \alpha_r} (\zeta-\alpha_r)^2 
\\ & \qquad + 
2(3\zeta-\alpha_4-\alpha_5-\alpha_6) (\alpha_4+\alpha_5+\alpha_6-\zeta) 
\\ & \qquad + 2(\alpha_5-\alpha_6)(2\zeta-\alpha_4-\alpha_5) + 8 (\alpha_4-\alpha_5)(\zeta-\alpha_4).  
\end{align*}
Now notice that on the second and third line we have products of positive terms, so that part is positive. Moreover $\zeta-\alpha_4 \geq \zeta-\alpha_r$ for $r<4$, so 
\[
\sum_{r<4:\zeta\geq \alpha_r} (\zeta-\alpha_r)^2  \leq \sum_{r<4:\zeta\geq \alpha_r} (\zeta-\alpha_4)^2 \leq 3(\zeta-\alpha_4)^2.
\]
Thus the first line of the expression for $fob$ is also positive and vanishes if and only if all squares vanish. In particular this implies that whenever $fob$ vanishes we have $\zeta+\alpha_6= ( \alpha_4+\alpha_5+\alpha_6-\zeta) + (\alpha_4-\alpha_5) + 2(\zeta-\alpha_4) =0$. We conclude that these cases do not lead to new regions where the first order behavior vanishes.
%
%If the part under consideration is given by $\alpha_k\geq \zeta \geq \alpha_{k+1}$ for $k=0,1,2,3$ (where in the case $k=0$ we ignore the first equation), the explicit expression for $fob$ is a homogeneous quadratic polynomial in 
%$\alpha_{k+1}$, $\alpha_{k+2}$, \ldots $\alpha_6$ and $\zeta$. We can express this polynomial in terms of the positive new variables $\beta_6=\alpha_5-\alpha_6$, $\beta_5=\alpha_4-\alpha_5$, \ldots, $\beta_{k+2} = \alpha_{k+1}-\alpha_{k+2}$, 
%$\beta_{k+1} = \zeta-\alpha_{k+1}$ and $\beta_k = \alpha_4+\alpha_5+\alpha_6-\zeta$. This quadratic polynomial in $\beta_r$ only has positive coefficients and is therefore positive on the domain in question. Moreover it contains non-zero multiples of $\beta_k^2$, $\beta_{k+1}^2$, \ldots, $\beta_5^2$ (but not $\beta_6^2$). Thus we must have $\beta_i=0$ for $k\leq i\leq 5$. Now $\alpha_6+\zeta=\beta_k + 2\beta_{k+1}+2\beta_{k+2} +\cdots + 2\beta_{4}+\beta_5$. Thus $fob$ only vanishes if $\alpha_6+\zeta=0$.

It should be observed that the fact that we did not feature these final four cases in our discussion of $\alpha_6+\zeta=0$ because the condition $\alpha_6+\zeta$ implies $\alpha_4=\zeta$ in these situations, i.e. they intersect the hyperplane  $\alpha_6+\zeta=0$ only on the boundary with the cases we did consider. 
\end{proof}

In essence this lemma contains all the information about the zero-set of $fob$. We have obtained the zero-set on a very small part of the parameter space $(\alpha;\zeta) \in V_{\frac12,0}$, but we can extend the result by using the symmetries of $fob$ to the entire $V_{\frac12,0}$. However, since those symmetries are rather convoluted, I find it convenient to consider the following theorem, which shows where $fob$ vanishes on a fundamental domain of just the translation action. 
\begin{theorem}\label{thm103}
In the domain 
\[
\alpha_1\geq \alpha_2\geq \cdots \geq \alpha_6\geq \alpha_1-1, \qquad \frac12\geq \zeta \geq 0, \qquad 
\sum_{r=1}^6 \alpha_r=1
\]
the function $fob$ vanishes exactly on the polytopes given in the following table. The third column in the table gives a simplified expression for the second order behavior on each of the polytopes.
\begin{center}
\begin{tabular}{l|l|l}
Hyperplane & Extra bounding inequalities & Second order behavior \\
\hline
$\zeta=-\alpha_6$ & $\alpha_1+\alpha_2\leq 1$, $\alpha_4+\alpha_5\geq 0$ & $(u_6/z)^{-2\alpha_6 + \sum_{r<6:\alpha_r+\alpha_6<0} (\alpha_r+\alpha_6)}$  
\\
$\zeta=\alpha_3$ & $\alpha_1+\alpha_2 \geq 1$, $\alpha_4+\alpha_5\leq 0$ 
& $(u_3z)^{\alpha_1+\alpha_2-1+ \sum_{r=4,5,6} (\alpha_r+\alpha_3) 1_{\alpha_r+\alpha_3<0} 
+ \sum_{r=1,2} (1-\alpha_r-\alpha_3) 1_{\alpha_r+\alpha_3\geq 1}}$
\\
$\alpha_4+\alpha_5=0$ & $\zeta\leq \alpha_3$, $\alpha_6+\zeta\leq 0$ & 
$\left( \frac{q}{u_4u_5}\right)^{(\zeta+\alpha_5)1_{\zeta+\alpha_5\leq 0}}(u_4u_5)^{\zeta+\alpha_6 -(\alpha_3+\alpha_6) 1_{\alpha_3+\alpha_6\leq 0}}$
\\
$\alpha_1+\alpha_2=1$ & $\zeta\geq \alpha_3$, $\alpha_6 + \zeta \geq 0$
& $\left( \frac{q}{u_1u_2}\right)^{(\alpha_2-\zeta)1_{\alpha_2\leq \zeta}}
(u_1u_2)^{\alpha_3-\zeta - (\alpha_3+\alpha_6) 1_{\alpha_3+\alpha_6 \leq 0}}$
\\
\hline
$\zeta=0$ & $\alpha_6\geq 0$& $1$  \\
$\alpha_3=\alpha_4$ & $\zeta\leq \alpha_3$, 
$\alpha_4+\alpha_5\leq 0$ & $q^{\zeta-\alpha_3}$ \\
$\alpha_1=\alpha_6+1$ & $\zeta\geq -\alpha_6$, $\alpha_1+\alpha_2\leq 1$ & $q^{-(\alpha_6+\zeta)}$ \\
$\zeta=\frac12$ & $\alpha_3\geq \frac12$ & $1$ \\
$\zeta=-\sum_{r=4}^6 \alpha_r$
%\alpha_4-\alpha_5-\alpha_6$ 
& $\alpha_4+\alpha_5\leq 0$, $\alpha_1+\alpha_2\leq 1$ & $1$\\
\hline
$\zeta=0$ & $\alpha_3\leq 0$ & $1$  \\
$\zeta=\frac12$ & $\alpha_6\leq -\frac12$ & $1$ \\
$\zeta=\sum_{r=3}^5 \alpha_r$ & $\alpha_1+\alpha_2\geq 1$, $\alpha_4+\alpha_5\geq 0$ & $1$ \\
$\alpha_2=\alpha_3$ & $\zeta\geq \alpha_3$, $\alpha_1+\alpha_2\geq 1$ & $q^{\alpha_3-\zeta}$ \\
$\alpha_5=\alpha_6$ & $\alpha_4+\alpha_5\geq 0$, $\alpha_6+\zeta\leq 0$ & $q^{\alpha_6+\zeta}$
\end{tabular}
\end{center}
\end{theorem}
\begin{proof}
Essentially the proof proceeds by explicitly calculating the orbit of the two polytopes on which $fob$ vanishes obtained from Lemma \ref{lemfobvanish}. 

The given domain is easily seen to be a fundamental domain for the action of the semidirect product of $S_6\times S_2$ (acting by permutations of the $\alpha_r$ and inverting $\zeta$) with the group of translations along integer vectors which preserve the balancing condition.

First we use the reflection from the proof of Lemma \ref{lem101} to see that in the fundamental domain of $W(\tilde E_6)$ the function $fob$ vanishes if $\alpha_6+\zeta=0$, or ($\zeta=0$ and $\alpha_6+\zeta\geq 0$), or ($\alpha_5=\alpha_6$ and $\alpha_6+\zeta\leq 0$). 

Subsequently we take the $W(\tilde E_6)$ orbits of these polytopes modulo $S_6\times S_2$ and the translations and determine which are the relevant polytopes. To perform the calculations we first determine the vertices of the 2 polytopes on which $fob$ vanishes in the fundamental domain \eqref{eqfdae6p}, to wit:
\begin{align*}
\zeta=0: &(1,0^5;0), \quad  (\frac12^2,0^4;0), \quad
(\frac13^3,0^3;0),  \quad (\frac14^4,0^2;0), \quad(\frac15^5,0;0), \quad
(\frac16^6;0). \\
\alpha_6+\zeta=0: & (1,0^5;0), \quad  (\frac12^2,0^4;0), \quad
(\frac13^3,0^3;0),  \quad (\frac14^4,0^2;0), \quad(\frac15^5,0;0), \quad (\frac14^5,-\frac14;\frac14), \quad (\frac{3}{10}^4, -\frac{1}{10}^2;\frac{1}{10}).
\end{align*}
Subsequently we can calculate the result of performing the same reflection from $W(\tilde E_6)$ on all vertices concurrently to obtain a new polytope on which $fob$ vanishes. We repeat this until we find 72 (=index of the group of translations and permutations in $W(\tilde E_6)$) different polytopes in the new larger domain. (In practice this calculation was performed by a computer). We call $P$ the list of polytopes obtained in this way.

For each polytope in $P$ we determine the hyperplane it is contained in (by determining the nullspace of the matrix with the vertices as rows). Typically there are several different polytopes lying on each hyperplane we find. We now want to glue those different polytopes lying on the same hyperplane together. 

To do this, we consider for each of the hyperplanes $H$, the union $U_H$ of all vertices of polytopes which are contained in $H$. Next we determine the convex hull $CH(U_H)$ of $U_H$. A bounding inequality of $CH(U_H)$  is of the form $L(\alpha_r;\zeta) \geq 0$, for some linear map $L$, where $L(\alpha_r;\zeta)=0$ is a hyperplane which contains at least $5$ 
vertices (as it is a 4-dimensional plane) in $U_H$. Thus for each subset of 5 vertices from $U$ we determine the equation $L(\alpha_r,\zeta)=0$ on which the vertices lie, and check whether $L(\alpha_r,\zeta)\geq 0$ or $L(\alpha_r,\zeta)\leq 0$ (or neither) is valid for all vertices in $U_H$. The inequalities which pass this test determine the convex hull of $U_H$. 

Having obtained bounding inequalities of $CH(U_H)$, we can calculate its vertices. Next we  perform a direct calculation to show that the first order behavior vanishes on all of $CH(U_H)$. Indeed, we on $CH(U_H)$ can simplify the formula \eqref{eqfob} enormously by using that we often know whether $\{\alpha_1+\alpha_2\}=\alpha_1+\alpha_2$ or $\{\alpha_1+\alpha_2\}=\alpha_1+\alpha_2-1$ on $CH(U_H)$, and likewise for the other fractional parts. 
%Typically there are only very few fractional parts, the arguments of which we can not determine the floor of. 
It turns out that $fob$ vanishes on all of $CH(U_H)$ for each of the hyperplanes $H$. As a corollary it must be true that $CH(U_H)$ is the union of the polytopes from $P$ contained in $H$ (as a convex hull of those polytopes, it clearly contains each of the polytopes in $P$ which lie on $H$, but as the polytopes in $P$ give the full zero-set of $fob$, we find that $CH(U_H)$ must also be contained in the union of the polytopes in $P$.)

A similar direct calculation then gives us the second order behavior for each of these larger polytopes $CH(U_H)$, from \eqref{eqsob}. 
%For each element in the $W(\tilde E_6)$ orbit of the polytopes we determined the hyperplane on which it was contained. In most cases several different polytopes lie on the same hyperplane and we want to paste these different polytopes together to one big polytope. If the union of these polytopes on the same hyperplanes truly does form a convex polytope, the bounding inequalities of the big polytope are a subset of those of its smaller constituent polytopes. Thus we proceed by calculating which of the bounding inequalities of the smaller polytopes are valid on all of the small polytopes. Subsequently we directly verify that on the polytope determined by the inequalities we just found it is indeed true that $fob$ vanishes (this last calculation is easily performed by hand, though somewhat tedious, once one knows the vertices of the relevant polytope). 
% See temp200812.nb for the calculations.
%
%While we do not explicitly prove that the union of the small polytopes forms the big polytope, it still follows from the above argument. Indeed, all inequalities defining the big polytope are valid on each small polytope, so each of those is a subset of the big polytope. Moreover, we know that the entire set on which $fob$ vanishes on the given hyperplane is given by the union of the small polytopes, and contains the big polytope. Thus the big polytope is also a subset of the union of the small polytopes.
%
%Given the polytopes we calculated the second order behavior $sob$ directly from \eqref{eqsob}.
\end{proof}

To interpret this table, we should note that for proper limits we not only have the condition that the first order behavior must vanish, but also that the second order behavior equals 1. To ensure the second order behavior equals 1, we must sometimes impose extra conditions on $\alpha_r$ and $\zeta$ (to make the exponents vanish in the second order behaviors given above). We now want to tabulate the polytopes on which both $fob=0$ and $sob=1$. There will actually be less different polytopes in this new table, due to these extra conditions. 

For example, in the case $\alpha_3=\alpha_4$ the extra condition is $\zeta=\alpha_3$, so we should arrive in the intersection of the $\zeta=\alpha_3$ polytope and the $\alpha_3=\alpha_4$ polytope. Indeed the condition $\alpha_3=\alpha_4$ and $\alpha_4+\alpha_5\leq 0$ also imply that $\alpha_3+\alpha_6=\alpha_4+\alpha_6 \leq \alpha_4+\alpha_5\leq 0$, so $\alpha_1+\alpha_2=1-\alpha_3-\alpha_4-\alpha_5-\alpha_6=1-(\alpha_3+\alpha_6)-(\alpha_4+\alpha_5) \geq 1$. 
%which should . The equations $\zeta=\alpha_3=\alpha_4$ and $\alpha_4+\alpha_5\leq 0$ imply that $\alpha_1+\alpha_2\geq 1$, so this means that we are in a special case of the $\zeta=\alpha_3$ case considered earlier. 
%In particular we can omit this case. Likewise for the polytopes coming from $\alpha_1=\alpha_6+1$, $\alpha_2=\alpha_3$ and $\alpha_5=\alpha_6$.

In the case $\alpha_4+\alpha_5=0$ we can choose to either make $\zeta+\alpha_6-(\alpha_3+\alpha_6)1_{\alpha_3+\alpha_6\leq 0}=0$ or $u_4u_5=1$. The latter condition implies that in the original beta integral we have $t_4t_5=1$, in which case we integrate over a pole. We can pick up residues in the elliptic beta integral before setting $t_5=1/t_4$, but then the elliptic beta integral reduces to $1=1$, so we do not expect interesting limits. Thus we only consider the cases in which all exponents vanish.

It turns out that the sets where both the first order behavior vanishes and the second order behavior equals 1 (with exclusion of the cases described in the previous paragraph) together imply that for each $\alpha$ there is a unique choice of $\zeta$ in the interval $[0,\frac12]$. The relation $\alpha$ to $\zeta$ is given in the following corollary.
\begin{corollary}\label{corwhenlim}
If we cut the domain
\[
\alpha_1\geq \alpha_2 \geq \cdots \alpha_6\geq \alpha_1-1, \qquad 
\sum_{r=1}^6 \alpha_r=1
\]
along the hyperplanes $\alpha_4+\alpha_5=0$, $\alpha_1+\alpha_2=1$, $\alpha_6=0$, $\alpha_6=-\frac12$, $\alpha_3=0$, $\alpha_3=\frac12$ we obtain 8 different polytopal pieces. For each $\alpha$ there is at most one  value of $\zeta\in [0,\frac12]$ for which both the first order behavior vanishes and the second order behavior becomes 1 for all values of the $u_r$ and $z$. This value of $\zeta$ is given by a linear function of the $\alpha_r$ on each of these 8 different polytopal pieces. These $8$ pieces and the corresponding values of $\zeta$ are given below.

\begin{tabular}{cccc|c|p{5cm}}
$\alpha_3$ & $\alpha_6$ & $\alpha_1+\alpha_2$ & $\alpha_4+\alpha_5$ & $\zeta$ & extra conditions \\
\hline
$0\leq \alpha_3\leq \frac12$ & $\geq 0$ \# & $\leq 1$ & $\geq 0$ & 0 & - \\
 $0\leq \alpha_3\leq \frac12$ & $\leq -\frac12$ \# & $\leq 1$ & $\geq 0$ & $\frac12$ & -\\
 $\leq 0$ \# & $-\frac12 \leq \alpha_6\leq 0$ & $\geq 1$ & $\leq 0$ & 0 & - 
\\
$\geq \frac12$ \# & $-\frac12 \leq \alpha_6\leq 0$ & $\geq 1$ & $\leq 0$ & $\frac12$ & - \\
 $0\leq \alpha_3\leq \frac12$ & $-\frac12 \leq \alpha_6\leq 0$ & $\leq 1$ \# & $\leq 0$ \# & $-\alpha_4-\alpha_5-\alpha_6$ & - \\
 $0\leq \alpha_3\leq \frac12$ & $-\frac12 \leq \alpha_6\leq 0$ & $\geq 1$ \# & $\geq 0$ \# & $\alpha_3+\alpha_4+\alpha_5$ & - \\
 $0\leq \alpha_3\leq \frac12$ & $-\frac12 \leq \alpha_6\leq 0$ & $\geq 1$ \# & $\leq 0$ \# & $\alpha_3$ &
  $\sum_{r>3} (\alpha_r+\alpha_3) 1_{\alpha_r+\alpha_3<0} +\sum_{r=1,2} (1-\alpha_r-\alpha_3) 1_{\alpha_r+\alpha_3\geq 1}= 1- \alpha_1-\alpha_2$ 
   \\
 $0\leq \alpha_3\leq \frac12$ & $-\frac12 \leq \alpha_6\leq 0$ & $\leq 1$ \# & $\geq 0$ \# & $-\alpha_6$ & $2\alpha_6=\sum_{r<6} (\alpha_r+\alpha_6)1_{\alpha_r+\alpha_6<0}$ \\
\end{tabular}

The inequalities denoted with a \# are bounding inequalities, while the other inequalities are corollaries of these one or two bounding inequalities. For example the inequalities $0\leq \alpha_3\leq \frac12$, $\alpha_1+\alpha_2\leq1$ and $\alpha_4+\alpha_5\geq 0$ follow from $\alpha_6\geq 0$.
\end{corollary}
\begin{proof}
We first show that cutting the domain along the 6 given hyperplanes indeed gives the eight pieces from the table. This is a somewhat long calculation which can be performed by computer, but is easy enough to do by hand. For example, if $\alpha_6 \geq 0$, then it follows directly that $\alpha_4+\alpha_5\geq 2\alpha_6 \geq 0$, $\alpha_1+\alpha_2 = 1-\alpha_3-\alpha_4-\alpha_5-\alpha_6 \leq 1-4\alpha_6 \leq 1$, $\alpha_3\geq \alpha_6\geq 0$ and $\alpha_3 \leq \frac12(\alpha_1+\alpha_2) \leq \frac12$. 

Observe that there is a direct correspondence of the $8$ pieces with the polytopes from Theorem \ref{thm103} which lie on a hyperplane of the form $\zeta=\cdots$. The extra conditions ensure moreover that the second order behavior is identically 1. 

Finally we need to check that the remaining polytopes from Theorem \ref{thm103} do not give even more cases with vanishing $fob$ and $sob=1$. For the four polytopes with $sob=q^{\cdots}$ we immediately see that to ensure $sob=1$ we must either be on the hyperplane $\zeta=\alpha_3$ or on $\zeta=-\alpha_6$, both of which were already included in the table. 

For the case $\alpha_1+\alpha_2$ we see that the exponent of $u_1u_2$ vanishes if $\alpha_3+\alpha_6\geq 0$ and $\zeta=\alpha_3$, or if $\alpha_3+\alpha_6\leq 0$ and $\zeta=-\alpha_6$. In both cases we are actually in one of the first two polytopes (of Theorem \ref{thm103}). Likewise we see that the cases where the polytope on $\alpha_4+\alpha_5=0$ has trivial second order behavior is included in the first two cases.
\end{proof}

\section{The relevant tiling of space}
The descriptions in the previous sections allow us to determine the combinations of $\alpha$ and $\zeta$ for which we can hope to obtain an appropriate limit, and, given $\alpha$ and $\zeta$, we can determine what kind of limit we may be able to take (i.e. series or integral). In this section we aim to accomplish two tasks. 
\begin{itemize}
\item First of all we want to give a simpler description of the possible $\alpha$, $\zeta$ combinations which can give limits, in particular we do not want to restrict $\alpha$ (such as in Corollary \ref{corwhenlim}). Moreover, due to our restriction to the  domain of Corollary \ref{corwhenlim} we have introduced some bounding hyperplanes which do not have a meaning for the limits, so we do not want to see those return. For example we have symmetric limits on the entire polytope $\alpha_r\geq 0$ and $\zeta=0$, not just on its intersection of this polytope with $\alpha_1\geq \alpha_2 \geq \cdots \geq \alpha_6$, so the hyperplane $\alpha_1=\alpha_2$ is meaningless in this case. 
\item Secondly, we should be able to quickly determine whether we can take a limit (given $\alpha_r$ and a corresponding $\zeta$, and if so, what kind (integral? series?) of limit this will be.
\end{itemize}
 It turns out that we can obtain a beautiful tiling of space in simplices and cross-polytopes achieving these goals.

%The boundaries of the fundamental domains are not the natural boundaries of the different regions in which the kind of limit of the elliptic beta integral is the same. Consider, for example, the case $0\leq \alpha_1\leq \alpha_2\leq \cdots \leq \alpha_6$, for which $\zeta=0$ is the correct value. By symmetry it follows that $\zeta=0$ actually works for the entire polytope $0\leq \alpha_r$ ($1\leq r\leq 6$) (which is the orbit of the given polytope under $S_6$ symmetry in the $\alpha_r$, and happens to be a convex polytope itself). 
%
%In similar ways we can extend the other polytopes.
The polytopes which turn out to be relevant are the following:
\begin{definition}
We define the following simplicial polytopes by giving their $6$ vertices. In each case $\beta\in \mathbb{Z}^6 \cup (\mathbb{Z}^6 + \rho)$ denotes an integer or half-integer vector with $\beta\cdot \rho=0$, $\gamma \in \mathbb{Z}^6 \cup (\mathbb{Z}^6 + \rho)$ denotes a vector with $\gamma \cdot \rho =2$, and $(a,b,c,d,e,f)$ forms a permutation of $(1,2,3,4,5,6)$.
\begin{itemize}
\item $P_{I,\beta}$ has vertices $\beta+e_r$ (for $1\leq r\leq 6$);
\item $\hat P_{I,\gamma}$ has vertices $\gamma-e_r$ (for $1\leq r\leq 6$);
\item $P_{III,\beta,\{a,b,c\} }$ has vertices $\beta+e_a$, $\beta+e_b$, $\beta+e_c$, 
$\beta+ \rho - e_d-e_e$, $\beta+\rho-e_e-e_f$, $\beta+\rho-e_f-e_d$ (where $\{d,e,f\}\cup \{a,b,c\} = \{1,2,3,4,5,6\}$);
\item $\hat P_{III,\gamma, \{a,b,c\}}$ has vertices 
$\gamma-e_a$, $\gamma-e_b$, $\gamma-e_c$, $\gamma-\rho+e_d+e_e$, $\gamma-\rho+e_e+e_f$, $\gamma-\rho +e_f+e_d$;
\end{itemize}
and the cross polytopes 
\begin{itemize}
\item $P_{II,\beta,a}$ has vertices $\beta+e_r$ (for $r\neq a$) and $\beta+\rho-e_a-e_r$ (for $r\neq a$).
\end{itemize}
\end{definition}
These polytopes correspond to the ones defined in \cite{vdBR1} as $P_{I,0^6} = P_I^{(0)}$, 
$P_{II, 0^6,1} = P_{II}^{(0)}$ and $P_{III, 0^6,  \{4,5,6\}}= P_{III}^{(0)}$.
Also observe that \[
P_{III,\beta+\rho-e_a-e_b-e_c,\{a,b,c\}} = P_{III, \beta, \{d,e,f\}},\] so for $P_{III}$ we only need to consider $\beta \in \mathbb{Z}^6$. Likewise we only have to consider $\hat P_{III}$ for $\gamma \in \mathbb{Z}^6$.

An alternative description of a polytope is by describing the bounding inequalities, or equivalently the planes on which the facets lie. Such a description allows us to efficiently calculate whether some vector is contained in a polytope (and if so, on what face it lies). It should be noted that all polytopes lie on the plane $\sum_{r=1}^6 \alpha_r=1$.
\begin{proposition}
The bounding inequalities for the polytopes defined above are given by 
\begin{itemize}
\item For $P_{I,\beta}$: 
\[
\alpha_r\geq \beta_r, \quad (1\leq r\leq 6);
\]
\item For $\hat P_{I,\gamma}$:
\[
\alpha_r \leq \gamma_r, \quad (1\leq r\leq 6);
\]
\item For $P_{III, \beta,\{a,b,c\}}$: 
\begin{align*}
\alpha_a+\alpha_b & \leq \beta_a+\beta_b+1, & 
\alpha_a+\alpha_c & \leq \beta_a+\beta_c+1, &  
\alpha_b+\alpha_c & \leq \beta_b+\beta_c+1, \\ 
\alpha_d+\alpha_e & \leq \beta_d+\beta_e  &
\alpha_d+\alpha_f & \leq \beta_d+\beta_f  &
\alpha_e+\alpha_f & \leq \beta_e+\beta_f  ;
\end{align*}
\item For $\hat P_{III,\gamma,\{a,b,c\}}$: 
\begin{align*}
\alpha_a+\alpha_b & \geq \gamma_a+\gamma_b-1, & 
\alpha_a+\alpha_c & \geq \gamma_a+\gamma_c-1, &  
\alpha_b+\alpha_c & \geq \gamma_b+\gamma_c-1, \\ 
\alpha_d+\alpha_e & \geq \gamma_d+\gamma_e  &
\alpha_d+\alpha_f & \geq \gamma_d+\gamma_f  &
\alpha_e+\alpha_f & \geq \gamma_e+\gamma_f ;
\end{align*}
\item $P_{II,\beta, a}$:
\begin{align*}
-\frac12& \leq \alpha_a -\beta_a \leq 0, \qquad 
\alpha_a -\beta_a \leq \alpha_r-\beta_r \leq 1+\alpha_a-\beta_a \quad (r\neq a), \\ 
0 & \leq \alpha_r+\alpha_s -\beta_r-\beta_s \leq 1 , \quad (r<s, \quad  r,s\neq a).
\end{align*}
\end{itemize}
\end{proposition}

Choosing all possible values of $\beta$ and $\gamma$, these polytopes completely tile space.
%One can prove this by determining for each of the facets of the given polytopes the neighboring polytope which shares the same facet. Using the translation and permutation symmetries of the configuration reduces the number of calculations required to a manageable amount. As a result one can also better visualize the configuration. Each facet of a simplicial polytope is also a facet of a cross-polytope and vice versa.
For each of these polytopes we determine the correct value of $\zeta$ from Corollary \ref{corwhenlim}. Moreover we can determine the location where the first order behavior is maximized or minimized, in these polytopes, from Theorem \ref{thm93}. Together this gives us the following information:
\begin{theorem}
The parameter space is tiled by the simplicial polytopes $P_{I,\beta}$, $\hat P_{I,\gamma}$, with $\beta, \gamma \in \mathbb{Z}^6 \cup (\mathbb{Z}^6 + \rho)$, the simplicial polytopes $P_{III, \beta,\{a,b,c\}}$ and $\hat P_{III,\gamma,\{a,b,c\}}$ for $\beta,\gamma \in \mathbb{Z}^6$ and $\{a,b,c\}$ any subset of $\{1,2,3,4,5,6\}$, and the cross polytopes
$P_{II,\beta,a}$ for $\beta \in \mathbb{Z}^6\cup (\mathbb{Z}^6 + \rho)$ and $a\in \{1,2,3,4,5,6\}$. 

The value of $\zeta$ for which both first and second order behavior of the integrand are equal to those of the elliptic beta integral in these tiles, the ``correct value'',  is given in the following table. The table also denotes for which value of $\zeta$ the first order behavior of the integrand is maximized and minimized in the interior of the polytope (on the boundaries the first order behavior can be extremal for more values of $\zeta$). 
\begin{center}
\begin{tabular}{l|c|cc}
Polytope & Correct $\zeta$ & Integrand maximized & Integrand minimized \\
\hline
$P_{I,\beta}$, $\beta\in \mathbb{Z}^6$ & $0$ & $\frac12$ & 0 \\
$P_{I,\beta}$, $\beta\in \mathbb{Z}^6 + \rho$ & $\frac12$ & $0$ & $\frac12$ \\
$\hat P_{I,\gamma}$, $\gamma \in \mathbb{Z}^6$ & $0$ & $0$ & $\frac12$ \\
$\hat P_{I, \gamma}$, $\gamma \in \mathbb{Z}^6 + \rho$ & $\frac12$ & $\frac12$ & $0$ \\
$P_{III, \beta, \{a,b,c\}}$ & $\alpha_a+\alpha_b+\alpha_c$ & $0$ or $\frac12$ & $\alpha_a+\alpha_b+\alpha_c$ \\
$\hat P_{III, \gamma, \{a,b,c\}}$ & $\alpha_a+\alpha_b+\alpha_c$ & $\alpha_a+\alpha_b+\alpha_c$ &$0$ or $\frac12$ \\
$P_{II, \beta, a}$, $\beta\in \mathbb{Z}^6$ & $\alpha_a$ & $\frac12$ & $0$ \\
$P_{II, \beta, a}$, $\beta\in \mathbb{Z}^6+ \rho $ & $\alpha_a$ & $0$ & $\frac12$ 
\end{tabular}
\end{center}

The first order behavior is symmetric under translations $\zeta\to \zeta+1$ and under $\zeta\to -\zeta$. In the table we always only give one representative modulo these symmetries. 
%Therefore the given values of $\zeta$ are just one of the infinitely many values where the first order behavior of the integrand is maximized, respectively minimized. Outside the interior of the polytopes the first order behavior might be locally constant, which allows for even more different values of $\zeta$ at which $fob$ has an extremal point.
\end{theorem}
Before giving a proof, let us briefly describe the correspondence between the kind of limits we can take and the information given in this theorem.
\begin{itemize}
\item If the correct value of $\zeta$ corresponds to a minimum of the integrand, we need to write the elliptic beta integral with a $p$-independent contour and interchange limit and integral. If this is impossible we can not take a limit. 
\item If the correct value of $\zeta$ corresponds to a maximum of the integrand we can easily take a limit by picking up residues and taking the limit in the resulting sum of residues. This will always lead to a limit, but those limits might be very cumbersome (involving the sum of several hypergeometric series).
\item If the correct value of $\zeta$ does not correspond to a maximum or a minimum of the integrand, we have to either obtain an integral limit, or try to only pick up residues toward the minimum of the integrand. This will not always work (if the $\alpha$ contains very negative numbers we are forced to pick up residues near a maximum of the integrand as well), so these limits might fail. Regardless, we can only obtain unilateral limits in this way.
\end{itemize}
\begin{proof}
We first show that these polytopes tile space. We first observe that the 6 different types of polytopes form orbits under the translation action by vectors in $\mathbb{Z}^6\cup (\mathbb{Z}^6 + \rho)$. 

Let us now show that the intersection of $P_{I,(0^6)}$ with the interior of the other polytopes is empty. Indeed consider $P_{I,(0^6)}$ and $P_{I,\beta}$ for some $\beta\in \mathbb{Z}^6 \cup
(\mathbb{Z}^6 + \rho)$.  For the interior of the polytopes, the boundary conditions should be strict, so in the intersection we have $\alpha_r > \max(\beta_r,0)$. This implies
$\sum_r \alpha_r > \sum_{r} \max(\beta_r,0)\geq 1$ if $\beta\neq 0^6$. Thus there are no such $\alpha_r$ (as they still need to satisfy the balancing condition). Similar arguments show that all other intersections are empty. 

We will now prove that all facets of all polytopes are also facets of another polytope.As the intersections of any two polytopes are  contained in some hyperplane, this shows that the polytopes tile space. The facets of each polytope correspond to the bounding inequalities, and it is straightforward to determine for each bounding facet which other polytope has the same facet.  For $P_{I,\beta}$ the facets are also facets of $P_{II,\beta,r}$ for $1\leq r\leq 6$. For $\hat P_{I,\gamma}$ the facets are also facets of $P_{II,\gamma+e_r-\rho,r}$ for $1\leq r\leq 6$. For $P_{III,\beta, \{a,b,c\}}$ the six facets are facets of 
$P_{II,\beta-\rho + e_a+e_b+e_c,r}$ for $r=a,b,c$ and $P_{II,\beta ,r}$ for $r\neq a,b,c$. For $\hat P_{III,\gamma,\{a,b,c\}}$ the facets are also facets of $P_{II, \gamma-e_a-e_b-e_c+e_r,r}$ for $r=a,b,c$ and $P_{II,\gamma-\rho +e_r,r}$ for $r\neq a,b,c$. In particular we see that the simplices only bound cross polytopes.

A cross polytope has many more faces, and below we tally which facets $P_{II,(0^6),1}$ has, and which other polytopes also bound these faces. We use the notational convention that when we write values between braces, we imply all vertices that one can obtain by permutating the terms in the braces. Thus $(0\{0^41\})$ denotes the set of vertices $(000001)$, $(000010)$, $(000100)$, $(001000)$, and $(010000)$. 
%\[
%-\frac12 \leq \alpha_1\leq 0, \qquad 
%\alpha_1\leq \alpha_r\leq 1+\alpha_1 \quad (r>1), \qquad 
%0\leq \alpha_r+\alpha_s\leq 1 \quad (1<r<s).
%\]
%The different facets (one for each bounding inequality) also bound the following polytopes
\begin{itemize}
\item $\alpha_1=0$ has vertices $(0\{0^41\})$ and bounds $P_{I,(0^6)}$;
\item $\alpha_1=-\frac12$ has vertices $(-\frac12\{-\frac12\frac12^4\})$ and bounds $\hat P_{I, (-\frac12,\frac12^5)}$;
\item $\alpha_1=\alpha_2$ has vertices $(0^2\{0^3,1\})$ and $(-\frac12^2,\frac12^4)$ and bounds $P_{II, (0^6), 2}$. Likewise the facet $\alpha_1=\alpha_r$ bounds $P_{II, (0^6),r}$;
\item $\alpha_1+1=\alpha_6$ has vertices $(0^5,1)$ and $(-\frac12\{-\frac12,\frac12^3\}\frac12)$ and bounds $P_{II,(-1,0^4,1),6}$. Likewise $\alpha_1+1=\alpha_r$ bounds $P_{II,(-1,0^5)+e_r,r}$;
\item $0=\alpha_2+\alpha_3$ has vertices $(0^3\{0^2,1\})$ and $(-\frac12\{-\frac12,\frac12\}\frac12^3)$ and bounds $P_{III,(0^6), \{4,5,6\}}$. Likewise $0=\alpha_r+\alpha_s$ bounds $P_{III, (0^6), \{a,b,c\}}$, where $\{a,b,c,r,s\}=\{2,3,4,5,6\}$;
\item $1=\alpha_5+\alpha_6$ has vertices $(0^4\{0,1\})$ and $(-\frac12\{-\frac12,\frac12^2\}\frac12^2)$ and bounds $\hat P_{III, (0^4,1^2),\{1,5,6\}}$. Likewise $1=\alpha_r+\alpha_s$ bounds $\hat P_{III, e_r+e_s, \{1,r,s\}}$.
\end{itemize}
Thus we have accounted for all 32 facets of $P_{II,(0^6),1}$.

%A similar (but much shorter) consideration shows that the 6 facets of $P_{I,(0^6)}$ are also facets of $P_{II, (0^6), r}$ (for $1\leq r\leq 6$). The facets of $P_{III, (0^6), \{1,2,3\}}$ are also facets of $P_{III, (0^6), r}$ ($4\leq r\leq 6$), or of $P_{II, (\frac12^3,-\frac12^3),r}$ ($1\leq r\leq 3$). By symmetry we do not have to consider $\hat P_{I}$ or $\hat P_{III}$. 

Having shown that these polytopes tile space, it suffices to consider each polytope and check from Corollary \ref{corwhenlim} which $\zeta$ is the correct $\zeta$ and from Theorem \ref{thm93} for which $\zeta$ the integrand is maximized or minimized. Due to the translation symmetry we only have to check one polytope of each type. 

For $P_{I,(0^6)}$ we observe that by symmetry we may order the $\alpha_r$. If we sort them in decreasing order then we are in the case $\alpha_6\geq 0$ from Corollary \ref{corwhenlim}. Moreover we are in the case $\alpha_1+\alpha_2\leq 1$ and $\alpha_4+\alpha_5\geq 0$ from Theorem \ref{thm93}. 

In the polytope $P_{III,(0^6), \{1,2,3\}}$, we see that $\alpha_1,\alpha_2,\alpha_3 \geq \alpha_4,\alpha_5,\alpha_6\geq \alpha_1-1,\alpha_2-1,\alpha_3-1$. Moreover the polytope is $S_3\times S_3$ symmetric. Thus if we sort the first 3 and the last 3 parameters so that $\alpha_1\geq \alpha_2\geq \alpha_3$ and $\alpha_4\geq \alpha_5\geq \alpha_6$ we are actually in the domain of Corollary \ref{corwhenlim}, and can easily be seen to be in the case $\alpha_1+\alpha_2 \leq 1$ and $\alpha_4+\alpha_5\leq 0$.  Note that $-\alpha_4-\alpha_5-\alpha_6=\alpha_1+\alpha_2+\alpha_3-1$, which is up to an integer shift the value of $\zeta$ which was correct according to the table of this theorem. For Theorem \ref{thm93} we are in the case  $\alpha_1+\alpha_2 \leq 1$ and $\alpha_4+\alpha_5 \leq 0$, thus in the third row of the table. 

That the values for the other simplices are correct follows either from a similar argument, or by consideration of the symmetries of the first and second order behavior.

Finally the cross polytope $P_{II,(0^6), 6}$ is symmetric in the first 5 $\alpha_r$'s. Observe moreover that $\alpha_r\geq \alpha_6\geq \alpha_r-1$ for $r\neq 6$, so if we sort the first 5 $\alpha_r$'s then we are in the domain from Corollary \ref{corwhenlim}. We moreover find that in the polytope we have $-\frac12\leq \alpha_6\leq 0$. To check that $0\leq \alpha_3\leq \frac12$ we have to do a little work, but $\alpha_3 \geq \frac14 (\alpha_3+\alpha_4+\alpha_5+\alpha_6) = \frac14 (1-\alpha_1-\alpha_2) \geq 0$ and $\alpha_3 \leq \frac12 (\alpha_1+\alpha_2) \leq \frac12$. Considering we have $\alpha_1+\alpha_2\leq 1$ and $\alpha_4+\alpha_5\geq 0$ we are in the bottom line of the table of Corollary \ref{corwhenlim}. For Theorem \ref{thm93} we see $\alpha_1+\alpha_2\leq 1$ and $\alpha_4+\alpha_5\geq 0$, so we are in the first row of the table of that theorem.
\end{proof}

\section{Interesting limits}\label{seclimits}
In previous sections we have determined when we can take interesting limits. What is left is determining which of these limits are actually distinct from each other, thereby creating a list of resulting identities. As described earlier, many of the resulting series are identical, so we do not have to consider the infinitude of limits near maxima. We just need to find evaluations for each type of series appearing.

The different limits are ordered by level of degeneration (or dimension of the face), and subsequently we treat those limits together which only differ by a shift along an element of the root lattice of $E_6$. For each such face we then consider the different integral limits we can take. 

Typically we get the same integral limit for all shifts, but in order to be able to take the limit we might have had to break the symmetry and specialize the three new parameters differently, which theoretically allows for multiple different limits. 

As discussed before, to each face there is associated a single bilateral series (up to the choice of one parameter which determines the exact $q$-geometric sequence $p$ follows when going to zero). There are two different kind of unilateral series, one for each direction away from the correct value of $\zeta$. Symmetry sometimes makes these two different series equivalent. The two series are always related by inversion of summation order, if we specialize the parameters such that the series become terminating. One will only obtain one of these two different kinds of unilateral series in one limit, unless that limit also contains a bilateral series. 

In the different limits we have renamed some parameters in order to simplify the expression and more clearly indicate the remaining permutation symmetry between the parameters. In the limit one can typically only permute those parameters $t_r$ whose associated $\alpha_r$'s are equal. It should still remain fairly clear how the limiting identity is related to the original integral.

Any references to equation numbers in this section are to the standard work of Gasper and Rahman \cite{GR}.

\subsection{Top level}
%As for the integral limits, limits at places where the integrand is minimized, we have to consider the vectors $(0,0,0,0,0,1)$, $(-\frac12,-\frac12,\frac12,\frac12,\frac12,\frac12)$, and $(-\frac12,-\frac12,-\frac12,\frac12,\frac12,\frac32)$. As for series limits, the simplest results comes from $(-\frac12,-\frac12,\frac12,\frac12,\frac12,\frac12)$, which involves a sum of two unilateral series, and 
%$(-\frac32,\frac12,\frac12,\frac12,\frac12,\frac12)$ which involves a bilateral series plus two unilateral series. There will always be a unilateral series in any expression involving a bilateral series, so we can't get simpler equations by using different vectors. 

\subsubsection{$(0^5,1)$ and shifts}% { $(-\frac12,-\frac12,-\frac12,\frac12,\frac12,\frac32)$}
The integral for $(-\frac12^3, \frac12^2, \frac32)$ is the same as that for $(0^5,1)$, except for symmetry breaking in the latter. The symmetric result is the Nassrallah-Rahman integral evaluation (6.4.1):
\begin{equation}\label{eqNR}
\frac{(q;q)}{2} \int \frac{(z^{\pm 2}, uz^{\pm 1};q)}
{\prod_{r=1}^5 (t_rz^{\pm 1};q)} \frac{dz}{2\pi i z}
 = \frac{\prod_{r=1}^5 (\frac{u}{t_r};q)}{\prod_{1\leq r<s\leq 5} (t_rt_s;q)}, \qquad 
% \]
%under the balancing condition
%\[
u= t_1t_2t_3t_4t_5.
\end{equation}
We get an asymmetric integral from $(-\frac12,-\frac12,\frac12,\frac12,\frac12,\frac12)$, 
which, under the balancing condition
$t_1t_2u_1u_2u_3u_4=q$ gives us
%\begin{multline}
\begin{equation}
\label{eqSBtop}
(q;q) \int \frac{\prod_{r=1}^4 (\frac{qz}{u_r};q) \theta( \frac{t_1t_2w}{z}, wz;q) }
{(t_1z^{\pm 1}, t_2z^{\pm 1};q) \prod_{r=1}^4 (u_r z;q)
\theta(t_1w,t_2w;q)} (1-z^2) \frac{dz}{2\pi iz } \\= 
\frac{\prod_{1\leq r<s\leq 4} (\frac{q}{u_ru_s};q)}
{(t_1t_2;q) \prod_{r=1}^4 (t_1u_r,t_2u_r;q)}.
\end{equation}
%\end{multline}
If we specialize $w=u_4$, the resulting  integral can be obtained by symmetry-breaking the Nassrallah-Rahman integral evaluation.

The related series identity is (II.25)
%\begin{multline}
\begin{equation}
\label{eqII25}
\frac{\prod_{r=1}^4 (\frac{qt_1}{u_r}, t_2u_r;q)}{(qt_1^2, \frac{t_2}{t_1};q) }
\rWs{8}{7}(t_1^2;t_1t_2,t_1u_1,t_1u_2,t_1u_3,t_1u_4;q,q) + (t_1\leftrightarrow t_2)
\\ =
\prod_{1\leq r<s\leq 4} (\frac{q}{u_ru_s};q),
\end{equation}
%\end{multline}
with the same balancing condition.
This series identity can be obtained from the asymmetric integral by picking up the residues of the poles at $z=t_1 q^{-n}$ and $z=t_2 q^{-n}$ ($n\in \mathbb{N}$). 

Finally we have an evaluation involving a bilateral and a unilateral series (though the unilateral term is best described as the sum of two identical unilateral series), obtained from 
$(-\frac32,\frac12,\frac12,\frac12,\frac12,\frac12)$.
The identity is 
%\begin{align*}
%1 & = 
%\frac{\prod_{r=2}^6 (\frac{qt_1}{t_r};q) }{(qt_1^2;q)\prod_{2\leq r<s\leq 6} (\frac{q}{t_rt_s};q)}
%\bigg(\frac{\prod_{r=2}^6 \theta( \frac{t_rx}{t_1};q)}
%{ \theta( \frac{x}{t_1^2}, \frac{x^2}{t_1^2};q) }+
%\frac{\prod_{r=2}^6\theta(\frac{t_1t_r}{x};q)}
%{\theta( \frac{t_1^2}{x^2}, \frac{1}{x};q)  }
%\bigg)
%\\ & \qquad \qquad \times 
%\rWs{8}{7} (t_1^2;t_1t_2,t_1t_3,t_1t_4,t_1t_5,t_1t_6;q,q)
%\\ & \qquad + 
%\frac{\prod_{r=2}^6 (t_1t_r, \frac{qt_1}{t_rx}, \frac{qx}{t_1t_r};q)}
%{(q, \frac{qt_1^2}{x^2}, \frac{t_1^2}{x},x, \frac{qx^2}{t_1^2};q) \prod_{2\leq r<s\leq 6} (\frac{q}{t_rt_s};q)}
%\rpsisx{8}{8}{\frac{t_1^2}{x}, \pm \frac{qt_1}{x}, \frac{t_1t_2}{x}, \frac{t_1t_3}{x}, \frac{t_1t_4}{x}, \frac{t_1t_5}{x}, \frac{t_1t_6}{x} \\
%\frac{q}{x}, \pm \frac{t_1}{x}, \frac{qt_1}{t_2x}, \frac{qt_1}{t_3x}, \frac{qt_1}{t_4x}, \frac{qt_1}{t_5x}, \frac{qt_1}{t_6x} }{q}
%\end{align*}
%or % t_r = u_{r-1} /t_1,  t_1^2 -> t
\begin{align*}
1 & = 
\frac{\prod_{r=1}^5 (\frac{qt}{u_r};q) }{(qt;q)\prod_{1\leq r<s\leq 5} (\frac{qt}{u_ru_s};q)}
\bigg(\frac{\prod_{r=1}^5 \theta( \frac{u_r}{tx};q)}
{ \theta( \frac{1}{tx}, \frac{1}{tx^2};q) }+
\frac{\prod_{r=1}^5\theta(u_rx;q)}
{\theta( tx^2, x;q)  }
\bigg)
\rWs{8}{7} (t;u_1,u_2,u_3,u_4,u_5;q,q)
\\ & \qquad + 
\frac{\prod_{r=1}^5 (u_r, \frac{qtx}{u_r}, \frac{q}{u_rx};q)}
{(q, qtx^2, tx ,\frac{1}{x}, \frac{q}{tx^2};q) \prod_{1\leq r<s\leq 5} (\frac{qt}{u_ru_s};q)}
\rpsisx{8}{8}{tx, \pm qx\sqrt{t}, u_1x, u_2x, u_3x, u_4x, u_5x \\
qx, \pm x\sqrt{t}, \frac{qtx}{u_1}, \frac{qtx}{u_2}, \frac{qtx}{u_3}, \frac{qtx}{u_4}, \frac{qtx}{u_5} }{q},
\end{align*}
with balancing condition 
\[
u_1u_2u_3u_4u_5  = qt^2.
\]
This identity can be obtained by combining (III.38) (expressing an $\rpsis{8}{8}$ as sum of two $\rWs{8}{7}$'s) with (II.25) (the evaluation of a sum of two $\rWs{8}{7}$'s given above).

\subsection{First degeneration}
%
%Integral limits exist for vectors $(0^4, \frac12^2)$ (around $0$), 
%$(-\frac12,0^2, \frac12^3)$, $(-\frac12^2, 0, \frac12^2, 1)$, and $(-\frac12^3, \frac12, 1^2)$ (around $\frac12$), and $(-\frac14^2, \frac14^3, \frac34)$ and $(-\frac14^3, \frac14^2, \frac54)$ (around $\frac14$).
%
%Up to shift invariance there are two types of vectors $(0^4, \frac12^2)$ and $(-\frac14^2, \frac14^3, \frac34)$. Thus we need to find the series associated to both of these vectors. 
%
%Taking $(-\frac12,0^2, \frac12^3)$ (around $\frac12$) gives an evaluation for the unilateral series of the first vector. $(-1,0, \frac12^4)$ (around $\frac12$) gives an evaluation for the associated bilateral series.
%
%The vector $(-\frac14^2, \frac14^3, \frac34)$ itself gives an evaluation of a sum of two unilaterals. The vector $(-\frac34, -\frac14, \frac14^2, \frac34^2)$ gives a sum of a bilateral plus two unilaterals. $(-\frac54, \frac14^3, \frac34^2)$ gives the sum of 2 bilaterals and one unilateral.   
%
\subsubsection{$(0^4, \frac12^2)$ and shifts} 
All integrals obtained are symmetry broken versions of the symmetric integral from $(0^4, \frac12^2)$ itself. This is the Askey-Wilson integral evaluation:
\begin{equation}\label{eqAW}
\frac{(q;q)}{2} \int \frac{(z^{\pm 2};q)}{\prod_{r=1}^4 (t_rz^{\pm 1};q)} \frac{dz}{2\pi i z}
= \frac{(t_1t_2t_3t_4;q)}{\prod_{1\leq r<s\leq 4} (t_rt_s;q)}.
\end{equation}
We can get a unilateral series limit from the vector $(-\frac12,0^2, \frac12^3)$ or a single bilateral series limit from the vector $(-1,0,\frac12^4)$. This gives the identities (II.20)
\[
\rWs{6}{5}(t^2;t u_1,tu_2,tu_3;q,\frac{q}{tu_1u_2u_3})
= \frac{(qt^2, \frac{q}{u_1u_2},\frac{q}{u_1u_3},\frac{q}{u_2u_3};q)}
{(\frac{qt}{u_1},\frac{qt}{u_2},\frac{qt}{u_3},\frac{q}{tu_1u_2u_3};q)},
\]
and (II.33)
\[
\rpsisx{6}{6}{ \pm \frac{qt_1}{x}, \frac{t_1t_3}{x}, \frac{t_1t_4}{x},\frac{t_1t_5}{x},\frac{t_1t_6}{x} \\
\pm \frac{t_1}{x}, \frac{qt_1}{t_3x},\frac{qt_1}{t_4x},\frac{qt_1}{t_5x},\frac{qt_1}{t_6x}  }{t_1t_2}
= \frac{(q,\frac{qt_1^2}{x^2}, \frac{qx^2}{t_1^2};q) \prod_{3\leq r<s\leq 6} (t_rt_s;q)}
{(t_1t_2;q) \prod_{r=3}^6 (\frac{qt_1}{t_rx}, \frac{qx}{t_1t_r};q)}.
\]

\subsubsection{$(-\frac14^2, \frac14^3, \frac34)$ and shifts}
Both integral limits obtained are equivalent and equal to (in the case $(-\frac14^2, \frac14^3, \frac34)$) 
\[
(q;q) \int \frac{(\frac{qz}{t_6};q) \theta(\frac{t_1t_2w}{z},wz;q)}
{(\frac{t_1}{z},\frac{t_2}{z}, t_3z,t_4z,t_5z;q) \theta(t_1w,t_2w;q)} \frac{dz}{2\pi i z} = \frac{(\frac{q}{t_3t_6}, \frac{q}{t_4t_6}, \frac{q}{t_5t_6};q)}
{(t_1t_3,t_2t_3,t_1t_4,t_2t_4,t_1t_5,t_2t_5;q)}.
\]
The double series identity obtained directly from this expression (from vector
$(-\frac14^2,\frac14^3, \frac34)$) is (II.24):
\[
\frac{(\frac{qt_1}{t_6} ;q) \prod_{r=3}^5 (t_2t_r;q) }
{(\frac{t_2}{t_1};q) \prod_{r=3}^5 (\frac{q}{t_rt_6};q)}
\rphisx{3}{2}{t_1t_3,t_1t_4,t_1t_5 \\ \frac{qt_1}{t_2},\frac{qt_1}{t_6} }{q} + (t_1\leftrightarrow t_2) = 1.
\]
We can also get two unilaterals plus a bilateral identity from $(-\frac34,-\frac14,\frac14^2,\frac34^2)$:
\begin{align*}
1&= \frac{(\frac{qt_1}{t_3}, t_2t_3,\frac{qt_1}{t_4}, t_2t_4;q) \theta(\frac{t_5x}{t_1}, \frac{t_6x}{t_1};q)}{(\frac{q}{t_3t_5}, \frac{q}{t_4t_5},\frac{q}{t_3t_6}, \frac{q}{t_4t_6};q) \theta(\frac{x}{t_1^2}, t_5t_6x;q)}
\rphisx{3}{2}{t_1t_2,t_1t_5,t_1t_6 \\ \frac{qt_1}{t_3},\frac{qt_1}{t_4}}{q}
\\& \qquad + \frac{(\frac{qt_2}{t_5}, t_1t_5,\frac{qt_2}{t_6}, t_1t_6;q) \theta(\frac{t_1t_3}{x},\frac{t_1t_4}{x};q)}
{(\frac{q}{t_3t_5}, \frac{q}{t_4t_5}, \frac{q}{t_3t_6}, \frac{q}{t_4t_6};q) \theta(\frac{t_1}{t_2x}, t_5t_6x;q)} 
\rphisx{3}{2}{t_1t_2,t_2t_3,t_2t_4 \\ \frac{qt_2}{t_5},\frac{qt_2}{t_6}}{q} 
\\& \qquad + \frac{(t_1t_2,t_2t_3,t_2t_4,t_1t_5,t_1t_6,\frac{qt_1}{t_5x},\frac{qt_1}{t_6x},\frac{qx}{t_1t_3},\frac{qx}{t_1t_4};q)}
{(q,\frac{q}{t_3t_5},\frac{q}{t_4t_5},\frac{q}{t_3t_6},\frac{q}{t_4t_6},\frac{t_1^2}{x},\frac{t_2x}{t_1};q) \theta(t_5t_6x;q)}
\rpsisx{3}{3}{ \frac{t_1^2}{x}, \frac{t_1t_3}{x}, \frac{t_1t_4}{x} \\
\frac{qt_1}{t_2x}, \frac{qt_1}{t_5x}, \frac{qt_1}{t_6x}}{q}.
\end{align*}
A different perhaps interesting identity comes from $(-\frac54,\frac14^3, \frac34^2)$. This identity can be obtained from the previous ones by solving for $\rpsis{3}{3}$ in the above identity, replacing the bilateral series below by that expression and then combining the balanced $\rphis{3}{2}$'s to an evaluation formula:
\begin{align*}
1&= \frac{(\frac{qt_1}{t_5},\frac{qt_1}{t_6};q) } { (\frac{q}{t_2t_5},\frac{q}{t_3t_5},\frac{q}{t_4t_5},\frac{q}{t_2t_6},\frac{q}{t_3t_6},\frac{q}{t_4t_6};q)   }
\\& \qquad \times \left( 
\frac{\theta( \frac{t_2x}{t_1},\frac{t_3x}{t_1},\frac{t_4x}{t_1},\frac{t_5x^2}{t_1},\frac{t_6x^2}{t_1};q)}{
\theta(\frac{x}{t_1^2}, \frac{x^3}{t_1^2}, t_5t_6x;q)}
+ \frac{ \theta( \frac{t_1t_2}{x^2},\frac{t_1t_3}{x^2},\frac{t_1t_4}{x^2},\frac{t_1t_5}{x},\frac{t_1t_6}{x};q)}
{\theta(\frac{t_1^2}{x^3}, \frac{1}{x^2}, t_5t_6x;q)}
\right)
\rphisx{3}{2}{t_1t_2,t_1t_3,t_1t_4 \\ \frac{qt_1}{t_5},\frac{qt_1}{t_6}}{q} \\
%& \qquad + 
%\frac{(\frac{qt_1}{t_5},\frac{qt_1}{t_6};q)}{ (\frac{q}{t_2t_5},\frac{q}{t_3t_5},\frac{q}{t_4t_5},\frac{q}{t_2t_6},\frac{q}{t_3t_6},\frac{q}{t_4t_6};q)  }
%\rphisx{3}{2}{t_1t_2,t_1t_3,t_1t_4 \\ \frac{qt_1}{t_5},\frac{qt_1}{t_6}}{q} \\
& \qquad + \frac{(t_1t_2,t_1t_3,t_1t_4,\frac{qt_1}{t_2x},\frac{qt_1}{t_3x},\frac{qt_1}{t_4x}, \frac{qx}{t_1t_5}, \frac{qx}{t_1t_6};q) \theta(\frac{t_5x^2}{t_1}, \frac{t_6x^2}{t_1};q)}
{(q,\frac{q}{t_2t_5},\frac{q}{t_3t_5}, \frac{q}{t_4t_5}, \frac{q}{t_2t_6}, \frac{q}{t_3t_6}, \frac{q}{t_4t_6},\frac{t_1^2}{x};q) \theta(t_5t_6x,\frac{x^3}{t_1^2};q)} \rpsisx{3}{3}{\frac{t_1^2}{x}, \frac{t_1t_5}{x}, \frac{t_1t_6}{x} \\ \frac{qt_1}{t_2x}, \frac{qt_1}{t_3x}, \frac{qt_1}{t_4x}}{q} \\
& \qquad + \frac{(t_1t_2,t_1t_3,t_1t_4,\frac{qx^2}{t_1t_2}, \frac{qx^2}{t_1t_3}, \frac{qx^2}{t_1t_4};q) \theta( \frac{t_1t_5}{x}, \frac{t_1t_6}{x}, \frac{t_5x^2}{t_1}, \frac{t_6x^2}{t_1};q)}
{(q,\frac{q}{t_2t_5}, \frac{q}{t_3t_5}, \frac{q}{t_4t_5}, \frac{q}{t_2t_6}, \frac{q}{t_3t_6}, \frac{q}{t_4t_6}, x^2,\frac{t_5x^2}{t_1}, \frac{t_6x^2}{t_1};q) \theta( \frac{t_1^2}{x^3}, t_5t_6x;q)} \rpsisx{3}{3}{\frac{t_1t_2}{x^2}, \frac{t_1t_3}{x^2}, \frac{t_1t_4}{x^2} \\
\frac{q}{x^2}, \frac{qt_1}{t_5x^2}, \frac{qt_1}{t_6x^2} }{q}.
\end{align*}

\subsection{Second degenerations}
\subsubsection{$(0^3,\frac13^3)$ and shifts}
The integral limits are a special case of the Askey-Wilson evaluation where one of the parameters has been set to 0, or a symmetry broken version of it:
\[
\frac{(q;q)}{2} \int \frac{(z^{\pm 2};q)}{(t_1z^{\pm 1}, t_2z^{\pm 1}, t_3z^{\pm 1};q)}
\frac{dz}{2\pi i z} = \frac{1}{(t_1t_2,t_1t_3,t_2t_3;q)}.
\] 
There are no series limits, as the correct location is a strict minimum of the first order behavior.

\subsubsection{$(-\frac12,\frac16^3,\frac12^2)$ and shifts}
There are no integral limits. A single unilateral series is obtained from $(-\frac12,\frac16^3,\frac12^2)$ and a single bilateral from $(-\frac56,\frac16^2, \frac12^3)$. This gives a limit of the $\rWs{6}{5}$ evaluation formula:
\[
\rphisx{5}{5}{t_1^2, \pm qt_1, t_1t_5,t_1t_6 \\ \pm t_1, \frac{qt_1}{t_5}, \frac{qt_1}{t_6},0}{\frac{q}{t_5t_6}} =
\frac{(qt_1^2, \frac{q}{t_5t_6};q)}{(\frac{qt_1}{t_5}, \frac{qt_1}{t_6};q)},
\]
and
\[
\rpsisx{5}{6}{\pm \frac{qt_1}{x}, \frac{t_1t_4}{x}, \frac{t_1t_5}{x}, \frac{t_1t_6}{x} \\
\pm \frac{t_1}{x}, \frac{qt_1}{t_4x}, \frac{qt_1}{t_5x}, \frac{qt_1}{t_6x},0}{ \frac{qt_1}{t_4t_5t_6x}}
=\frac{(q,\frac{q}{t_4t_5}, \frac{q}{t_4t_6}, \frac{q}{t_5t_6}, \frac{qt_1^2}{x^2}, \frac{qx^2}{t_1^2};q)}
{(\frac{qt_1}{t_4x}, \frac{qt_1}{t_5x}, \frac{qt_1}{t_6x}, \frac{qx}{t_1t_4}, \frac{qx}{t_1t_5}, \frac{qx}{t_1t_6};q)}.
\]

\subsubsection{$(-\frac13,0^2, \frac13^2,\frac23)$ and shifts}
The integral limit is found from either $(-\frac16^2, \frac16^2, \frac12^2)$ or $(-\frac13^2, 0, \frac13^2, 1)$. A specialization is obtained from $(-\frac13,0^2, \frac13^2,\frac23)$. 
This integral equals
\[
(q;q) \int \frac{\theta(\frac{t_1t_2w}{z}, wz;q)}{(\frac{t_1}{z}, \frac{t_2}{z}, u_1z,u_2z;q) 
\theta(t_1w,t_2w;q)} \frac{dz}{2\pi i z} 
= 
\frac{(t_1t_2u_1u_2;q)}{(t_1u_1,t_1u_2,t_2u_1,t_2u_2;q)}.
\]
A single unilateral series is obtained from $(-\frac13,0^2,\frac13^2,\frac23)$, giving the $q$-Gauss sum (II.8):
\[
\rphisx{2}{1}{t_1t_4, t_1t_5 \\ \frac{qt_1}{t_6}}{t_2t_3} = \frac{( \frac{q}{t_4t_6}, \frac{q}{t_5t_6};q)}{(t_2t_3, \frac{qt_1}{t_6};q)}.
\]
A sum of two (different type) unilaterals is obtained from $(-\frac16^2, \frac16^2, \frac12^2)$, which is (II.23):
\[
1= \frac{(t_2t_3, t_2t_4;q)}{(\frac{t_2}{t_1}, t_1t_2t_3t_4;q)} \rphisx{2}{1}{t_1t_3, t_1t_4 \\ \frac{qt_1}{t_2}}{q} 
+ (t_1 \leftrightarrow t_2).
\]
A sum of two bilaterals from $(-1,0,\frac13^2,\frac23^2)$ (Ex.~5.10, though it has typo's in \cite{GR}):
\begin{align*}
1 &= \frac{(t_1t_2,\frac{qt_1}{t_3x},\frac{qt_1}{t_4x}, \frac{qx}{t_1t_5}, \frac{qx}{t_1t_6};q) \theta( \frac{t_5x^2}{t_1}, \frac{t_6x^2}{t_1};q)}
{(q,\frac{q}{t_3t_5}, \frac{q}{t_4t_5}, \frac{q}{t_3t_6}, \frac{q}{t_4t_6};q) \theta(t_5t_6x,\frac{x^3}{t_1^2};q)}
\rpsisx{2}{2}{\frac{t_1t_5}{x}, \frac{t_1t_6}{x} \\ \frac{qt_1}{t_3x}, \frac{qt_1}{t_4x}}{q}
\\ & \qquad + 
\frac{(t_1t_2, \frac{qt_1}{t_5x^2}, \frac{qt_1}{t_6x^2}, \frac{qx^2}{t_1t_3}, \frac{qx^2}{t_1t_4} ;q) \theta( \frac{t_1t_5}{x}, \frac{t_1t_6}{x};q)}
{(q,\frac{q}{t_3t_5}, \frac{q}{t_4t_5}, \frac{q}{t_3t_6}, \frac{q}{t_4t_6};q) \theta(\frac{t_1^2}{x^3}, t_5t_6x;q)}
\rpsisx{2}{2}{\frac{t_1t_3}{x^2}, \frac{t_1t_4}{x^2} \\ \frac{qt_1}{t_5x^2}, \frac{qt_1}{t_6x^2}}{t_1t_2}.
\end{align*}
A sum of a bilateral and a unilateral from $(-\frac12,-\frac16,\frac16^2, \frac12, \frac56)$:
\begin{align*}
1 & = \frac{(t_2t_3,t_2t_4,t_1t_5,\frac{qt_1}{t_6x}, \frac{qx}{t_1t_3}, \frac{qx}{t_1t_4};q)}{(q,\frac{q}{t_3t_6}, \frac{q}{t_4t_6}, \frac{t_2x}{t_1};q) \theta(t_5t_6x;q)} \rpsisx{2}{2}{\frac{t_1t_3}{x}, \frac{t_1t_4}{x} \\ \frac{qt_1}{t_2x}, \frac{qt_1}{t_6x}}{q} 
\\ & \qquad + 
\frac{(t_1t_5,\frac{qt_2}{t_6};q) \theta(\frac{t_1t_3}{x}, \frac{t_1t_4}{x};q)}{(\frac{q}{t_3t_6}, \frac{q}{t_4t_6};q) \theta( \frac{t_1}{t_2x}, t_5t_6x;q)} \rphisx{2}{1}{t_2t_3,t_2t_4 \\ \frac{qt_2}{t_6}}{q}.
\end{align*}

\subsection{Third degeneration}
\subsubsection{$(0^2,\frac14^4)$ and shifts}
We only obtain integral limits which are the Askey-Wilson integral evaluations with two parameters set to 0 and the symmetry broken version of it:
\[
\frac{(q;q)}{2} \int \frac{(z^{\pm 2};q)}{(t_1z^{\pm 1}, t_2z^{\pm 1};q)} \frac{dz}{2\pi i z}
= \frac{1}{(t_1t_2;q)}.
\]

\subsubsection{$(-\frac12,\frac14^4,\frac12)$ and shifts}
A single unilateral series is obtained from the original vector:
\[
\rphisx{4}{5}{t_1^2, \pm q t_1, t_1t_6 \\ \pm t_1, \frac{qt_1}{t_6}, 0, 0}{ \frac{qt_1}{t_6}}
= \frac{(qt_1^2;q)}{(\frac{qt_1}{t_6};q)}.
\]
A single bilateral series from $(-\frac34, \frac14^3, \frac12^2)$: 
\[
\rpsisx{4}{6}{\pm \frac{qt_1}{x}, \frac{t_1t_5}{x}, \frac{t_1t_6}{x} \\
\pm \frac{t_1}{x}, \frac{qt_1}{t_5x}, \frac{qt_1}{t_6x},0,0}{ \frac{qt_1^2}{t_5t_6x^2}}
= \frac{(q,\frac{q}{t_5t_6}, \frac{qt_1^2}{x^2}, \frac{qx^2}{t_1^2};q)}
{(\frac{qt_1}{t_5x}, \frac{qt_1}{t_6x}, \frac{qx}{t_1t_5}, \frac{qx}{t_1t_6};q)}.
\]

\subsubsection{$(-\frac18^2,\frac18,\frac38^3)$ and shifts}
We get integral limits in three different vectors, which are all equivalent to or special cases of
\[
(q;q) \int \frac{\theta( \frac{t_1t_2w}{z}, wz;q)}{(\frac{t_1}{z}, \frac{t_2}{z}, uz;q)
\theta(t_1w,t_2w;q)} \frac{dz}{2\pi i z} = \frac{1}{(t_1u,t_2u;q)}.
\]
We also obtain series limits. A double unilateral series from the original vector:
\[
1=\frac{(t_2t_3;q)}{(\frac{t_2}{t_1};q)} \rphisx{2}{1}{t_1t_3, 0 \\ \frac{qt_1}{t_2}}{q} 
+ \frac{(t_1t_3;q)}{(\frac{t_1}{t_2};q)} \rphisx{2}{1}{t_2t_3, 0 \\ \frac{qt_2}{t_1}}{q}.
\]
Other series are impossible, because we are at a minimum.
% A unilateral plus bilateral from $(-\frac58,-\frac18,\frac18, \frac38^2,\frac78)$.
% \begin{align*}
% 1&=\frac{(t_2t_3, \frac{qt_1}{t_6x^2}, \frac{qx^2}{t_1t_3};q) \theta(\frac{t_1t_4}{x}, \frac{t_1t_5}{x};q)}
% {(q,\frac{q}{t_3t_6},\frac{t_2x^2}{t_1};q) \theta(t_4t_6x,t_5t_6x;q)}
% \rpsisx{2}{2}{\frac{t_1t_3}{x^2} ,0 \\ \frac{qt_1}{t_2x^2}, \frac{qt_1}{t_6x^2} }{q}
% \\& \qquad  + \frac{(\frac{qt_2}{t_6};q) \theta(\frac{t_1t_3}{x^2}, \frac{t_1t_4}{x}, \frac{t_1t_5}{x};q)}
%  {(\frac{q}{t_3t_6};q) \theta(\frac{t_1}{t_2x^2}, t_4t_6x,t_5t_6x;q)} \rphisx{2}{1}{t_2t_3 ,0 \\ \frac{qt_2}{t_6}}{q}
% \end{align*}
% A sum of two bilaterals from $(-\frac58^2, \frac18, \frac38, \frac78^2)$.
% \begin{align*}
% 1 &= \frac{(\frac{qt_1}{t_5x^2}, \frac{qt_1}{t_6x^2}, \frac{qx^2}{t_1t_3};q) \theta( \frac{t_2t_3}{x^2}, \frac{t_1t_2}{x}, \frac{t_1t_4}{x}, \frac{t_2t_4}{x};q)}
% {(q,\frac{qt_1}{t_2}, \frac{t_2}{t_1},\frac{q}{t_3t_5}, \frac{q}{t_3t_6};q) \theta(t_4t_5x,t_4t_6x,t_5t_6x^3;q)}
% \rpsisx{2}{2}{\frac{t_1t_3}{x^2},0 \\ \frac{qt_1}{t_5x^2}, \frac{qt_1}{t_6x^2}}{q} + (t_1\leftrightarrow t_2)
% \end{align*}
 
% $(-\frac18^3, \frac38^2, \frac58)$ gives a sum of three unilaterals.
% 
%  A sum of two different unilaterals and a bilateral from $(-\frac78,-\frac18, \frac38^3, \frac78)$. And, taking limits to $\frac12$, a sum of 2 bilaterals from $(-\frac38,-\frac18^2, \frac38^2, \frac78)$.

\subsubsection{$(-\frac38,-\frac18,\frac18^2,\frac58^2)$ and shifts}
We only obtain series limits. A single unilateral from $(-\frac38,\frac18^3, \frac38, \frac58)$ (around $\frac38$)  (II.5):
\[
\rphisx{1}{1}{t_1t_5 \\ \frac{qt_1}{t_6}}{\frac{q}{t_5t_6}} = \frac{(\frac{q}{t_5t_6};q)}{(\frac{qt_1}{t_6};q)}.
\]
We obtain a sum of a unilateral and a bilateral from  the original vector (around $\frac18$):
\begin{align*}
1& = \frac{\theta( \frac{t_1t_3}{x}, \frac{t_1t_4}{x};q)}{\theta(\frac{t_1}{t_2x}, t_5t_6x;q)}
\rphisx{2}{1}{t_2t_3,t_2t_4 \\ 0}{q} 
\\ & \qquad 
+ \frac{(t_2t_3,t_2t_4,\frac{qx}{t_1t_3}, \frac{qx}{t_1t_4};q)}{(q,\frac{t_2x}{t_1};q) \theta(t_5t_6x;q)} 
\rpsisx{2}{2}{\frac{t_1t_3}{x}, \frac{t_1t_4}{x} \\ \frac{qt_1}{t_2x},0}{q}.
\end{align*}
A double bilateral series from $(-\frac38^2, \frac18^2, \frac58, \frac78)$
\begin{align*}
1 &= \frac{(\frac{qt_1}{t_6x}, \frac{qx}{t_1t_3}, \frac{qx}{t_1t_4};q) \theta(\frac{t_2t_3}{x}, \frac{t_2t_4}{x};q)}
{(q,\frac{q}{t_3t_6}, \frac{q}{t_4t_6};q) \theta( \frac{t_2}{t_1},t_5t_6x^2;q)}
\rpsisx{2}{2}{\frac{t_1t_3}{x}, \frac{t_1t_4}{x} \\ \frac{qt_1}{t_6x},0}{q} + (t_1\leftrightarrow t_2).
\end{align*}

\subsubsection{$(-\frac14,0^2,\frac14,\frac12^2)$ and shifts}
We obtain an integral limit only from this original vector, equal to 
\[
(q;q) \int \frac{\theta(\frac{t v}{z};q)}{(\frac{t}{z}, uz;q)} \frac{dz}{2\pi i z}
= \frac{(\frac{q}{v}, t uv;q)}{(tu;q)}.
\]
The original vector gives a single unilateral series as limit (II.3):
\[
\rphisx{1}{0}{t_1t_4 \\ - }{t_2t_3} = \frac{(t_1t_2t_3t_4;q)}{(t_2t_3;q)}.
\] 
A bilateral series can be obtained from  $(-\frac12, 0^2, \frac14,\frac12,\frac34)$ (II.29):
\[
\rpsisx{1}{1}{\frac{t_1t_4}{x} \\ \frac{qt_1}{t_6x}}{t_2t_3} = \frac{(q,\frac{q}{t_4t_6};q) \theta(t_5t_6x;q)}{(t_2t_3, t_1t_5, \frac{qt_1}{t_6x}, \frac{qx}{t_1t_4};q)}.
\]
\subsection{Fourth degeneration}

\subsubsection{$(0,\frac15^5)$ and shifts}
We only obtain an integral limit, which is the Askey-Wilson integral with 3 parameters set to 0 or a symmetry-broken version of it:
\[
\frac{(q;q)}{2} \int \frac{(z^{\pm 2};q)}{(t z^{\pm 1};q)} \frac{dz}{2\pi i z} = 1.
\]
We don't have any series limits.

\subsubsection{$(-\frac12, \frac{3}{10}^5)$ and shifts}
We only obtain series limits. A unilateral series from the original vector:
\[
\rphisx{ 3 }{5  }{t_1^2,\pm qt_1 \\ \pm t_1,0,0,0}{qt_1^2} = (qt_1^2;q).
\]
We also obtain a single bilateral from $(-\frac{7}{10}, \frac{3}{10}^4, \frac12)$:
\[
\rpsisx{3}{6}{\pm \frac{qt_1}{x}, \frac{t_1t_6}{x} \\ 
\pm \frac{t_1}{x}, \frac{qt_1}{t_6x},0,0,0 }{\frac{qt_1^3}{t_6x^3}} 
= \frac{(q,\frac{qt_1^2}{x^2}, \frac{qx^2}{t_1^2};q)}{(\frac{qt_1}{t_6x}, \frac{qx}{t_1t_6};q)}.
\]

\subsubsection{$(-\frac{1}{10}^2, \frac{3}{10}^4)$ and shifts}
This gives a symmetry-broken integral limit in the original vector, and the same limit in $(-\frac15^3, \frac25^2, \frac45)$:
\[
(q;q) \int \frac{\theta(\frac{t_1t_2w}{z}, wz;q)}{(\frac{t_1}{z},\frac{t_2}{z};q) \theta(t_1w,t_2w;q)} \frac{dz}{2\pi i z} =1.
\]
We also obtain a double unilateral series limit from the original vector:
\[
1 = 
\frac{1}{(\frac{t_2}{t_1};q)} \rphisx{2}{1}{0,0 \\ \frac{qt_1}{t_2}}{q} +(t_1 \leftrightarrow t_2).
\]
%From $(-\frac15^3,\frac25^2, \frac45)$ we can also obtain a double unilateral series limit (after symmetry-breaking and going to $\frac12$), but by killing different sequences of poles, we can also obtain a unilateral+ bilateral, or a double bilateral limit.
%
%at $\frac{1}{10}$ (min, vlak op $[0,\frac{1}{10}]$)

\subsubsection{$(-\frac25,\frac15^4, \frac35)$ and shifts}
We obtain a single unilateral limit from the original vector:
\[
\rphisx{0}{1}{- \\ \frac{qt_1}{t_6}}{\frac{qt_1}{t_6}} = \frac{1}{(\frac{qt_1}{t_6};q)}.
\]
From $(-\frac{3}{10}^2, \frac{1}{10}^2, \frac{7}{10}^2)$ we obtain a double bilateral series:
\[
1 = \frac{(\frac{qx}{t_1t_3}, \frac{qx}{t_1t_4};q) \theta( \frac{t_2t_3}{x}, \frac{t_2t_4}{x};q)}
{(q;q) \theta( \frac{t_2}{t_1}, t_5t_6x^2;q)} \rpsisx{2}{2}{\frac{t_1t_3}{x}, \frac{t_1t_4}{x} \\ 0,0}{q} + (t_1\leftrightarrow t_2).
\]
%And from $(-\frac{9}{10}, -\frac{3}{10}, \frac{1}{10}, \frac{7}{10}^3)$ we obtain the second type of unilateral series plus two bilaterals
%\begin{align*}
%1 &= \frac{(\frac{qt_1}{t_3};q) \theta( \frac{t_2t_3}{x}, \frac{t_4x^3}{t_1}, \frac{t_5x^3}{t_1}, \frac{t_6x^3}{t_1};q)}{(\theta(t_4t_5x^2,t_4t_6x^2,t_5t_6x^2,\frac{x^4}{t_1^2};q)}
%\rphisx{0}{1}{- \\ \frac{qt_1}{t_3}}{\frac{qt_1}{t_3} }
%\\ & \qquad + \frac{(\frac{qx^4}{t_1t_3};q) \theta( \frac{t_1t_2}{x}, \frac{t_2t_3}{x}, \frac{t_1t_4}{x}, \frac{t_1t_5}{x}, \frac{t_1t_6}{x};q)}
%{(q,\frac{t_1^2}{x^4};q) \theta(t_4t_5x^2,t_4t_6x^2,t_5t_6x^2,\frac{t_2x^3}{t_1};q)}
%\rpsisx{2}{2}{\frac{t_1^2}{x^4}, \frac{t_1t_3}{x^4} \\ 0,0}{q} 
%\\ & \qquad + \frac{(\frac{qx}{t_1t_2}, \frac{qx}{t_2t_3};q) 
%\theta( \frac{t_1t_3}{x^4}, \frac{t_1t_4}{x}, \frac{t_1t_5}{x}, \frac{t_1t_6}{x};q)}
%{(q;q) \theta( \frac{t_1}{t_2x^3}, t_4t_5x^2,t_4t_6x^2,t_5t_6x^2;q)}
%\rpsisx{2}{2}{\frac{t_1t_2}{x}, \frac{t_2t_3}{x} \\ 0,0}{q}
%\end{align*}
%
%at $\frac25$ (max, vlak op $[\frac25,\frac12]$)

\subsubsection{$(-\frac15,0^2, \frac25^3)$ and shifts}
We obtain an integral limit in the original vector and the identical limit in $(-\frac{1}{10}^4, \frac12^2)$:
\[
(q;q) \int \frac{\theta(\frac{tu}{z};q)}{(\frac{t}{z};q)} \frac{dz}{2\pi i z} = (\frac{q}{u};q).
\]
A unilateral series limit is obtained also from the original vector (II.1):
\[
\rphisx{1}{0}{0 \\ -}{t_2t_3} = \frac{1}{(t_2t_3;q)}.
\]
%at $\frac15$ (min, vlak op $[0,\frac15]$) 

\subsubsection{$(-\frac{3}{10}, \frac1{10}^3, \frac12^2)$ and shifts}
We obtain a unilateral series from the original vector (II.2):
\[
\rphisx{0}{0}{-\\ -}{\frac{q}{t_5t_6}} = (\frac{q}{t_5t_6};q).
\]
A single bilateral series is obtained from $(-\frac25,0^2,\frac15,\frac35^2)$: 
\[
\rpsisx{1}{1}{\frac{t_1t_4}{x} \\ 0}{t_2t_3} = \frac{(q;q) \theta( t_5t_6x;q)}{(t_2t_3,\frac{qx}{t_1t_4};q)}.
\]
%at $\frac{3}{10}$, (max, vlak op $[\frac{3}{10}, \frac12]$). 

\subsection{Fifth degeneration}
\subsubsection{$(\frac16^6)$ and shifts}
We obtain a symmetric integral limit:
\begin{equation}\label{eqlast}
\frac{(q;q)}{2} \int (z^{\pm 2};q) \frac{dz}{2\pi i z} = 1.
\end{equation}
From $(-\frac13^3, \frac23^3)$ we obtain a symmetry broken version of this integral. 

\subsubsection{$(-\frac1{12}^3, \frac{5}{12}^3)$ and shifts}
Taking the limit after symmetry breaking gives
\[
(q;q) \int \theta( \frac{t_1t_2t_3}{z};q) \frac{dz}{2\pi i z} =1.
\]
The integral \eqref{eqlast} can be rewritten to equal the sum of two instances of this integral.

\subsubsection{$(-\frac23,\frac13^5)$ and shifts}
We get a bilateral series limit:
\[
\rpsisx{2}{6}{\pm \frac{qt_1}{x} \\ \pm \frac{t_1}{x}, 0,0,0,0}{\frac{qt_1^4}{x^4}}
= (q,\frac{qt_1^2}{x^2}, \frac{qx^2}{t_1^2};q).
\]
Writing this bilateral series explicitly, we see that this identity is the sum of two versions of Jacobi's triple product formula, which we also find below.

\subsubsection{$(-\frac{5}{12}, \frac{1}{12}^3,\frac{7}{12}^2)$ and shifts}
We obtain the bilateral series limit, which is just Jacobi's triple product formula (II.28): 
\[
\rpsisx{0}{1}{- \\ 0}{\frac{q}{xt_5t_6}} = 
\sum_{k\in \mathbb{Z}} q^{\binom{k}{2}} \left(- \frac{q}{xt_5t_6} \right)^k 
= (q;q) \theta(t_5t_6x;q).
\]

\end{document}